\PassOptionsToPackage{english}{babel}
\documentclass[11pt]{article}
\usepackage{amsmath, amsfonts, amssymb, amsthm,  graphicx, mathtools, enumerate, dsfont}

\usepackage[blocks, affil-it]{authblk}

\allowdisplaybreaks

\usepackage[square]{natbib}
\usepackage[CJKbookmarks=true,
            bookmarksnumbered=true,
			bookmarksopen=true,
			colorlinks=true,
			citecolor=red,
			linkcolor=blue,
			anchorcolor=red,
			urlcolor=blue]{hyperref}
\usepackage[usenames]{color}

\usepackage[letterpaper, left=1.2truein, right=1.2truein, top = 1.2truein, bottom = 1.2truein]{geometry}
\usepackage[ruled, vlined, lined, commentsnumbered]{algorithm2e}

\usepackage{prettyref,soul}
\usepackage[ngerman]{babel}
\usepackage{bbm}
\usepackage{hyperref}
    \hypersetup{ colorlinks = true, linkcolor = blue, citecolor = blue}

\newtheorem{lemma}{Lemma}
\newtheorem{proposition}{Proposition}
\newtheorem{thm}{Theorem}

\newtheorem{corollary}{Corollary}
\newtheorem{assumption}{Assumption}
\newtheorem{remark}{Remark}

\makeatletter
\newcommand{\neutralize}[1]{\expandafter\let\csname c@#1\endcsname\count@}
\makeatother

\newenvironment{thmbis}[1]
  {\renewcommand{\thethm}{\ref*{#1}$'$}%
   \neutralize{thm}\phantomsection
   \begin{thm}}
  {\end{thm}}
  
\newenvironment{propbis}[1]
  {\renewcommand{\theproposition}{\ref*{#1}$'$}%
   \neutralize{proposition}\phantomsection
   \begin{proposition}}
  {\end{proposition}}

\newenvironment{propbisbis}[1]
  {\renewcommand{\theproposition}{\ref*{#1}$''$}%
   \neutralize{proposition}\phantomsection
   \begin{proposition}}
  {\end{proposition}}

\newrefformat{eq}{(\ref{#1})}
\newrefformat{chap}{Chapter~\ref{#1}}
\newrefformat{sec}{Section~\ref{#1}}
\newrefformat{algo}{Algorithm~\ref{#1}}
\newrefformat{fig}{Fig.~\ref{#1}}
\newrefformat{tab}{Table~\ref{#1}}
\newrefformat{rmk}{Remark~\ref{#1}}
\newrefformat{clm}{Claim~\ref{#1}}
\newrefformat{def}{Definition~\ref{#1}}
\newrefformat{cor}{Corollary~\ref{#1}}
\newrefformat{lmm}{Lemma~\ref{#1}}
\newrefformat{lemma}{Lemma~\ref{#1}}
\newrefformat{prop}{Proposition~\ref{#1}}
\newrefformat{app}{Appendix~\ref{#1}}
\newrefformat{ex}{Example~\ref{#1}}
\newrefformat{exer}{Exercise~\ref{#1}}
\newrefformat{soln}{Solution~\ref{#1}}
\newrefformat{cond}{Condition~\ref{#1}}

\def\P{{\mathbb P}}

\def\text#1{\mbox{\rm #1}}

\newcommand{\argmin}{\mathop{\rm argmin}}

\newcommand{\RR}{\mathbb{R}}

\newcommand{\sgn}{\text{sign}}

\newcommand{\supp}{{\rm supp}}

\newcommand{\reals}{\mathbb{R}}
\newcommand{\Z}{\mathbb{Z}}

\newcommand{\eigmin}{\lambda_{\text{min}}}

\newcommand{\E}{\mathbb{E}}
\renewcommand{\P}{\mathbb{P}}
\newcommand{\EE}{\mathbb{E}^{ij}}
\newcommand{\PP}{\mathbb{P}^{ij}}
\newcommand{\FF}{F^{ij}}
\newcommand{\ff}{f^{ij}}

\newcommand{\WW}{\widetilde{W}_{ij}}
\newcommand{\VV}{\widetilde{V}_{ij}}
\newcommand{\XX}{\widetilde{X}_{ij}}
\newcommand{\YY}{\widetilde{Y}_{ij}}

\newcommand{\cov}{\text{Cov}}
\newcommand{\var}{\text{Var}}

\newcommand{\fm}{f_{j_{1:m}}}

\newcommand{\diam}{\text{diam}}

\newcommand{\num}{{n\choose 2}^{-1}}

\renewcommand{\exp}{\textup{exp}}
\renewcommand{\cal}{\mathcal}
\newcommand{\define}{:=}

\newcommand{\ms}{\quad~}

\DeclarePairedDelimiter{\abs}{\lvert}{\rvert}
\DeclarePairedDelimiter{\bbrace}{\lbrace}{\rbrace}
\DeclarePairedDelimiter{\parr}{(}{)}
\DeclarePairedDelimiter{\fence}{[}{]}

\DeclarePairedDelimiter{\nm}{\|}{\|}

\let\hat\widehat
\let\bar\overline
\let\tilde\widetilde

\usepackage{xr}
\begin{document}
\selectlanguage{english}

\setlength{\abovedisplayskip}{5pt}
\setlength{\belowdisplayskip}{5pt}
\setlength{\abovedisplayshortskip}{5pt}
\setlength{\belowdisplayshortskip}{5pt}

\title{Tail behavior of dependent V-statistics and its applications}

\author[1]{Yandi Shen}
\author[2]{Fang Han}
\author[3]{Daniela Witten}
\affil[1]{
University of Washington
 \\
ydshen@uw.edu
}
\affil[2]{
University of Washington
 \\
fanghan@uw.edu
}
\affil[3]{
University of Washington
 \\
dwitten@uw.edu
}

\date{}

\maketitle

\begin{abstract}
We establish exponential inequalities and Cram\'{e}r-type moderate deviation theorems for a class of V-statistics under strong mixing conditions. Our theory is developed via kernel expansion based on random Fourier features. 
This type of expansion is new and useful for handling many notorious classes of kernels. While the developed theory has a number of applications, we apply it to lasso-type semiparametric regression estimation and high-dimensional multiple hypothesis testing.

\end{abstract}


\section{Introduction}
\label{sec:intro}


Consider the following V-statistic of order $m$ with symmetric kernel $f$,
\begin{align}
\label{eq:U}
V_n \define n^{-m}\sum_{i_1, \ldots, i_m =1}^n f\parr*{X_{i_1},\ldots, X_{i_m}},
\end{align}
where $\{X_i\}_{i=1}^n$ is a stationary sequence with marginal measure $\P$ on the $d$-dimensional real space. 
The purpose of this paper is to establish useful tail bounds and resulting Cram\'er-type moderate deviation results for \eqref{eq:U}.  Such results play pivotal roles in the analysis of many time series problems.

In \eqref{eq:U}, if the summation is taken over $m$-tuples $(i_1,\ldots, i_m)$ of distinct indices, 
the resulting is a U-statistic. 
In many applications, the techniques of analyzing U- and V-statistics are the same. The tail behavior of V- and U-statistics in the i.i.d. case has been extensively studied.  For example, \cite{hoeffding1963probability} and \cite{arcones1993limit} obtained Hoeffding- and Bernstein-type inequalities for nondegenerate and degenerate U- and V-statistics. More results are in \cite{gine2000exponential} and \cite{adamczak2006moment}. 

The analysis of V- and U-statistics when the observed data are no longer independent has attracted increasing attention in statistics and probability. However, most efforts have been put on limit theorems and bootstrap consistency. See, for instance, \cite{yoshihara1976limiting}, 
\cite{denker1982statistical}, \cite{denker1983u}, 
\cite{dehling1989empirical},  
\cite{dewan2001asymptotic},  \cite{hsing2004weighted},  \cite{dehling2006limit}, \cite{dehling2010central}, \cite{beutner2012deriving}, \cite{leucht2012degenerate}, \cite{leucht2013degenerate},  \cite{zhou2014inference}, \cite{atchade2014martingale}, among many others. 
As we shall observe from later sections, it is also important to consider the tail behavior of V- and U-statistics. 
There are few papers in this area. Exceptions include \cite{han2018exponential}  and \cite{borisov2015note}, who proved Hoeffding-type inequalities for U- and V-statistics under $\phi$-mixing conditions. There, the results are either limited to nondegenerate ones, or relying on assumptions difficult to verify. 

In this paper, we show that for a strongly mixing stationary sequence, Bernstein-type inequalities, of the form conjectured in \cite{borisov2015note} up to some logarithmic terms, hold for a large class of V- and U-statistics.  
In order to establish these results, this paper is centered around the following kernel expansion condition.  
\begin{enumerate}
\item[{\bf (A)}] For any $t > 0$, there exists a symmetric kernel $\widetilde{f}$ such that 
\begin{align*}
\abs*{f(x_1,\ldots,x_m) - \widetilde{f}(x_1,\ldots,x_m)} \leq t
\end{align*}
uniformly over a pre-specified sufficiently large set $\mathcal{C}$. Here, the choice of $\widetilde{f}$ depends on $t$ and the choice of $\cal{C}$, and is required to be of the form
\begin{align*}
\widetilde{f}(x_1,\ldots,x_m) = \sum_{j_1,\ldots,j_m=1}^K f_{j_1,\ldots,j_m}e_{j_1}(x_1)\ldots e_{j_m}(x_m),
\end{align*}
where $\bbrace*{f_{j_1,\ldots,j_m}}_{j_1,\ldots,j_m=1}^K$ is a real sequence, $\bbrace*{e_j(\cdot)}_{j=1}^K$ is a sequence of uniformly bounded real functions, and $K$ is some positive integer. 
\end{enumerate}

The above expansion is nonasymptotic in nature. Another perhaps interesting observation is that orthogonality of the bases $\bbrace*{e_j(\cdot)}_{j=1}^K$ is not required. Obviously, if Condition (A) holds, then it will simplify the analysis substantially. For example, when it is satisfied with $\mathcal{C}=\supp(\P^m)$, 
the following Bernstein-type inequality for V-statistics under strong mixing conditions is immediate: 
\begin{align*}
\P\bbrace*{\abs*{V_n - \theta}\geq (x+C_1t)} \leq 6\sum_{p=r}^m \exp\parr*{-\frac{C_2nx^{2/p}}{A_p^{1/p} + x^{1/p}M_p^{1/p}}} .
\end{align*}
Here $\theta$ is the expectation of the kernel $f$ in \eqref{eq:U} under the product measure, $t$ is the approximation error in Condition (A), $r-1$ is the degeneracy level of $f$, and $\{A_p\}_{p=1}^m,\{M_p\}_{p=1}^m$, whose definitions shall be revealed later in Proposition \ref{prop:bern_U_alpha}, are related to the variance and infinity norm of $f$. Based on this exponential inequality, a Cram\'er-type moderate deviation is immediate for nondegenerate V-statistics (and hence also for many nondegenerate U-statistics) under strong mixing conditions. 

Condition (A), although clearly related to the kernel expansion conditions made in literature (cf. \cite{leucht2012degenerate}, \cite{zhou2014inference}, and \cite{borisov2015note} for some recent developments), is very difficult to verify in practice.
For example, in the special case $m=2$, \cite{denker1982statistical} and \cite{borisov2015note}, among many others, considered using orthogonal expansion of $f$ with regard to the Hilbert space $L^2(\P^2)$ and studied such bases that are uniformly bounded. In such case, when $f$ is in $L^2(\P^2)$, the most natural choice is its spectral decomposition. However, uniform boundedness of the eigenfunctions, as mentioned in \cite{leucht2012degenerate}, is ``difficult or even impossible to check". 

As the main contribution of this paper, we provide a series of easily verifiable sufficient conditions for (A). These general conditions cover many kernels that are of interest in statistics and are summarized in Table \ref{tab:1}. 
These kernels are applicable to the analysis of many statistically important problems.

Technically speaking, there are two main difficulties in verifying (A). First, the approximation has to hold uniformly, which is more strict than the $L^2$ approximation that is sufficient for establishing weak convergence type results. Second, the expansion bases $\{e_j(\cdot)\}_{j=1}^K$ in Condition (A) are required to be uniformly bounded for leveraging the refined time series version Bernstein-type inequalities in literature. Although orthogonality is no longer required, these restrictions are still technically difficult to handle, and exclude many  classical approaches in multivariate function approximation. For example, uniform polynomial approximation by the Stone-Weierstrass theorem will have very poor performance, since high orders of the polynomials lead to a  large upper bound of the bases. In a recent paper, in order to establish weak convergence type results under mixing conditions, \cite{leucht2012degenerate} proposed to use Lipschitz-continuous scale and wavelet functions to construct $\widetilde{f}$. This approach is inappropriate in our setting, since the wavelet bases are not uniformly bounded or will render a bound too large to be useful.

Now we describe the most fundamental idea in establishing Condition (A). For sufficiently smooth kernels, we construct $\widetilde{f}$ through a probabilistic argument via the Fourier inversion formula.
The construction, used in Theorem \ref{thm:random_fourier_alpha}, is based on a set of Fourier bases with random frequencies. This idea has its origin in kernel learning \citep{rahimi2008random} and is applied to a broad class of smooth kernels with arbitrary order. 
Motivated by statistical applications, Theorems \ref{thm:Lipschitz_alpha} and \ref{thm:discontinuous_alpha} further generalize the results to less smooth kernels via constructing an intermediate kernel between $f$ and $\widetilde{f}$. This intermediate kernel is close enough to the original $f$ and also smooth enough so that we can apply Theorem $\ref{thm:random_fourier_alpha}$ to it. 

In this paper,  we study two statistical problems whose analysis is enabled by our theoretical results. We first consider estimation of a high-dimensional stochastic partially linear model, which is of great interest in statistics and econometrics. With the aid of our newly developed exponential inequalities, we show that the penalized Honor\'{e}-Powell estimator proposed in \cite{han2017adaptive} is able to recover the optimal $s\log p/n$ consistency rate under some explicitly stated mixing conditions on the data-generating scheme. It would not have been possible to establish this result using the Hoeffding-type results in \cite{borisov2015note} or \cite{han2018exponential}. Second, we consider a simultaneous independence test of bivariate time series, where, thanks to the developed Cram\'er-type moderate deviation, we allow the number of pairs to grow at an arbitrarily fast polynomial rate while still maintaining an asymptotically valid size.

Throughout the paper, we will use the following notation. For a given measure $\mathbb{Q}$, $\supp(\mathbb{Q})$ denotes its support. $\|\cdot\|_{p}$ stands for the $L^p$ norm of a real vector, with $\|\cdot\|$ as a shorthand for the $L^2$ norm. Let $\Z$ denote the set of all integers, $\RR^d$ denote the $d$-dimensional real space, and $C^k(\RR^d)$ denote the class of functions on $\RR^d$ that are continuously differentiable up to order $k$. For any positive integer $n$, $[n]$ stands for the set $\{1,\ldots,n\}$. $\nm*{f}_{L^p(\RR^d)}$ represents the $L^p$ norm of a function $f$ on $\RR^d$, that is, $\nm*{f}_{L^p(\RR^d)} \define \parr*{\int_{\RR^d}\abs*{f(x)}^pdx}^{1/p}$ with $dx = dx_1\ldots dx_d$ for any $x\in\RR^d$. When there is no confusion, the domain of the function will be omitted. $0_d$ stands for the zero vector of length $d$. For two real sequences $\{a_n\}_{n\geq 1}$ and $\{b_n\}_{n\geq 1}$, $a_n = O(b_n)$ if there is some absolute constant $C$ such that $a_n\leq Cb_n$ for all $n\geq 1$, and $a_n = o(b_n)$ if $a_n/b_n\rightarrow 0$ as $n\rightarrow 0$. We write $a_n\asymp b_n$ if $a_n = O(b_n)$ and $b_n = O(a_n)$. For two real numbers $a$ and $b$, $a\vee b = \max\{a,b\}$. For two subsets $A,B$ of $\cal{X},\cal{Y}$, respectively, $A\times B$ is defined to be the Cartesian product $\Big\{(x,y):x\in\cal{A},y\in{B}\Big\}$, and similarly for multiple products.

The rest of the paper is organized as follows. Section \ref{sec:main} provides  the main results. Particularly, in Section \ref{subsec:bernstein}, we formally state the kernel expansion (Condition (A)) and the resulting theories in a simple setting. 
In Section \ref{subsec:expansion}, we present a set of assumptions under which (A) holds. Section \ref{subsec:support} extends all the obtained results to the general case.  Section \ref{sec:app} provides two statistical applications. 
Lastly, extension to $\tau$-mixing cases is discussed in Section \ref{sec:discussion}. Proof ideas are provided in the main text, with details relegated to a supplement.

\section{Main results}
\label{sec:main}

\subsection{Exponential inequalities and Cram\'er-type moderate deviation}
\label{subsec:bernstein}
Let $\{\widetilde{X}_i\}_{i=1}^n$ be $n$ i.i.d. copies of $X_1$.  
A kernel $f$ is called \emph{centered} if 
\begin{align*}
\theta \define \E\bbrace*{f\parr*{\widetilde{X}_1,\ldots,\widetilde{X}_m}} = 0.
\end{align*}
A kernel $f$ is called \emph{symmetric} if 
\[
f(x_1,\ldots,x_m)=f(x_{i_1},\ldots,x_{i_m})
\]
for any sequence $x_1,\ldots,x_m$ and any permutation $(i_1,\ldots,i_m)$ of $(1,\ldots,m)$. In this paper, unless otherwise stated, we restrict ourselves to symmetric kernels.

A symmetric kernel $f$ is called \emph{degenerate of level $r-1$} ($2 \leq r \leq m$) if 
\begin{align*}
\E\bbrace*{f\parr*{x_1,\ldots,x_{r-1},\widetilde{X}_r,\ldots, \widetilde{X}_m}} = \theta
\end{align*}
for any $(x_1^\top,\ldots, x_{r-1}^\top)^\top\in\supp(\P^{r-1})$. 
$f$ is called \emph{fully degenerate} if it is degenerate of level $m-1$. If $f$ is not degenerate of any positive level, that is, $r=1$, it is called \emph{nondegenerate}.


For any two $\sigma$-algebras $\cal{A}$ and $\cal{B}$, the \emph{strong mixing} (hereafter also called \emph{$\alpha$-mixing}) coefficient is defined as
\begin{align*}
\alpha(\cal{A},\cal{B}) \define \sup_{A\in\cal{A},B\in\cal{B}}\abs*{\P(A\cap B) - \P(A)\P(B)}.
\end{align*}
Let $\{X_i\}_{i\in\Z}$ be a stationary sequence defined on the probability space $(\Omega, \cal{A}, \P)$. This sequence is called \emph{$\alpha$-mixing} if 
\begin{align*}
\alpha(i) \define \alpha(\cal{M}_0,\cal{G}_i) \rightarrow 0 \text{ as } i\rightarrow\infty,
\end{align*}
where $\cal{M}_0 \define \sigma(X_j,j\leq 0)$ and $\cal{G}_i \define \sigma(X_j,j\geq i)$ for $i\geq 1$ are the $\sigma$-fields generated by $\{X_j,j\leq 0\}$ and $\{X_j,j\geq i\}$ respectively.


We first present the model assumption, assuming the sequence to be geometrically $\alpha$-mixing. Extension to $\tau$-mixing case will be discussed in Section \ref{sec:discussion}.
\begin{itemize}
\item[{\bf (M)}] $\{X_i\}_{i=1}^n$ in \eqref{eq:U} is assumed to be part of a stationary sequence $\{X_i\}_{i\in\Z}$ in $\RR^d$, which is assumed to be geometrically $\alpha$-mixing with coefficient
\begin{align*}
\alpha(i) \leq \gamma_1\exp(-\gamma_2 i)~~~\text{for all }i\geq 1,
\end{align*}
where $\gamma_1,\gamma_2$ are two positive absolute constants.  
\end{itemize}

We then state the key condition on $f$. 
The kernel $f$ is said to satisfy Condition (A) with a specified symmetric set $\cal{C}\subset\reals^{md}$ if the following holds. 







\begin{enumerate}
\item[{\bf (A)}] For any $t > 0$, there exists a symmetric kernel $\widetilde{f}=\widetilde{f}(;t,\cal{C})$ such that 
\begin{align}\label{eq:approximation}
\abs*{f(x_1,\ldots,x_m) - \widetilde{f}(x_1,\ldots,x_m)} \leq t
\end{align}
uniformly over $\cal{C}$, and $\widetilde{f}$ admits the following expansion:
\begin{align}
\label{eq:expansion}
\widetilde{f}(x_1,\ldots,x_m)  = \sum_{j_1,\ldots,j_m=1}^K f_{j_1,\ldots,j_m}e_{j_1}(x_1)\ldots e_{j_m}(x_m).
\end{align}
Here $\bbrace*{f_{j_1,\ldots,j_m}}_{j_1,\ldots,j_m=1}^K$ is a real-valued sequence, $\bbrace*{e_j(\cdot)}_{j=1}^K$ is a sequence of real-valued functions on $\RR^d$, and $K$ is some positive integer, all of which could depend on $t$ and the choice of $\cal{C}$. Moreover, suppose that in the expansion \eqref{eq:expansion}, 
there exist positive constants 
\begin{align*}
F= F(t,\cal{C}), \quad B= B(t, \cal{C}),\quad \mu_a= \mu_a(t,\cal{C}) 
\end{align*}
such that 
\begin{align*}
\sum_{j_1,\ldots,j_m=1}^K \abs*{f_{j_1,\ldots,j_m}} \leq F, \quad \sup_{j\in[K],x\in\cal{C}_1}\abs*{e_j(x)} \leq B, \quad \sup_{j\in[K]}\fence*{\E\bbrace*{\abs*{e_j(X_1)}^{a}}}^{1/a} \leq \mu_a
\end{align*}
for all $1\leq a\leq 2+\delta$, with some absolute constant $\delta>0$. Here $\cal{C}_1\subset\RR^d$ stands for the projection of $\cal{C}$ onto the first argument (detailed definition in \eqref{eq:section}).
\end{enumerate}

In Condition (A), the approximation in \eqref{eq:approximation} is required to hold uniformly over the symmetric set $\cal{C}\subset\RR^{md}$. 
Here, the word ``symmetric'' means that for any $(x_1^\top,\ldots,x_m^\top)^\top \in\cal{C}$, $(x_{i_1}^\top,\ldots,x_{i_m}^\top)^\top\in\cal{C}$ for any permutation $(i_1,\ldots,i_m)$ of $[m]$. In the following, when $\{X_i\}_{i=1}^n$ are compactly supported, a natural choice of $\cal{C}$ is $\supp(\P^m)$. More generally, an ideal choice of $\cal{C}$ is some sufficiently large compact set such as $[-M,M]^{md}$, where $M = M(n)$ depends on $n$ and will increase to $\infty$ as $n\rightarrow \infty$, so that the approximation in \eqref{eq:approximation} remains in a compact set and the unwieldy part outside $\cal{C}$ is negligible as $n\to\infty$.  

Following Condition (A), we will use the notation $\{A_p\}_{p=1}^m$ and $\{M_p\}_{p=1}^m$ defined as follows, 
\begin{align}
\label{eq:bern_seq}
A_p \define \mu_1^{2(m-p)}F^2\bbrace*{\sigma^2+B^2(\log n)^4/n}^p \text{ and } M_p \define \mu_1^{(m-p)}FB^p(\log n)^{2p}
\end{align}
with 
\begin{align*}
\sigma^2 \define \frac{64\gamma_1^{\delta/(2+\delta)}}{1-\exp\bbrace*{-\gamma_2\delta/(2+\delta)}}\mu_{2+\delta}^2.
\end{align*}

Built on these assumptions and notation, the following is a Bernstein-type inequality for V-statistics under the geometric $\alpha$-mixing condition and the simple setting that $\cal{C}=\supp(\P^m)$. General settings with $\cal{C}$ not necessarily equal to $\supp(\P^m)$ will be discussed in Section \ref{subsec:support}. 


\begin{proposition}
\label{prop:bern_U_alpha}
Suppose $n \geq 2$ and Condition (M) holds. Assume that $f$ in \eqref{eq:U} satisfies Condition (A) with $\cal{C} = \supp(\P^m)$. Then, if $f$ is degenerate of level $r-1$, there exist positive constants $C_1 = C_1(m), C_2 = C_2(m,\gamma_1,\gamma_2)$ such that, for any $x > 0$ and $t > 0$,
\begin{align*}
\P\bbrace*{\abs*{V_n - \theta}\geq (x+C_1t)} \leq 6\sum_{p=r}^m \exp\parr*{-\frac{C_2nx^{2/p}}{A_p^{1/p} + x^{1/p}M_p^{1/p}}}.
\end{align*}
In particular, if $f$ is centered and fully degenerate, 
\begin{align*}
&\P\bbrace*{\abs*{V_n}\geq (x+C_1t)} \leq 6\exp\bbrace*{-\frac{C_2nx^{2/m}}{A_m^{1/m}+x^{1/m}M_m^{1/m}}}.
\end{align*}
\end{proposition}

It should be emphasized that, although seemingly technical, the proof of Proposition \ref{prop:bern_U_alpha} is straightforward given Condition (A). What may be far from obvious is that Condition (A) is, in fact, a useful condition that is satisfied by a wide range of kernels. We will elaborate on this point in Section  \ref{subsec:expansion}. In addition, since each U-statistic can be written as a linear combination of V-statistics of different orders, the analysis of U-statistics essentially reduces to that of V-statistics. We omit the details for general U-statistics, but will study several examples in the following sections. 

We now proceed to develop a Cram\'er-type moderate deviation for nondegenerate V-statistics. A statistic $T_n$ is said to be Cram\'er-type moderately deviated with range $c_n$ if
\begin{align}
\label{eq:MDP}
\frac{\P\parr*{T_n \geq x}}{1-\Phi(x)} = 1 + o(1)
\end{align}
holds uniformly for $x\in[0, c_n]$, where $\Phi(\cdot)$ is the standard normal distribution function. 
Property \eqref{eq:MDP} for properly normalized sample mean statistics in the i.i.d. case has been extensively studied (cf. Chapter 8 in \cite{petrov1975sums}). 
\cite{babu1978probabilities} and \cite{ghosh1977probabilities}, among others, studied sequences under $\alpha$- and $\phi$-mixing conditions, respectively, and proved \eqref{eq:MDP} accordingly. More recently, in a seminal paper \cite{chen2016self}, \eqref{eq:MDP} is proved for the self-normalized sample mean statistic under either the geometric $\beta$-mixing condition or the geometric moment contraction condition \citep{wu2004limit}. 

For nondegenerate V- and U-statistics in the i.i.d. case, \cite{malevich1979large} and \cite{vandemaele1983large} proved \eqref{eq:MDP} under different ranges of $c_n$. 
As a direct consequence of Proposition \ref{prop:bern_U_alpha}, we obtain for  the first time the following Cram\'er-type moderate deviation for nondegenerate V-statistics under strong mixing conditions. Write
\begin{align}
\label{eq:MDP_notation}
f_1(x) &\define \E\bbrace*{f\parr*{x, \widetilde{X}_2,\ldots, \widetilde{X}_m}}~~{\rm and}~~ \nu^2 \define \var\bbrace*{f_1(X_1)} + 2\sum_{i>1} \text{Cov}\bbrace*{f_1(X_1),f_1(X_i)}.
\end{align}
$\nu^2$ is guaranteed to be finite under the model assumption (M) and the moment condition in \eqref{eq:MDP_moment} below (cf. Lemma \ref{lemma:alpha_covariance} in the supplement). In the sequel, we will omit the dependence on $\cal{C}$ in Condition (A) when it will not lead to confusion. 


\begin{proposition}
\label{prop:MDP_U}
With the same setting as in Proposition \ref{prop:bern_U_alpha}, assume further that $f$ is centered, nondegenerate, $\nu^2 > 0$, and for some $\gamma > 0$ and $\eta > 0$,
\begin{align}
\label{eq:MDP_moment}
\E\bbrace*{f^{2+\gamma^2+\eta}\parr*{\widetilde{X}_1,\ldots,\widetilde{X}_m}} < \infty.
\end{align}
Then, for all $0\leq x_n\leq \gamma\sqrt{\log n}$, it holds that
\begin{align*}
\frac{\P\bbrace*{\frac{\sqrt{n}}{m\nu} V_n\geq x_n}}{1-\Phi(x_n)} = 1 + o(1)
\end{align*}
whenever the following two conditions hold:
\begin{align}
\label{eq:MDP_condition}
\mu_1^{m-2}F\sigma^2 = o\bbrace*{n^{1/2}(\log n)^{-3}} \text{ and } \mu_1^{m-2}B^2F = o\bbrace*{n^{3/2}(\log n)^{-8}}.
\end{align}
Here $F(t),B(t),\mu_1(t),\sigma^2(t)$ are presented in Proposition \ref{prop:bern_U_alpha} with some $t = O(1/\bbrace*{\sqrt{n}(\log n)^2})$.
\end{proposition}

Proposition \ref{prop:MDP_U} is a direct consequence of Proposition \ref{prop:bern_U_alpha} combined with Theorem 1.1 in \cite{babu1978probabilities}. The moment parameter $\gamma$ in \eqref{eq:MDP_moment} only affects the range $[0,\gamma\sqrt{\log n}]$ for which \eqref{eq:MDP} holds. It is unclear if this range can be further improved with the given conditions. As a matter of fact, its extension to the power range ($n^\delta$,$0<\delta<1$) is still open in the sample mean setting. We also note that the moment condition \eqref{eq:MDP_moment} can be further relaxed to finite $(2+\gamma^2+\eta)$th moment on $f_1$, and a similar Cram\'er-type moderate deviation holds for the class of nondegenerate U-statistics.

\subsection{Kernel expansion with uniformly bounded basis}
\label{subsec:expansion}

This subsection is fully devoted to verifying Condition (A). We restrict our attention to two particular types of the set $\mathcal{C}$: 
\[
\mathcal{C}=[-M,M]^{md}
\]
for continuous kernels, and 
\[
\mathcal{C}=[-M,M]^{md}\setminus \{\text{some sufficiently small open balls surrounding jump points}\}
\]
for discontinuous kernels. 

In what follows, we present our results in order of decreasing level of smoothness. We start with Theorem \ref{thm:random_fourier_alpha}, for which relatively strong smoothness is required. This result is then extended to Theorem \ref{thm:Lipschitz_alpha}, where the kernel is only required to be uniformly continuous or Lipschitz-continuous. Theorem \ref{thm:discontinuous_alpha} further extends the result to certain discontinuous kernels. After introducing these general guidelines, the stability result in Proposition \ref{prop:stability_sum_alpha} will enable us to analyze more complex kernels grown from small building block kernels whose expansions are verifiable by Theorems \ref{thm:random_fourier_alpha}-\ref{thm:discontinuous_alpha}.

We emphasize that in what follows, special focus is put on kernels with two particular structures: one, shift-invariant symmetric kernels in the case $m=2$ with $f(x,y) = f_0(x-y)$ for some $f_0:\RR^d\rightarrow\RR$; second, kernels with a product form, that is,
\begin{align}
\label{eq:product_form}
f(x_1,\ldots,x_m) = \prod_{\ell=1}^d h_\ell(x_{1,\ell},\ldots,x_{m,\ell})
\end{align}
for a sequence of symmetric kernels $\{h_\ell\}_{\ell=1}^d$ defined on $\RR^m$. The first case has wide applications in both statistics and machine learning, and the second case is the most common way of building a kernel with multi-dimensional arguments.

We will now present the first result on relatively smooth kernels. Recall that for a real function $g\in L^1(\RR^d)$, its Fourier transform is defined as
\begin{align*}
\hat{g}(u) \define \int_{\RR^d} g(x)e^{-2\pi iu^\top x}dx,
\end{align*}
where $dx = dx_1\ldots dx_d$ for any $x\in\RR^d$. 

\begin{thm}[Smooth kernels]
\label{thm:random_fourier_alpha}
For any given $M>0$, let $\cal{C} = [-M,M]^{md}$. Suppose that $f$ in \eqref{eq:U} satisfies the following condition (B1) with set $\cal{C}$:
\begin{enumerate}
\item[{\bf (B1)}] there exists a symmetric function $\bar{f} = \bar{f}(;\cal{C})$ such that
\begin{align}
\label{eq:f_bar}
\bar{f}(x_1,\ldots,x_m)=f(x_1,\ldots,x_m)\text{ for all } (x_1^\top,\ldots,x_m^\top)^\top\in\cal{C},
\end{align}
$\bar{f}\in L^1(\RR^{md})$ is continuous, and its Fourier transform $\hat{\bar{f}}$ satisfies
\begin{align}
\label{eq:fourier_moment}
\mu_q^q\parr*{\hat{\bar{f}}} \define \int_{\RR^{md}}\abs*{\hat{\bar{f}}(u)}\|u\|^qdu < \infty
\end{align}
for some $q \geq 1$.
\end{enumerate}
Then, for any $t> 0$, Condition (A) is satisfied with set $\cal{C}$ and constants 
\begin{align}
\label{eq:alpha_constant}
F = 2^m\nm*{\hat{\bar{f}}}_{L^1}, \quad B = 1, \quad \mu_a = 1
\end{align}
for all $a\geq 1$.
\end{thm}

Theorem \ref{thm:random_fourier_alpha} immediately implies the following corollary for shift-invariant kernels. 
\begin{corollary}
\label{cor:random_fourier_m2}
Let $m = 2$ and the kernel $f$ be shift-invariant with $f(x,y) = f_0(x-y)$ for some $f_0:\RR^d\rightarrow \RR$. For any given $M>0$, let $\cal{C}_0= [-2M,2M]^{d}$. 
Suppose that $f_0$ satisfies Condition (B1) with set $\cal{C}_0$ and alternative kernel $\bar{f}_0$ such that $\bar{f}_0(-x) = \bar{f}_0(x)$ for all $x\in\RR^d$. Then, for any $t > 0$, Condition (A) is satisfied with set $\cal{C} = [-M,M]^{2d}$ and constants
\begin{align*}
F = 4\nm*{\hat{\bar{f}}_0}_{L^1}, \quad B = 1, \quad \mu_a = 1
\end{align*}
for all $a\geq 1$. 
\end{corollary}

\begin{remark}
\label{remark:smooth}
In both Theorem \ref{thm:random_fourier_alpha} and Corollary \ref{cor:random_fourier_m2}, due to the independence of constants $(F,B,\mu_a)$ on the approximation bias $t$, the value $t$ in Condition (A) can be chosen arbitrarily small. We avoid the choice $t = 0$ so that the approximating kernel $\widetilde{f}$ remains a finite series and thus technical conditions on series convergence are not necessary. In addition, if $\bar{f}=f$ (or $\bar{f}_0 = f_0$ in Corollary \ref{cor:random_fourier_m2}), the constants are also independent of $M$ so that $M$ is allowed to be chosen arbitrarily large. We avoid the case $M=\infty$ so that the uniform approximation in \eqref{eq:approximation} is still taken over a compact set, and again, $\widetilde{f}$ remains a finite series. 
\end{remark}


The proof of Theorem \ref{thm:random_fourier_alpha} is based on a key realization that Condition (A)  is intrinsically connected to the idea of randomized feature mapping \citep{rahimi2008random} in the kernel learning literature. More specifically, when $\bar{f}\in L^1(\RR^{md})$ is continuous and $\hat{\bar{f}}\in L^1(\RR^{md})$, the Fourier inversion formula implies that 
\begin{align*}
\bar{f}(x_1,\ldots,x_m) = \int_{\RR^{md}} \hat{\bar{f}}(u_1,\ldots,u_m)e^{2\pi i(u_1^\top x_1+\ldots+u_m^\top x_m)}du_1\ldots du_m,
\end{align*}
where the right-hand side can be seen as the expectation of a Fourier basis with random frequency, which follows the sign measure of $\hat{\bar{f}}$. Due to the boundedness of the Fourier bases, Hoeffding's inequality guarantees an exponentially fast rate for a sample mean statistic of Fourier bases 
\begin{align*}
s_K(x_1,\ldots,x_m) \define \frac{1}{K}\sum_{j=1}^K\exp\Big\{2\pi i(u_{j,1}^\top x_1+\ldots+u_{j,m}^\top x_m)\Big\}
\end{align*}
to approximate $\bar{f}$ at each fixed point $x\in\RR^{md}$. The elements $\exp\{2\pi i(u_{j,1}^\top x_1+\ldots+u_{j,m}^\top x_m)\}$ in $s_K(x_1,\ldots,x_m)$ naturally decompose to bounded basis functions of inputs $x_j$.  An entropy-type argument is then used so that the approximation holds uniformly over the compact set $[-M,M]^{md}$.

We now introduce two more corollaries of Theorem \ref{thm:random_fourier_alpha}. 
We first define the space of Schwartz functions on $\RR^d$ with notation $\cal{S}(\RR^d)$ (cf. Chapter 6 in \cite{stein2011fourier}). Given a $d$-tuple $\alpha = (\alpha_1,\ldots,\alpha_d)$ of non-negative integers, the monomial $x^{\alpha}$ is defined as
\begin{align*}
x^{\alpha} \define x_1^{\alpha_1}\ldots x_d^{\alpha_d}.
\end{align*}
Similarly, the differential operator $(\partial/\partial x)^{\alpha}$ is defined as 
\begin{align*}
\parr*{\frac{\partial}{\partial x}}^{\alpha} \define \parr*{\frac{\partial}{\partial x_1}}^{\alpha_1}\ldots\parr*{\frac{\partial}{\partial x_d}}^{\alpha_d} = \frac{\partial^{|\alpha|}}{\partial x_1^{\alpha_1}\ldots\partial x_d^{\alpha_d}},
\end{align*}
where $|\alpha| = \alpha_1 + \ldots + \alpha_d$ is the order of the multi-index $\alpha$. The Schwartz space $\cal{S}(\RR^d)$ consists of all indefinitely differentiable functions $f$ on $\RR^d$ such that
\begin{align*}
\sup_{x\in\RR^d}\abs*{x^{\alpha}\parr*{\frac{\partial}{\partial x}}^\beta f(x)} < \infty
\end{align*}
for arbitrary multi-indices $\alpha$ and $\beta$. One example of a Schwartz function on $\RR^d$ is the $d$-dimensional Gaussian function $\exp(-\pi\|x\|^2)$. More generally, any smooth function with compact support on $\RR^d$ is Schwartz.

\begin{corollary}
\label{cor:schwartz}
(a) Let $k = md + 2$. Suppose the kernel $f\in C^k(\RR^{md})\bigcap L^1(\RR^{md})$, and $\partial^{k}f/\partial x_i^k \in L^1(\RR^{md})$ for all $i\in[md]$. Then, for any given $M>0$ and $\cal{C} = [-M,M]^{md}$, $f$ satisfies Condition (B1) in Theorem \ref{thm:random_fourier_alpha} with set $\cal{C}$, $\bar{f} = f$, and $q = 1$. If $f\in\cal{S}(\RR^{md})$, then for any given $M>0$ and $\cal{C} = [-M,M]^{md}$, $f$ satisfies Condition (B1) in Theorem \ref{thm:random_fourier_alpha} with set $\cal{C}$, $\bar{f} = f$, and arbitrary $q\geq 1$.\\
(b) If $f$ takes the product form \eqref{eq:product_form} for $h_\ell:\RR^m\rightarrow\RR$, then for any given $M>0$ and $\cal{C}  = [-M,M]^{md}$, $f$ satisfies Condition (B1) in Theorem \ref{thm:random_fourier_alpha} with set $\cal{C}$, $\bar{f} = f$, and $q = 1$ whenever for all $\ell\in[d]$, $h_\ell\in C^{m+2}(\RR^m)\bigcap L^1(\RR^m)$ and $\partial^{m+2}h_\ell/\partial x_i^{m+2}\in L^1(\RR^m)$ for all $i\in[m]$.
\end{corollary}

Recall that a real function $g_0:\RR^d\rightarrow\RR$ is said to be \emph{positive definite (PD)} if for any positive integer $n$ and real vectors $\{x_i\}_{i=1}^n\in\RR^d$, the matrix $A = (a_{i,j})_{i,j=1}^n$ with $a_{i,j} = g_0(x_i-x_j)$ is positive semi-definite (PSD). 

\begin{corollary}[]
\label{cor:random_fourier_alpha}
Let $m = 2$ and $f$ be shift-invariant with $f(x,y) = f_0(x-y)$. Further suppose that for any given $M>0$ and $\cal{C}_0 = [-2M,2M]^{d}$, 
$f_0$ satisfies Condition (B1) in Theorem \ref{thm:random_fourier_alpha} with set $\cal{C}_0$, some $q\geq 1$, and $\bar{f}_0$ which is PD. Then, 
for any $t > 0$, Condition (A) is satisfied with set $\cal{C} = [-M,M]^{2d}$ and constants
\begin{align}
\label{eq:alpha_constant_m2}
F = 2\bar{f}_0(0), \quad B = 1, \quad \mu_a = 1
\end{align}
for all $a\geq 1$. Moreover, given any $M>0$, Condition (A) still holds with $\cal{C} = [-M,M]^{2d}$ and constants in \eqref{eq:alpha_constant_m2} if $f_0$ satisfies Condition (B1) with set $\cal{C}_0$, some $0 < q < 1$, and $\bar{f}_0$ which is both PD and Lipschitz continuous. 
\end{corollary}

We now list several commonly-used kernels covered by Theorem \ref{thm:random_fourier_alpha} and its three corollaries.  
\begin{enumerate}
\item The $d$-dimensional Gaussian kernel $f(x,y) = f_0(x-y) = \exp\parr*{-\|x-y\|^2/2}$ is shift-invariant with $f_0$ being both Schwartz and PD. Thus, $f$ satisfies the conditions of  Corollary \ref{cor:random_fourier_alpha}.
\item For the $d$-dimensional Cauchy kernel $f(x,y) = f_0(x-y)$ with $f_0(x) = \prod_{\ell=1}^d 2/\parr*{1+x_\ell^2}$, $f_0$ is PD and for any given $M>0$ and $\cal{C}_0=[-M,M]^{d}$, its Fourier transform $\hat{f}_0(u) = \exp\parr*{-\|u\|_1}$ satisfies Condition (B1) with set $\cal{C}_0$, $\bar{f}_0 = f_0$, and arbitrary $q \geq 1$. Therefore, it satisfies the conditions of Corollary \ref{cor:random_fourier_alpha}. 
\item The $d$-dimensional Laplacian kernel $f(x,y) = f_0(x-y) = \exp(-\|x-y\|_1)$ is shift-invariant and PD, but $f_0$ is not differentiable at the point $0_d$. The Fourier transform of $f_0$ is the Cauchy measure $\hat{f}_0(u) = \prod_{\ell=1}^d \bbrace*{2/(1+u_\ell^2)}$, which has fractional moment and thus, for any given $M>0$ and $\cal{C}_0=[-M,M]^{d}$, satisfies Condition (B1) with set $\cal{C}_0$, $\bar{f}_0 = f_0$, and any $0<q<1$. Since $f_0$ is both PD and Lipschitz, it satisfies the conditions in Corollary \ref{cor:random_fourier_alpha}.
\item The 1-dimensional ``hat" kernel: $f(x,y) = f_0(x-y)$ with $f_0(x)$ equal to $x+1$ for $-1\leq x\leq 0$, $1-x$ for $0\leq x\leq 1$ and $0$ otherwise. $f_0$ is not differentiable at points $\{-1,0,1\}$ but is PD and $1$-Lipschitz. Its Fourier transform is $\hat{f}_0(u) = \bbrace*{1-\cos\parr*{2\pi u}}/(2\pi^2u^2)$ and thus has fractional moment. Therefore, for any given $M>0$ and $\cal{C}_0 = [-M,M]$, $f_0$ satisfies Condition (B1) with set $\cal{C}_0$, $\bar{f}_0 = f_0$, and any $0 < q < 1$, and hence is also covered by Corollary \ref{cor:random_fourier_alpha}. 
\end{enumerate}
 
Corollary \ref{cor:schwartz} roughly describes the level of smoothness required in order to apply Theorem \ref{thm:random_fourier_alpha}. The bottleneck lies in \eqref{eq:fourier_moment} therein, which requires that for certain $\bar{f}$ that coincides with $f$ on some given set $\cal{C}$, the Fourier transform of $\bar{f}$ has finite first moment. All the previous examples can be analyzed via Theorem \ref{thm:random_fourier_alpha} by choosing $\bar{f} = f$, but such trivial solution does not work, for example, in the case of 1-dimensional cosine kernel $f(x,y) = f_0(x-y) = \cos(x-y)\mathbbm{1}\parr*{|x-y|<\pi/2}$, where $f_0$ is not PD and its Fourier transform $\hat{f}_0(u) = 2\cos\parr*{\pi^2 u}/(1-4\pi^2u^2)$ only has fractional moment. Moreover, we wish to relax another assumption in Theorem \ref{thm:random_fourier_alpha}, that is, $\hat{\bar{f}}\in L^1(\RR^{md})$. This assumption excludes all discontinuous kernels with jumps points. 
One important example in this exclusion is the 1-dimensional size-$\delta$ ($\delta > 0$) box kernel $f(x,y) = f_0(x-y) = \mathbbm{1}\parr*{|x-y| \leq \delta}$ defined on $\RR^2$, which jumps at points $\{(x,y):|x-y| = \delta\}$. For sufficiently large set $\cal{C}$, any $\bar{f}$ that agrees with $f$ on $\cal{C}$ still has jump points, thus its Fourier transform $\hat{\bar{f}}$ is not in $L^1(\RR^2)$ and thus $f$ is not covered by Theorem \ref{thm:random_fourier_alpha}. Other important discrete kernels include those that involve the indicator function and sign function. 

We now employ the standard smoothing technique through mollifiers to extend Theorem \ref{thm:random_fourier_alpha} to less smooth cases. The next result deals with Lipschitz kernels considered in \cite{leucht2012degenerate} as well as some uniformly continuous kernels. We will use the following two notation from integration with polar coordinates (with convention $(-1)!! = 0!! = 1$):
\begin{align}
\label{eq:polar_constant}
\Gamma_1(n) &\define 
\begin{cases}
\parr*{(n-2)!!}^{-1}(2\pi)^{\frac{n}{2}} & n \text{ is even}\\
\parr*{(n-2)!!}^{-1}2(2\pi)^{\frac{n-1}{2}} & n \text{ is odd}
\end{cases},
\quad \Gamma_2(n) &\define
\begin{cases}
\frac{(n-1)!!}{(n-2)!!}\frac{\sqrt{2\pi}}{2} \quad n \text{ is even}\\
\frac{(n-1)!!}{(n-2)!!}\frac{2}{\sqrt{2\pi}} \quad n \text{ is odd}
\end{cases}.
\end{align}

\begin{thm}[Lipschitz kernels]
\label{thm:Lipschitz_alpha}
For any given $M>0$, consider $\cal{C} = [-M,M]^{md}$.\\
(a) Suppose $f$ in \eqref{eq:U} satisfies the following Condition (B2) with set $\cal{C}$:
\begin{enumerate}
\item[{\bf (B2)}] there exists a symmetric function $\bar{f} = \bar{f}(;\cal{C})$ such that \eqref{eq:f_bar} holds, and $\bar{f} \in L^1(\RR^{md})$ is bounded and uniformly continuous. Moreover, its Fourier transform satisfies 
\begin{align*}
\abs*{\hat{\bar{f}}(u)} \leq \frac{L_F}{1+\|u\|^{md+\varepsilon}}
\end{align*}
for some $\varepsilon > 0$, where $L_F$ is some positive constant. 
\end{enumerate}
Then, for any $t>0$, Condition (A) is satisfied with set $\cal{C}$ and constants
\begin{align*}
F = (1+\varepsilon^{-1})2^mc_1L_F, \quad B = 1, \quad \mu_a = 1
\end{align*}
for any $a\geq 1$, where $c_1 = \Gamma_1(md)$.\\ 
(b) Suppose $f$ in \eqref{eq:U} satisfies the following Condition (B3) with set $\cal{C}$:
\begin{enumerate}
\item[{\bf (B3)}] there exists a symmetric function $\bar{f} = \bar{f}(;\cal{C})$ such that \eqref{eq:f_bar} holds, and $\bar{f}\in L^1(\RR^{md})$ is $L$-Lipschitz with respect to the $L^2$-norm on $\RR^{md}$. Moreover, its Fourier transform satisfies 
\begin{align*}
\abs*{\hat{\bar{f}}(u)} \leq \frac{L_F}{1+\|u\|^{md}},
\end{align*}
where $L_F$ is some positive constant. 
\end{enumerate}
Then, for any $t>0$, Condition (A) is satisfied with set $\cal{C}$ and constants
\begin{align*}
F = 2^{m+1}c_1L_F\log(2c_2L/t\vee 2), \quad B = 1, \quad \mu_a = 1
\end{align*}
for any $a\geq 1$, where $c_1 = \Gamma_1(md), c_2 = \Gamma_2(md)$.  
\end{thm}

It is worth noting that neither Condition (B2) nor (B3) implies the other. While (B2) makes weaker assumptions on the smoothness of $f$, it requires faster decay of the Fourier transform than (B3). The upper bounds in both (B2) and (B3) are conditions on tail behavior of the Fourier transform and naturally arise in Fourier analysis (cf. Chapter 8.4 in \cite{folland2013real}). 
Corollary \ref{cor:Lipschitz_sufficient} ahead gives many examples that satisfy these conditions.

As will be seen soon, Conditions (B2) and (B3) are usually milder than (B1). 
On the other hand, they require more delicate analyses. The key ingredient of the proof of Theorem \ref{thm:Lipschitz_alpha} is a smoothing argument via mollifiers. More specifically, let $K(x)$ be the standard Gaussian density function on $\RR^{md}$, and $K_h(x) \define K(x/h)/h^{md}$ with variance parameter $h$. Define $\bar{f}_h\define (\bar{f}*K_h)(x)$ as the convolution of $\bar{f}$ and $K_h$, which is a smooth approximation of $\bar{f}$ and preserves many nice properties of the Gaussian distribution. Since $\bar{f}_h$ is smooth enough and thus is guaranteed to satisfy Condition (B1) in Theorem \ref{thm:random_fourier_alpha}, it serves as the intermediate step between $\bar{f}$ (and thus the original kernel $f$) and the desired $\widetilde{f}$ in Condition (A). The variance parameter $h$ controls the trade-off between approximation error and smoothness: small $h$ leads to finer approximation of $\bar{f}$ by $\bar{f}_h$, but makes $\bar{f}_h$ less smooth and thus renders a larger constant $F$. In the special case of (B2), since the Fourier transform of $\bar{f}$ has sharp enough decay, we can upper bound the $L^1$ norm of the Fourier transform of $\bar{f}_h$ with an $h$-free constant, and thus $F$ is independent of the approximation error $t$.

Before giving examples that can be covered by Theorem \ref{thm:Lipschitz_alpha}, we first state two corollaries. First we state separately the version of Theorem \ref{thm:Lipschitz_alpha} for shift-invariant kernels in the case $m=2$.
 
\begin{corollary}
\label{cor:Lipschitz_shift_invariant}
Let $m=2$ 
and the kernel $f$ be shift-invariant with $f(x,y) = f_0(x-y)$. For any $M>0$, let $\cal{C}_0 = [-2M,2M]^{2d}$
.\\
(a) Suppose that $f_0$ satisfies Condition (B2) in Theorem \ref{thm:Lipschitz_alpha} with set $\cal{C}_0$, some $\varepsilon > 0$, and $\bar{f}_0$ that satisfies $\bar{f}_0(x)=\bar{f}_0(-x)$ for all $x\in\RR^d$. Then, for any $t>0$, Condition (A) is satisfied with set $\cal{C} = [-M,M]^{2d}$ and constants
\begin{align*}
F = (1+\varepsilon^{-1})4c_1L_F, \quad B = 1, \quad \mu_a = 1
\end{align*}
for any $a\geq 1$, where $c_1 = \Gamma_1(d)$.\\ 
(b) Suppose that $f_0$ satisfies Condition (B3) in Theorem \ref{thm:Lipschitz_alpha} with set $\cal{C}_0$ and $\bar{f}_0$ such that $\bar{f}_0(x)=\bar{f}_0(-x)$ for all $x\in\RR^d$. Then, for any $t>0$, Condition (A) is satisfied with set $\cal{C} = [-M,M]^{2d}$ and constants
\begin{align*}
F = 8c_1L_F\log(2c_2L/t\vee 2), \quad B = 1, \quad \mu_a = 1
\end{align*}
for any $a\geq 1$, where $c_1 = \Gamma_1(d), c_2 = \Gamma_2(d)$. 
\end{corollary}
For shift-invariant kernels, Corollary \ref{cor:Lipschitz_shift_invariant} further allows us to reduce to the case $m=1$ when checking the conditions in Theorem \ref{thm:Lipschitz_alpha}. Next, we establish an analogous version of Corollary \ref{cor:schwartz} for Theorem \ref{thm:Lipschitz_alpha}.
\begin{corollary}
\label{cor:Lipschitz_sufficient}
(a) Let $k = md + 1$. Suppose that $f\in C^k(\RR^{md})\bigcap L^1(\RR^{md})$, and $\partial^kf/\partial x_i^k\in L^1(\RR^{md})$ for all $i\in[md]$. If $f$ is uniformly continuous, then for any given $M>0$ and $\cal{C}=[-M,M]^{md}$, it satisfies Condition (B2) in Theorem \ref{thm:Lipschitz_alpha} with set $\cal{C}$, $\bar{f} = f$, $\varepsilon = 1$, and 
\begin{align*}
L_F = L_F(k) = \|f\|_{L^1} + (md)^{k-1}\sum_{i=1}^{md}\nm*{\frac{\partial^k}{\partial x_i^k}f}_{L^1}.
\end{align*}
Moreover, for any given $M>0$ and $\cal{C}=[-M,M]^{md}$, Condition (B3) is satisfied with set $\cal{C}$, $\bar{f} = f$, and $L_F(md)$ whenever the above conditions hold for $k = md$ and $f$ is $L$-Lipschitz.\\
(b) 
Given any $M>0$, consider $\cal{C} = [-M,M]^{md}$. Assume $f$ takes the product form in \eqref{eq:product_form} for $h_\ell:\RR^{m}\rightarrow\RR$. If $f$ is uniformly continuous, and for all $\ell\in[d]$, $h_\ell \in C^{m+1}(\RR^m)\bigcap L^1(\RR^m)$ and $\partial^{m+1}h_\ell/\partial x_i^{m+1}\in L^1(\RR^m)$ for all $i\in[m]$, then for any $t > 0$, Condition (A) is satisfied with set $\cal{C}$ and constants 
\begin{align*}
F = 2^{m+d}c_1^d\prod_{\ell=1}^d C_\ell, \quad B = 1, \quad \mu_a = 1
\end{align*}
for any $a\geq 1$, where $c_1 = \Gamma_1(m)$ and $C_\ell = \|h_\ell\|_{L^1} + m^m\sum_{i=1}^m\nm*{\partial^{m+1}h_\ell/\partial x_i^{m+1}}_{L^1}$. Moreover, if $f$ is $L$-Lipschitz and for all $\ell\in[d]$, $h_\ell \in C^m(\RR^m)\bigcap L^1(\RR^m)$ and $\partial^mh_\ell/\partial x_i^m\in L^1(\RR^m)$ for all $i\in[m]$, then for any $t>0$, Condition (A) is satisfied with set $\cal{C}$ and constants
\begin{align*}
F = 2^{m+d}c_1^d\log^d\parr*{2c_2L/t\vee 2}\prod_{\ell=1}^d C_\ell, \quad B = 1, \quad \mu_a = 1
\end{align*}
for any $a\geq 1$, where $c_1 = \Gamma_1(m), c_2 =\Gamma_2(m)$, and $C_\ell = \|h_\ell\|_{L^1} + m^{m-1}\sum_{i=1}^m\nm*{\partial^mh_\ell/\partial x_i^m}_{L^1}$.
\end{corollary}
Comparing Corollary \ref{cor:Lipschitz_sufficient} with Corollary \ref{cor:schwartz}, it is immediate that Theorem \ref{thm:Lipschitz_alpha} in general requires less smoothness on $f$ than Theorem \ref{thm:random_fourier_alpha}.  Moreover, when $f$ takes the product structure in \eqref{eq:product_form}, the exponents of $\|u\|$ in Conditions (B2) and (B3) can be further reduced to $m+\varepsilon$ ($\varepsilon > 0$) and $m$, respectively. With the help of Theorem \ref{thm:Lipschitz_alpha}, we are now able to analyze the cosine kernel, defined as $f(x,y) = f_0(x-y) = \prod_{\ell=1}^d \cos(x_\ell-y_\ell)\mathbbm{1}\parr*{\abs*{x_\ell-y_\ell}\leq \pi/2}$. For simplicity, we will focus on the 1-dimensional case. Note that even though the trigonometric identity $\cos(x-y) = \cos(x)\cos(y) + \sin(x)\sin(y)$ gives a direct expansion of $\cos(x-y)$, there is no trivial expansion of the indicator $\mathbbm{1}\parr*{|x-y|\leq \pi/2}$. Since $f$ is shift-invariant, we will analyze the function $f_0$ and verify that it satisfies the conditions in Corollary \ref{cor:Lipschitz_shift_invariant}. First note that $f_0$ is 1-Lipschitz and thus uniformly continuous. Moreover, its Fourier transform is $\abs*{\hat{f}_0(u)} = \abs*{2\cos(\pi^2 u)/(1-4\pi^2u^2)}$, thus for any given $M>0$ and $\cal{C}_0 = [-2M,2M]$, $f_0$ satisfies Condition (B2) with set $\cal{C}_0$, $\bar{f}_0 = f_0$, $\varepsilon = 1$, and $L_F = 2$. Therefore, the cosine kernel satisfies Condition (A) with set $\cal{C} = [-M,M]^2$ and an absolute constant $F(t)$ that is independent of $t$.

With a similar smoothing argument, we now establish the result for discontinuous kernels. Motivated by some most commonly-used kernels in statistics, we will focus on the case that $m = 2$ and $f$ is shift-invariant with $f(x,y) = f_0(x-y)$ for some $f_0:\RR^d \rightarrow\RR$ taking the following product form
\begin{align*}
f_0(x_1,\ldots,x_d) = \prod_{\ell=1}^d h_{0,\ell}(x_\ell).
\end{align*}
Our result, however, could be easily generalized to more complicated settings based on the stability results in Proposition \ref{prop:stability_sum_alpha} ahead.

\begin{thm}[Discontinuous kernels]
\label{thm:discontinuous_alpha}
Consider any any $M_1, M_2 > 0$. Suppose the following condition holds:
\begin{enumerate}
\item[{\bf (B4)}] Assume that for each $\ell\in[d]$, there exists a real function $\bar{h}_{0,\ell}$ that is piecewise constant with $J_\ell$ jump points: $y_{\ell,1},\ldots, y_{\ell,J_\ell}$, and $\bar{h}_{0,\ell}\in L^1(\RR)$ and its Fourier transform satisfies $\abs*{\hat{\bar{h}}_{0,\ell}(u)} \leq C_\ell/|u|$ for some $C_\ell>0$. Suppose that $\{\bar{h}_{0,\ell}\}_{\ell=1}^d$ are uniformly upper bounded in absolute value by some positive constant $\Delta$. Moreover, for  
\begin{align}
\label{eq:discontinuous_support}
\cal{C} \define \bbrace*{[-M_1,M_1]^{2d}} \bigcap \Big\{(x,z)\in\RR^{2d}:\abs*{(x_\ell - z_\ell) - y_{\ell,k}}\geq M_2, k\in[J_\ell]\Big\},
\end{align}
the kernel $\bar{f}(x,y)$ defined as
\begin{align*}
\bar{f}(x,y) \define \bar{f}_0(x-y) \define \prod_{\ell=1}^d \bar{h}_{0,\ell}(x_\ell-y_\ell)
\end{align*}
is symmetric and satisfies \eqref{eq:f_bar} for set $\cal{C}$.
\end{enumerate}
Then, for any $t > 0$, Condition (A) in Proposition \ref{prop:bern_U_alpha} is satisfied with set $\cal{C}$ and constants
\begin{align*}
F = 4\prod_{\ell=1}^d\bbrace*{\int_{-1}^1\abs*{\hat{\bar{h}}_{0,\ell}(u)}du +4C_\ell\log(1/h\vee 2)}, \quad B = 1, \quad \mu_a = 1
\end{align*}
for any $a\geq 1$, where $h = M_2\log^{-1/2}(2d\Delta^d/t\vee 2)/\sqrt{2}$.
\end{thm}



Compared to Theorems \ref{thm:random_fourier_alpha} and \ref{thm:Lipschitz_alpha}, Theorem \ref{thm:discontinuous_alpha} considers a more delicate set $\cal{C}$ to guarantee the uniform approximation of $\bar{f}$ by $\bar{f}_h$ (recall the definition of $\bar{f}_h$ right after Theorem \ref{thm:Lipschitz_alpha}). The extra truncation in \eqref{eq:discontinuous_support},
\begin{align*}
\Big\{(x,z)\in\RR^{2d}:\abs*{(x_\ell - z_\ell) - y_{\ell,k}}\geq M_2, k\in[J_\ell]\Big\},
\end{align*}
is necessary since $\bar{f}_h$ does not necessarily converge to $\bar{f}$ at its discontinuity points as $h\downarrow 0$. We therefore consider the approximation error uniformly over those points that are at least $M_2$ away from each jump point, which can be controlled sharply thanks to the fast decay of Gaussian tails.


We now discuss two important discrete kernels that are covered by Theorem \ref{thm:discontinuous_alpha}. First consider the 1-dimensional box kernel of size 1: $f(x,y) = f_0(x-y) = \mathbbm{1}\parr*{|x-y|\leq 1}$. It can be readily checked that $f_0(x) = h_{0,1}(x) = \mathbbm{1}\parr*{|x|\leq 1}$ is in $L^1(\RR)$, and is piecewise constant with jump points $\{y_{1,1},y_{1,2}\} = \{-1,+1\}$ and upper bound $\Delta = 1$. Its Fourier transform is $\hat{h}_{0,1}(u) = \sin(2\pi u)/(\pi u)$, thus $f$ satisfies Condition (B4) with $\bar{h}_{0,1} = h_{0,1}$ and $C_1 = 1/\pi$. Moreover, using the relation $|\sin(x)|\leq |x|$ for all $x\in\RR$, it can be readily checked that in the definition of $F(t)$ in Theorem \ref{thm:discontinuous_alpha},
\begin{align*}
\int_{-1}^1 \abs*{\hat{\bar{h}}_{0,1}(u)}du = 2\int_0^1 \abs*{\frac{\sin(2\pi u)}{\pi u}}du \leq 4.
\end{align*}

Next, consider the kernel related to Kendall's tau, which is defined on $\RR^2\times\RR^2$ as $f(x,y) \define \sgn(x_1-y_1)\sgn(x_2-y_2)$, with $\sgn(x)$ being the sign function that takes values $1,-1$ and $0$ for $x >0, x<0$ and $x=0$. Note that $f$ is shift-invariant with $f_0(x) = \sgn(x_1)\sgn(x_2)$, thus $f_0$ is in the product form with $h_{0,1}(x) = h_{0,2}(x) = \sgn(x)$. Since $h_{0,1}\notin L^1(\RR)$, we consider its truncated version $\bar{h}_{0,1}(x)\define \sgn(x)\mathbbm{1}\parr*{|x|\leq 2M_1}$ (with $\bar{h}_{0,2}$ identically defined) and the corresponding truncated version of Kendall's tau:
\begin{align}
\label{eq:Kendall_truncate}
\bar{f}(x,y;M_1)\define \sgn(x_1-y_1)\sgn(x_2-y_2)\mathbbm{1}\parr*{\abs*{x_1-y_1}\leq 2M_1}\mathbbm{1}\parr*{\abs*{x_2-y_2}\leq 2M_1}.
\end{align}
Indeed, $\bar{f}$ is symmetric and satisfies $f(x,y) = \bar{f}(x,y)$ uniformly over $\cal{C}$ defined in \eqref{eq:discontinuous_support} with $d = 2$. Clearly, $\bar{h}_{0,1}$ is piecewise constant with three jump points $\{y_{1,1},y_{1,2},y_{1,3}\} = \{-2M_1,0,2M_1\}$ and upper bound $\Delta = 1$. 
The Fourier transform of $\bar{h}_{0,1}$ is $\hat{\bar{h}}_{0,1}(u) = i\bbrace*{\cos(4\pi uM_1)-1}/(\pi u)$, and thus $f$ satisfies Condition (B4) with $\bar{h}_{0,1}, \bar{h}_{0,2}$, and $C_1 = C_2 = 2/\pi$. Moreover, using the relation $\cos(2x) = 2\cos^2(x) - 1$ and $|\sin(x)| \leq x$ for $x \geq 0$, we have
\begin{align*}
&\ms\int_{-1}^{1}\abs*{\hat{\bar{h}}_{0,1 }(u)}du = 2\int_0^1\abs*{\frac{1-\cos(4\pi uM_1)}{\pi u}}du = 2\int_0^{M_1}\frac{1-\cos(4\pi v)}{\pi v}dv\\
&= \frac{4}{\pi}\int_0^{M_1}\frac{\sin^2(2\pi v)}{v}dv \leq \frac{4}{\pi}\int_0^1\frac{\abs*{\sin(2\pi v)}}{v}dv + \frac{4}{\pi}\int_1^{M_1}\frac{1}{v}dv = 8 + \frac{4}{\pi}\log(M_1).
\end{align*}
Therefore, $F(t)\asymp \log^2(M_1/h)\asymp \log^2(M_1/M_2) + \log^2\log(1/t)$ provided that $M_1$ is sufficiently large and $M_2,t$ are sufficiently small. 
In such cases where $F$ cannot be upper bounded by some absolute constant, one can apply separately the Bernstein inequality derived in Theorem 2 in \cite{merlevede2009bernstein} (cf. Lemma \ref{lemma:sample_mean_Bern} in the supplement) and the degenerate version of Proposition \ref{prop:bern_U_alpha} to the sample mean term and degenerate terms in the Hoeffding decomposition of $V_n$, and then combine the results to obtain sharper control of the tail probability. See Lemmas \ref{lemma:perturbation}-\ref{lemma:perturbation_indep} in the supplement for examples of separate control of the sample mean term and the degenerate terms in the context of partially linear regression.

It is worthwhile to mention that, although Theorem \ref{thm:discontinuous_alpha} only covers discrete kernels that are piecewise constant in each component, its proof technique is ready to be used to analyze kernels with more complicated structures. Moreover, the stability result developed in the following Proposition \ref{prop:stability_sum_alpha} will allow us to handle more complicated kernels that are built upon simpler ones. Define the symmetrized version of an order-$m$ kernel as
\begin{align*}
f_\circ(x_1,\ldots,x_m) \define \frac{1}{m!}\sum_{\pi}f(x_{\pi(1)},\ldots,x_{\pi(m)}),
\end{align*}
where the summation is taken over all $m!$ permutations of $[m]$. We will consider composite kernels built from $f^1,\ldots,f^N$, which are $N$ not necessarily symmetric kernels of order $m$ and take not necessarily the same $m$ arguments.

\begin{proposition}
\label{prop:stability_sum_alpha}
Suppose that for any $t>0$ and $i\in[N]$, there exists a not necessarily symmetric approximating kernel $\widetilde{f}^i$ such that $|f^i-\widetilde{f}^i|\leq t$ uniformly over some not necessarily symmetric set $\cal{C}^i\subset\RR^{md}$. In addition, assume $\widetilde{f}^i$ satisfies expansion \eqref{eq:expansion} with constants $\bbrace*{F^i(t,\cal{C}^i),B^i(t,\cal{C}^i),\mu_a^i(t,\cal{C}^i)}$ for all $a\geq 1$. Define
\begin{align}
\label{eq:composite_C}
\cal{C}^0 \define \cal{C}^1\times \ldots\times \cal{C}^N \text{ and } \cal{C}^0_\circ \define \bigcap_\pi \Big\{\pi(\cal{C}^0)\Big\},
\end{align}
where $\pi$ is taken over all permutations of $(x^1_1,\ldots,x^1_m,\ldots,x^N_1,\ldots,x^N_m)$. Then, for any given $t>0$, we have
\begin{enumerate}
\item (Additivity) The symmetrized version $f^0_\circ$ of the kernel 
\begin{align*}
f^0(x^1_1,\ldots,x^1_m,x^2_1,\ldots,x^2_m,\ldots,x^N_1,\ldots,x^N_m) \define \sum_{i=1}^N\lambda_if^i(x^i_1,\ldots,x^i_m)
\end{align*}
satisfies Condition (A) with $\cal{C} = \cal{C}^0_\circ$ and expansion constants 
\begin{align*}
\bbrace*{F(t),B(t),\mu_a(t)} = \bbrace*{\sum_{i=1}^N \abs*{\lambda_i}F^i(t_i), B^1(t_1)\vee\ldots\vee B^N(t_N), \mu_a^1(t_1)\vee\ldots\vee\mu_a^N(t_N)},
\end{align*}
for all $a\geq 1$, where $\{\lambda_i\}_{i=1}^N$ are real numbers, and $t_i \leq t/\parr*{\sum_{i=1}^N|\lambda_i|}$ for each $i\in[N]$.
\item (Multiplicativity) Suppose $f^1,\ldots,f^N$ are upper bounded in absolute value by $M^1,\ldots,M^N$ on $\cal{C}^1,\ldots, \cal{C}^N$, respectively, with $M^i\geq 1$ for $i\in[N]$. Then, the symmetrized version $f^0_\circ$ of the kernel 
\begin{align*}
f^0(x^1_1,\ldots,x^1_m,x^2_1,\ldots,x^2_m,\ldots,x^N_1,\ldots,x^N_m)\define \prod_{i=1}^N f^i(x^i_1,\ldots,x^i_m)
\end{align*}
satisfies Condition (A) with $\cal{C}=\cal{C}^0_\circ$ and expansion constants 
\begin{align*}
\bbrace*{F(t),B(t),\mu_a(t)} = \bbrace*{\prod_{i=1}^NF^i(t_i), B^1(t_1)\vee\ldots\vee B^N(t_N), \mu_a^1(t_1)\vee\ldots\vee\mu_a^N(t_N)}
\end{align*}
for all $a\geq 1$, where $t_i \leq t(M^i + t)/\bbrace*{N\prod_{i=1}^N (M^i + t)}$ for each $i\in[N]$.
\end{enumerate}
\end{proposition}

As an example, Spearman's rho correlation coefficient can be readily analyzed with the help of Proposition \ref{prop:stability_sum_alpha}. Given a stationary sequence $\{X_i\}_{i=1}^n = \{(X_{i,1}, X_{i,2})\}_{i=1}^n$ in $\RR^2$, following \cite{hoeffding1948class}, the sample Spearman's rho is proportional to
\begin{align*}
\rho \define n^{-3}\sum_{i, j, k=1}^n\sgn\parr*{X_{i,1}-X_{j,1}}\sgn\parr*{X_{i,2} - X_{k,2}},
\end{align*}
and is thus a third-order V-statistic generated by the asymmetric kernel $f:\RR^6\rightarrow\RR$ defined as $f(x,y,z) \define \sgn(x_1-y_1)\sgn(x_2-z_2)$, with $x,y,z \in\RR^2$. Note that $\rho$ can be equivalently written as
\begin{align*}
\rho = n^{-3}\sum_{i,j,k=1}^n f_\circ(X_i,X_j,X_k),
\end{align*}
where $f_\circ$ is the symmetrized version of $f$. Note that $f$ can be written as the product of two kernels $f(x,y,z) = f^1(x,y)f^2(x,z)$ with $f^1(x,y)\define f^1_0(x-y)\define\sgn(x_1-y_1)$ and $f^2(x,y)\define f^2_0(x-y)\define \sgn(x_2-y_2)$ being both shift-invariant and in the product form. Moreover, both $f^1$ and $f^2$ are upper bounded in absolute value by $1$. 
With $d = m = 2$, let
\begin{align*}
\cal{C}^1 &\define \Big\{ [-M_1,M_1]^4\Big\}\bigcap \Big\{(u,v)\in\RR^4: \abs*{(u_1-v_1) - y_k}\geq M_2, k\in[J]\Big\},\\
\cal{C}^2 &\define \Big\{ [-M_1,M_1]^4\Big\}\bigcap \Big\{(u,v)\in\RR^4: \abs*{(u_2-v_2) - y_k}\geq M_2, k\in[J]\Big\},
\end{align*}
where $J = 3$ and $\{y_1,y_2,y_3\} = \{-2M_1, 0, 2M_1\}$ as illustrated for the truncated sign function after Theorem \ref{thm:discontinuous_alpha}. Then, the symmetric set $\cal{C}^0_\circ$ in \eqref{eq:composite_C} in Proposition \ref{prop:stability_sum_alpha} becomes
\begin{align}
\label{eq:rho_C}
\Big\{[-M_1,M_1]^6\Big\}\bigcap \Big\{(x^1,x^2,x^3)\in\RR^6: \abs*{(x^i_\ell-x^j_\ell) - y_k}\geq M_2,  i\neq j, \ell\in[2], k\in[3]\Big\},
\end{align}
where we collapse the dimension from $8$ to $6$ due to the common first argument of $f^1$ and $f^2$. Consider the truncated version of $f$ as follows:
\begin{align*}
\bar{f}(x,y,z;M_1) \define \sgn(x_1-y_1)\sgn(x_2-z_2)\mathbbm{1}\parr*{\abs*{x_1-y_1}\leq 2M_1}\mathbbm{1}\parr*{\abs*{x_2-z_2}\leq 2M_1}. 
\end{align*}
Since $f(x,y,z) = \bar{f}(x,y,z)$ uniformly over $\cal{C}^0_\circ$, it suffices to consider the approximation of $\bar{f}$. Using the same proof technique of Theorem \ref{thm:discontinuous_alpha} and the calculation for the truncated sign function after its presentation, we obtain that for any $t > 0$, there exists a $\widetilde{f}^1$ such that $\abs*{f^1 - \widetilde{f}^1}\leq t$ uniformly over $\cal{C}^1$ and $\widetilde{f}^1$ satisfies expansion \eqref{eq:expansion} with constants $(F(t),B(t),\mu_a(t)) \asymp (\log(M_1/h), 1, 1)$ for all $a\geq 1$, where $h \asymp M_2\log^{-1/2}(1/t)$. A similar argument holds for $f^2$. Therefore, the second part of Proposition \ref{prop:stability_sum_alpha} gives
\begin{align*}
F(t) \asymp \log^2(M_1/M_2) + \log^2\log(1/t), \quad B(t) \asymp 1, \quad \mu_a(t) \asymp 1
\end{align*}
for all $a\geq 1$, provided that $M_1$ is sufficiently large and $M_2, t$ are sufficiently small.

\subsection{Extension to general distributions}
\label{subsec:support}

In this subsection, we extend the results in Propositions \ref{prop:bern_U_alpha} and \ref{prop:MDP_U}  to arbitrarily supported $\{X_i\}_{i=1}^n$ via a standard truncation argument. Recall that, given an order-$m$ symmetric kernel $f$ defined on $\RR^{md}$ of degeneracy level $r - 1$, its Hoeffding decomposition takes the form
\begin{align}
\label{eq:Hoeffding_decomposition}
f(x_1,\ldots,x_m)  - \theta=  \sum_{1\leq i_1<\ldots <i_r\leq m}f_r\parr*{x_{i_1},\ldots,x_{i_r}} + \ldots + f_m(x_1,\ldots,x_m),
\end{align}
where $\{f_p\}_{p=1}^m$ are real symmetric functions on $\RR^{pd}$ recursively defined as 
\begin{equation*}
\begin{aligned}
&f_1(x) \define g_1(x),\\
&f_p\parr*{x_1,\ldots,x_p} \define g_p(x_1,\ldots,x_p) -\sum_{k=1}^p f_1(x_k)- \ldots - \sum_{1\leq k_1<\ldots <k_{p-1} \leq p}f_{p-1}\parr*{x_{k_1},\ldots,x_{k_{p-1}}},
\end{aligned}
\end{equation*}
for $p=2,\cdots,m$, with $\{g_p\}_{p=1}^m$ defined as $g_m = f - \theta$, and 
\begin{align*}
g_p(x_1,\ldots,x_p) \define \E\bbrace*{f(x_1,\ldots,x_p,\widetilde{X}_{p+1},\ldots,\widetilde{X}_m)} - \theta
\end{align*}
for $1\leq p\leq m-1$. One can readily check that $f_p$ is centered and fully degenerate for any $p\in[m]$. 

When $\{X_i\}_{i=1}^n$ are not compactly supported, one additional technical issue is to control each $|f_p - \widetilde{f}_p|$ over certain set $\cal{C}_p\subset\RR^{pd}$, where $\{f_p\}_{p=1}^m$ and $\{\widetilde{f}_p\}_{p=1}^m$ are the fully degenerate terms in the Hoeffding decomposition of $f$ and $\tilde f$, respectively. Due to the symmetry of $f$ and the set $\cal{C}$ in Condition (A), a natural choice of $\cal{C}_p$ is the projection of $\cal{C}$ onto $\RR^{pd}$, that is, 
\begin{align}
\label{eq:section}
\cal{C}_p \define \Big\{(x_1^\top,\ldots,x_p^\top)^\top \in \RR^{pd}: \text{ there exists } (x_{p+1}^\top,\ldots,x_m^\top)^\top \text{ s.t. } (x_1^\top,\ldots,x_m^\top)^\top \in \cal{C})\Big\}.
\end{align}
For instance, when $\cal{C}$ is  $[-M,M]^{md}$, each $\cal{C}_p$ is $[-M,M]^{pd}$. Further define, for each $p\in[m-1]$,
\begin{align*}
s_0 \define \bbrace*{\E f^2(\widetilde{X}_1,\ldots,\widetilde{X}_m)}^{1/2},~~~ s_p \define \sup_{\cal{C}_p}\bbrace*{\E f^2(x_1,\ldots,x_p,\widetilde{X}_{p+1},\ldots,\widetilde{X}_m)}^{1/2},
\end{align*}
and 
\begin{align*}
v_0 \define \bbrace*{\P\parr*{(\widetilde{X}_1^\top,\ldots,\widetilde{X}_m^\top)^\top\notin \cal{C}}}^{1/2},~~~ v_p \define\sup_{\cal{C}_p}\bbrace*{\P\parr*{(\widetilde{X}_{p+1}^\top,\ldots,\widetilde{X}_m^\top)^\top\notin \cal{C}^{(x_1,\ldots,x_p)}}}^{1/2},
\end{align*}
where
\begin{align*}
\cal{C}^{(x_1,\ldots,x_p)} \define \Big\{(x_{p+1}^\top, \ldots, x_m^\top)^\top \in \RR^{(m-p)d}: (x_1^\top,\ldots,x_m^\top)^\top \in \cal{C}\Big\}
\end{align*}
for all $(x_1^\top,\ldots,x_p^\top)^\top \in\cal{C}_p$. 



The following proposition provides, under Condition (A), a uniform upper bound on $|f_p - \widetilde{f}_p|$ over $\cal{C}_p$ for $p\in[m]$.

\begin{proposition}
\label{prop:hoeffding}
Suppose that $f$ satisfies Condition ($A$) with some symmetric set $\cal{C}\subset\RR^{md}$ and constants $(F(t,\cal{C}),B(t,\cal{C}),\mu_a(t,\cal{C}))$. Then, for any $p\in[m]$, it holds that
\begin{align*}
\abs*{f_p(x_1,\ldots,x_p) - \widetilde{f}_p(x_1,\ldots,x_p)} \leq t'
\end{align*}
uniformly over $\cal{C}_p$ defined in \eqref{eq:section}, where $t'$ is defined to be
\begin{align}
\label{eq:bias}
t' \define C\parr*{t + \sum_{i=0}^{m-1}s_iv_i + FB^m \sum_{i=0}^{m-1}v_i}
\end{align}
for some absolute positive constant $C = C(m)$. 
\end{proposition}

As shown by \eqref{eq:bias}, the gap between the approximation bias $t$ in $|f-\widetilde{f}|$ and $t'$ in $|f_p-\widetilde{f}_p|$ relies on the magnitudes of $\{s_iv_i\}_{i=0}^{m-1}$ and $\{FB^mv_i\}_{i=0}^{m-1}$. On one hand, $\{s_p\}_{p=0}^{m-1}$ can often be upper bounded through direct calculation, and in particular, could be uniformly upper bounded by $\|f\|_{\infty}$ when it is finite. On the other hand, by definition, the sequence $\{v_i\}_{i=1}^m$ is governed by the tail behavior of the underlying data sequence. In many applications, due to the light dependence of $F$ on the magnitude of $\cal{C}$ (non-dependence or polylogarithmically), by choosing a sufficiently large $\cal{C}$ (a typical choice is $[-n^\eta,n^\eta]^{md}$ for some sufficiently large $\eta$ depending on the moment conditions of $\{X_i\}_{i=1}^n$), the gap between $t$ and $t^\prime$ and the residual probability $R_{n,s}$ defined below are both negligible. In the special case when $X_1$ is compactly supported, $\{v_i\}_{i=0}^{m-1}$ are all zero by choosing $\cal{C} = \supp(\P^m)$, and thus $t \asymp t^\prime$. 

The following are the analogous versions of Propositions \ref{prop:bern_U_alpha} and \ref{prop:MDP_U} for $\{X_i\}_{i=1}^n$ with arbitrary support. For any symmetric set $\cal{C}\subset\RR^{md}$, and positive integers $1\leq s \leq m$ and $n$, we use the notation
\begin{align*}
R_{n,s} \define \P\parr*{\text{there exists } s\leq p\leq m, 1\leq i_1,\ldots,i_p\leq n \text{ s.t. } (X_{i_1}^\top,\ldots,X_{i_p}^\top)^\top \notin \cal{C}_p}
\end{align*}
to represent the residual probability outside the sets $\{\cal{C}_p\}_{p=1}^m$ defined in \eqref{eq:section}.

\begin{propbis}{prop:bern_U_alpha}
\label{prop:bern_U_alpha_general}
Suppose $n\geq 2$, Condition (M) holds, and $f$ in \eqref{eq:U} satisfies Condition (A) with set $\cal{C}$ and constants $(F(t,\cal{C}),B(t,\cal{C}),\mu_a(t,\cal{C}))$. Then, if $f$ is degenerate of level $r-1$, there exist positive absolute constants $C_1 = C_1(m), C_2 = C_2(m,\gamma_1,\gamma_2)$ such that, for any $x > 0$ and $t > 0$,
\begin{align*}
\P\bbrace*{\abs*{V_n - \theta}\geq (x+C_1t')} \leq 2\sum_{p=r}^m \exp\parr*{-\frac{C_2nx^{2/p}}{A_p^{1/p} + x^{1/p}M_p^{1/p}}} + R_{n,r},
\end{align*}
where $\{A_p\}_{p=1}^m$, $\{M_p\}_{p=1}^m$ are defined in \eqref{eq:bern_seq} with $(F(t),B(t),\mu_a(t))$, and $t'$ is defined in \eqref{eq:bias}.
\end{propbis}

Note that, when $X_1$ is supported within $[-M_X,M_X]^d$ for some $M_X > 0$, by choosing $\cal{C} = [-M_X,M_X]^{md}$ in Condition (A), we have $t'\asymp t$ and $R_{n,r} = 0$ for any $r\in[m]$. Thus Proposition \ref{prop:bern_U_alpha_general} is reduced to Proposition \ref{prop:bern_U_alpha}. 

\begin{propbis}{prop:MDP_U}
\label{prop:MDP_U_general}
Suppose that Condition (M) holds. Assume $f$ in \eqref{eq:U} is centered, nondegenerate, and $\nu^2$ defined in \eqref{eq:MDP_notation} is strictly positive. Assume further that $f$ satisfies Condition (A) with set $\cal{C}$ and constants $(F(t,\cal{C}),B(t,\cal{C}),\mu_a(t,\cal{C}))$, and $f$ satisfies condition \eqref{eq:MDP_moment} with some $\gamma, \eta > 0$. Moreover, suppose that $R_{n,2} = o(n^{-\gamma^2/2}/\sqrt{\log n})$, and for some $t = O(1/\{\sqrt{n}(\log n)^2\})$, \eqref{eq:MDP_condition} holds and $t'$ defined in \eqref{eq:bias} satisfies $t' = O(t)$. Then, the conclusion of Proposition \ref{prop:MDP_U} holds.
\end{propbis}


We now combine Proposition \ref{prop:bern_U_alpha_general} with Theorems \ref{thm:random_fourier_alpha}-\ref{thm:discontinuous_alpha} to derive the following user-friendly tail bounds in the case of arbitrarily supported $\{X_i\}_{i=1}^n$. 
The first corollary is on continuous kernels under the choice $\cal{C} = [-M,M]^{md}$ in Condition (A) with corresponding $\cal{C}_p = [-M,M]^{pd}$. A typical choice in practice is $M = n^{\eta}$ for some sufficiently large $\eta$.


\begin{corollary}
\label{cor:continuous_general}
Suppose (M) holds and for some given $M>0$ and $\cal{C} = [-M,M]^{md}$, $f$ satisfies one of the Conditions (B1)-(B3) with set $\cal{C}$. Let $f$ be degenerate of level $r - 1$. Then, there exist positive absolute constants $C_1 = C_1(m), C_2 = C_2(m,\gamma_1,\gamma_2)$ such that, for any $x>0$ and $t>0$, it holds that
\begin{align*}
\P\bbrace*{\abs*{V_n - \theta}\geq (x+C_1t')} \leq 2\sum_{p=r}^m \exp\parr*{-\frac{C_2nx^{2/p}}{A_p^{1/p} + x^{1/p}M_p^{1/p}}} + n\sum_{\ell=1}^d \P\parr*{\abs*{X_{1,\ell}}\geq M},
\end{align*}
where $t'$ is defined in \eqref{eq:bias}, and $\{A_p\}_{p=1}^m,\{M_p\}_{p=1}^m$ take the values $A_p \asymp F(t)^2, M_p \asymp F(t)(\log n)^{2p}$ with the corresponding $F(t)$ under Conditions (B1)-(B3), respectively.
\end{corollary}


The second corollary is on discontinuous kernels of the form considered in Theorem \ref{thm:discontinuous_alpha} and $\cal{C}\subset\RR^{2d}$ in Condition (A) is chosen to be
\eqref{eq:discontinuous_support}. It can be readily checked that in this case, the corresponding $\cal{C}_1$ defined by \eqref{eq:section} is $[-M_1,M_1]^d$. A typical choice in practice is $M_1 = n^{\eta_1}, M_2 = n^{-\eta_2}$ for sufficiently large $\eta_1,\eta_2$.

\begin{corollary}
\label{cor:discontinuous_tail}
Consider the same setting as Theorem \ref{thm:discontinuous_alpha}. Suppose that (M) holds and for some given $M_1,M_2 > 0$, $f$ satisfies Condition (B4) therein with $M_1,M_2$. Assume further that $X_1$ is absolutely continuous and uniformly over all $\ell\in[d]$ and $1\leq i < j \leq n$, the density of $(X_{i,\ell} - X_{j,\ell})$ is upper bounded by some constant $D = D(n)>0$. Let $f$ be degenerate of level $r-1$.\\ 
(a) For any $x > (\abs*{f_0(0)} + F(t))/n$ and $t > 0$, it holds that
\begin{align*}
&\ms\P\parr*{\abs*{V_n - \theta}\geq x+C_1t'}\\
& \leq 2\sum_{p=r}^2 \exp\parr*{-\frac{C_2ny^{2/p}}{A_p^{1/p} + y^{1/p}M_p^{1/p}}} + n^2\parr*{\sum_{\ell=1}^d J_\ell}M_2D + n\sum_{\ell=1}^d \P\parr*{\abs*{X_{1,\ell}}\geq M_1},
\end{align*}
where $y = x - (\abs*{f_0(0)} + F(t))/n$.\\
(b) 
Suppose further that $0<M_2\leq \inf_{\ell\in[d],k\in[J_\ell]}\{\abs*{y_{\ell,k}}\}$ for $y_{\ell,k}$ defined in Condition (B4) in Theorem \ref{thm:discontinuous_alpha}. Then, for any $x > 0$ and $t > 0$, it holds that
\begin{align*}
&\ms\P\parr*{\abs*{V_n - \theta}\geq x+C_1t'}\\
& \leq 2\sum_{p=r}^2 \exp\parr*{-\frac{C_2nx^{2/p}}{A_p^{1/p} + x^{1/p}M_p^{1/p}}} + n^2\parr*{\sum_{\ell=1}^d J_\ell}M_2D + n\sum_{\ell=1}^d \P\parr*{\abs*{X_{1,\ell}}\geq M_1}.
\end{align*}
In both (a) and (b), $t'$ is defined in \eqref{eq:bias}, $C_1$ is some absolute constant,  $C_2 = C_2(\gamma_1,\gamma_2)$, and $\{A_p\}_{p=1}^2$ and $\{M_p\}_{p=1}^2$ take the values $A_p \asymp F(t)^2, M_p \asymp F(t)(\log n)^{2p}$, with $F(t)$ specified in Theorem \ref{thm:discontinuous_alpha}.
\end{corollary}

The extra condition in part (b), $\abs*{y_{\ell,k}}> 0$ for all $\ell\in[d]$ and $k\in[J_\ell]$, is satisfied by, for example, the box kernel discussed after Theorem \ref{thm:discontinuous_alpha}. However, it excludes Kendall's tau and Spearman's rho. 





We end this section with a summary table of example kernels that satisfy Condition (A), along with the corresponding set $\cal{C}$ and expansion constants $(F(t, \cal{C}),B(t,\cal{C}),\mu_a(t,\cal{C}))$.  Throughout the table, $M_1$ is assumed to be sufficiently large and $M_2, t$ are assumed to be sufficiently small. $C$ represents a large enough absolute positive constant. 
We use the notation $\|X_{1,\cdot}\|_a \define \max_{\ell\in[d]}\bbrace*{\E\parr*{|X_{1,\ell}|^a}}^{1/a}$.

\begin{table}[h!]
\centering
\caption{Example kernels with corresponding set $\cal{C}$ and constants $(F(t,\cal{C}),B(t,\cal{C}),\mu_a(t,\cal{C}))$ in Condition (A).} 
\begin{tabular}{| c | c | c | c |} 
 \hline
 Name & Definition & $\cal{C}$ & $(F,B,\mu_a)$ \\ [0.5ex] 
 \hline\hline
 Linear & $x^\top y$ & $[-M,M]^{2d}$ & $ (d,M,\|X_{1,\cdot}\|_a)$\\
 Gaussian & $\exp(-\|x-y\|^2/2)$ & $[-M,M]^{2d}$ & $(2,1,1)$ \\ 
 Laplacian & $\exp(-\|x-y\|_1)$ & $[-M,M]^{2d}$ & $(2, 1,1)$ \\
 Cauchy & $\prod_{\ell=1}^d2/\bbrace*{1+(x_\ell-y_\ell)^2}$ & $[-M,M]^{2d}$ & $(2^{d+1},1,1)$ \\
 Hat & $(1-|x-y|)\mathbbm{1}(|x-y|\leq 1)$ & $[-M,M]^2$ & $(2,1,1)$ \\
 Cosine & $\cos(x-y)\mathbbm{1}(|x-y|\leq \pi/2)$ & $[-M,M]^2$ & $(32,1,1)$ \\ 
 Size-1 Box & $\mathbbm{1}(|x-y|\leq 1)$ & \eqref{eq:discontinuous_support}, $d = 1$ & $\parr*{C\log\parr*{\frac{1}{M_2}}+C\log\log\parr*{\frac{1}{t}},1,1}$\\
 Kendall's Tau & $\sgn(x_1-y_1)\sgn(x_2-y_2)$ & \eqref{eq:discontinuous_support}, $d = 2$ & $\parr*{C\log^2\parr*{\frac{M_1}{M_2}} + C\log^2\log\parr*{\frac{1}{t}}, 1,1}$\\
 Spearman's Rho &  $\sgn(x_1-y_1)\sgn(x_2-z_2)$ & \eqref{eq:rho_C} & $\parr*{C\log^2\parr*{\frac{M_1}{M_2}} + C\log^2\log\parr*{\frac{1}{t}}, 1,1}$\\[1ex] 
 \hline
\end{tabular}\label{tab:1}
\end{table}

\section{Applications}
\label{sec:app}
In this section, we discuss two statistical applications of the developed theory to high-dimensional estimation and testing problems. 

\subsection{Stochastic partially linear model}
\label{subsec:PLR}

\subsubsection{Set-up}
\label{subsubsec:model}

Consider the following high-dimensional stochastic partially linear model, 
\begin{align}
\label{eq:model}
Y_i = X_i^\top \beta^* + g(W_i) + \varepsilon_i, \qquad i=1,\ldots, n,
\end{align}
where the sequence $\{X_i, W_i, \varepsilon_i\}_{i=1}^n\in \RR^p \times \RR \times \RR$ is jointly stationary, $g(\cdot)$ is some unknown function considered to be a nuisance parameter, and the noise sequence $\{\varepsilon_i\}_{i=1}^n$ is assumed to be independent from $\{X_i\}_{i=1}^n$ and $\{W_i\}_{i=1}^n$. The dimension $p$ of the linear component is allowed to be much larger than $n$, and we assume the true parameter of interest $\beta^*$ is $s$-sparse. 

The partially linear model is of fundamental importance in statistics and econometrics. In many applications, temporal dependence and high dimensionality naturally occur (cf. Chapter 18 in \cite{li2007nonparametric} and \cite{buhlmann2011statistics}). 
Recently, \cite{han2017adaptive} studied the penalized Honor\'{e}-Powell estimator \citep{honore1997pairwise}
\begin{align}
\label{est:PRD}
\hat{\beta}_{h_n} \define \argmin_{\beta\in\RR^p} \bbrace*{{n\choose 2}^{-1}\sum_{i<j}\frac{1}{h_n}K\parr*{\frac{\WW}{h_n}}\parr*{\tilde{Y}_{ij} - \tilde{X}_{ij}^\top \beta}^2 + \lambda_n\|\beta\|_1},
\end{align}
where $\YY\define Y_i-Y_j, \XX \define X_i-X_j, \WW \define W_i - W_j$, $K(\cdot)$ is a positive density kernel and $h_n$ is the bandwidth parameter. 
The authors showed that, under i.i.d.-ness of $\{Y_i,X_i,W_i\}_{i=1}^n$ and some other regularity conditions,  the pairwise difference estimator in \eqref{est:PRD} achieves the minimax rate with high probability:
\begin{align}
\label{eq:optimal_rate}
\nm*{\hat{\beta}_{h_n}- \beta^*}^2 = O_{\P}\parr*{\frac{s\log p}{n}}.
\end{align}
However, it is unknown in the literature whether the same conclusion is true when the data are dependent. Note that the first term of the loss function in \eqref{est:PRD} is essentially  a U-statistic of order 2. In what follows, we will make use of Corollary \ref{cor:continuous_general} to provide explicit conditions under which the optimal rate in \eqref{eq:optimal_rate} can be recovered in the high-dimensional time series setting. 

\subsubsection{Optimality results}
\label{subsubsec:assumptions}

We pose the following assumptions on the data generating scheme of model \eqref{eq:model}, which largely follow those in \cite{han2017adaptive}.

\begin{assumption}[Mixing conditions]
\label{ass:alpha_mixing}
$\bbrace*{X_{i,k}, W_i, \varepsilon_i}_{i=1}^n$ for all $k\in[p]$ and $\bbrace*{X_{i,k},X_{i,\ell}, W_i}_{i=1}^n$ for all $(k,\ell)\in[p]\times[p]$ are assumed to be part of a stationary and geometrically $\alpha$-mixing sequence, respectively,  with coefficient $\alpha(i) \leq \gamma_1\exp(-\gamma_2 i)$ for all $i\geq 1$ and positive constants $\gamma_1,\gamma_2$.
\end{assumption}

\begin{remark}
\label{remark:mixing}
In Assumption \ref{ass:alpha_mixing}, instead of assuming that the $(p+2)$-dimensional vector $\{X_i,W_i,\varepsilon_i\}_{i=1}^n$ is part of a stationary and geometrically $\alpha$-mixing sequence, only subprocesses with a fixed dimension $3$ are required to be stationary and geometrically $\alpha$-mixing. 
\end{remark}

\begin{assumption}[Kernel condition]
\label{ass:kernel_alpha}
$K(\cdot)$ is chosen to be a continuous, positive definite density kernel such that $K(u)\geq 0$ for all $u\in\RR$ and $\int K(u)du = 1$. We also assume there exists some absolute positive number $M_K$ larger than $1$, such that
\begin{align*}
\max\bbrace*{\int |u|^3 K(u) du,~ \sup_{u\in\RR}|u|K(u),~ \sup_{u\in\RR} K(u)} \leq M_K.
\end{align*}
\end{assumption}



\begin{assumption}[Density condition of $W$]
\label{ass:density_W}
We assume there exists a positive number $M_W\geq 0$ such that
\begin{align*}
\max\bbrace*{\abs*{\frac{\partial f_{W\mid X}(w, x)}{\partial w}}, f_W(w)}\leq M_W
\end{align*}
for arbitrary $(w,x)$ in the range of $(W,X)$, where $f_{W\mid X}$ represents the conditional density of $W$ given $X$.
\end{assumption}

\begin{assumption}[Regularity of $\WW$ around 0]
\label{ass:informative}
For any pair $(i,j)\in [n]\times [n]$ and $i\neq j$, $\ff_{\WW}(0) \geq M_\ell > 0$ for some positive constant $M_\ell$, where $\ff_{\WW}$ is the density of $\WW$ under the product measure.
\end{assumption}

\begin{assumption}[Restricted eigenvalue condition]
\label{ass:RE}
There exists a positive constant $\kappa_\ell$  such that
\begin{align*}
v^\top \EE\parr*{\XX\XX^\top\mid\WW = 0}v \geq \kappa_\ell \|v\|^2
\end{align*}
for all $v\in\{\beta\in\RR^p:\|\beta_{S_*^c}\|_1\leq 3\|\beta_{S_*}\|_1\}$, where $S_*$ is the support of $\beta^*$ and for any measurable $f$, $\EE\bbrace*{f(X_i,X_j)}$ is the integration of $f$ under the product measure.
\end{assumption}


\begin{assumption}[Distribution of $\{X_i\}_{i=1}^n$]
\label{ass:sub-Gaussian-1}
Uniformly over $(i,j)\in [n] \times [n]$, $i\neq j$ and $k\in[p]$, $X_{ik}$ and $X_{ik}\mid \WW = 0$ are sub-Gaussian with constant $\kappa_x^2$.
\end{assumption}

\begin{assumption}[Distribution of $\{\varepsilon_i\}_{i=1}^n$]
\label{ass:dist_eps_alpha}
$\{\varepsilon_i\}_{i=1}^n$ is assumed to have finite $(2+\delta_1)$th moment for some positive $\delta_1$. 
\end{assumption}

\begin{assumption}[Distribution of $\{W_i\}_{i=1}^n$]
\label{ass:dist_W}
$\{W_i\}_{i=1}^n$ is assumed to have finite $(4+\delta_2)$th moment for some positive $\delta_2$. 
\end{assumption}

\begin{assumption}[Smoothness of $g(\cdot)$]
\label{ass:g}
The nonlinear component $g(\cdot)$ in \eqref{eq:model} is assumed to be $L$-Lipschitz, that is, there exists some absolute positive constant $L$ such that for any $w_1,w_2$ in the domain of $W$, it holds that
\begin{align*}
\abs*{g(w_1)-g(w_2)} \leq L\abs*{w_1-w_2}.
\end{align*}
\end{assumption}



The following is the main result of this section. It is built upon Corollaries \ref{cor:random_fourier_alpha} and \ref{cor:continuous_general}. 
\begin{thm}
\label{thm:error_bound}
Assume that there exist positive constants $C_1,C_2$ such that $C_1(\log p/n)^{1/2} \leq C_2$. In addition, suppose that Assumptions \ref{ass:alpha_mixing}-\ref{ass:g} hold and the tuning parameters $\lambda_n,h_n$ in estimator \eqref{est:PRD} satisfy
\begin{align*}
C_1(\log p /n)^{1/2} \leq h_n \leq C_2 \text{ and } \lambda_n \geq  C_3\bbrace*{h_n + (\log p/n)^{1/2}},
\end{align*}
and moreover,
\begin{align*}
n \geq C_3\bbrace*{s^2(\log p) \vee (\log p)^3(\log n)^4 \vee (\log p)^{\frac{8+2\delta_1}{\delta_1}}(\log n)^{\frac{16+4\delta_1}{\delta_1}}\vee (\log p)^{\frac{16+2\delta_2}{4+\delta_2}}(\log n)^{\frac{16+2\delta_2}{2+\delta_2}}},
\end{align*}
where $s$ is the sparsity level of the true $\beta^*$, $\delta_1,\delta_2$ are positive constants in Assumptions \ref{ass:dist_eps_alpha} and \ref{ass:dist_W}, and $C_3$ only depends on $C_1,C_2,\gamma_1,\gamma_2, M_K, M_W, M_\ell,\kappa_\ell,\kappa_x,\E\parr*{|\varepsilon|^{2+\delta_1}}, \E\parr*{|W|^{4+\delta_2}}, L$. Then, it holds that 
\begin{align*}
\nm*{\hat{\beta}_{h_n}-\beta^*}^2\leq cs\lambda_n^2
\end{align*}
with probability at least $1- \exp\parr*{-c'\log p} - c''n^{-\delta_1/(4+\delta_1)} - c''n^{-\delta_2/(8+\delta_2)}$, where $c, c',c''$ are positive constants that only depend on $C_1,C_2,\gamma_1,\gamma_2, M_K, M_W, M_\ell,\kappa_\ell,\kappa_x,\E\parr*{|\varepsilon|^{2+\delta_1}}, \E\parr*{|W|^{4+\delta_2}}, L$.
\end{thm}

If we choose $h_n$ on the order of $(\log p/n)^{1/2}$ and $\lambda_n$ further on the order of $(\log p/n)^{1/2}$, then the squared $\ell_2$ error of the estimator \eqref{est:PRD} can be upper bounded with high probability on the order of $s\log p/n$, which achieves the minimax Lasso rate as if there were no nonparametric component. This error bound extends the optimal rate of the partially linear model derived in \cite{muller2015partial}, \cite{zhu2017nonasymptotic}, and \cite{han2017adaptive} under the high-dimensional i.i.d. setting to the time series setting.

\subsection{Multiple testing of independence in high dimensions}
\label{subsec:testing}


Testing pairwise independence is of fundamental importance in statistics. In practice, it is often the case that, for each variable, we observe a time series instead of independent realizations. See, for instance, \cite{haugh1976checking}, \cite{hong1996testing}, and \cite{duchesne2003robust}, among others, for testing independence in the time series setting. In this section, with the help of Proposition \ref{prop:MDP_U} and the following corollary, we will extend such tests to a multiple testing problem of a growing number of hypotheses.

The following corollary reproduces the Cram\'er-type moderate deviation in Proposition \ref{prop:MDP_U_general} under the more verifiable Conditions (B1)-(B4) in Theorems \ref{thm:random_fourier_alpha}-\ref{thm:discontinuous_alpha}.


\begin{corollary}
\label{cor:MDP_U}
Assume the kernel $f$ in \eqref{eq:U} is centered, nondegenerate, satisfies \eqref{eq:MDP_moment} with some $\gamma,\eta > 0$, and $\nu^2$ defined in \eqref{eq:MDP_notation} is strictly positive.\\
(a) Suppose that there exists $M>0$ and $\cal{C} = [-M,M]^{md}$ such that $f$ satisfies one of the Conditions (B1)-(B3) with set $\cal{C}$, and for some $t = O(1/(\sqrt{n}(\log n)^2))$, $t' = t'(t, M)$ defined in \eqref{eq:bias} with set $\cal{C}$ satisfies $t' = O(t)$. Assume further the following holds, 
\begin{align*}
n\sum_{\ell=1}^d\P\parr*{\abs*{X_{1,\ell}}\geq M} = o\bbrace*{n^{-\gamma^2/2}(\log n)^{-1/2}}.
\end{align*}
Then, the result of Proposition \ref{prop:MDP_U} holds provided that \eqref{eq:MDP_condition} holds with constants $(F(t), B(t), \mu_a(t))$ under Conditions (B1)-(B3), respectively.\\
(b) Suppose that there exist $M_1, M_2 > 0$ such that $f$ satisfies (B4) with $M_1, M_2$, and the density of each $(X_{i,\ell} - X_{j,\ell})$ is upper bounded by some $D = D(n) > 0$ uniformly over all $\ell\in[d]$ and $1\leq i<j\leq n$. 
Assume that for some $t = O(1/(\sqrt{n}(\log n)^2))$, $t' = t'(t, M_1,M_2)$ defined in \eqref{eq:bias} with $\cal{C}$ in \eqref{eq:discontinuous_support} satisfies $t' = O(t)$. Assume further the following holds,
\begin{align*}
\bbrace*{n^2\parr*{\sum_{\ell=1}^d J_\ell}M_2D} \vee \bbrace*{n\sum_{\ell=1}^d\P\parr*{\abs*{X_{1,\ell}}\geq M_1}} = o\bbrace*{n^{-\gamma^2/2}(\log n)^{-1/2}}. 
\end{align*}
Then, the result of Proposition \ref{prop:MDP_U} holds provided that $(\abs*{f_0(0)} + F(t))/n = o(1/(\sqrt{n}(\log n)^2))$ and \eqref{eq:MDP_condition} hold with constants $(F(t), B(t), \mu_a(t))$ under Condition (B4).
\end{corollary}

\subsubsection{A general testing scheme}

Consider the following multiple hypothesis testing problem. Given $p$ independent stationary bivariate time series $\{X^{\ell}_i\}_{i=1}^n = \{(X^\ell_{i,1},X^\ell_{i,2})\}_{i=1}^n$ for $\ell\in[p]$, where $p$ is allowed to grow with $n$, we aim to test the hypothesis 
\[
H_0: \{X^\ell_{i,1}\}_{i=1}^n \text{ is independent of } \{X^\ell_{i,2}\}_{i=1}^n \text{ for all } \ell\in[p].
\]
In what follows, we will use the letter $\ell$ in the superscript to represent the $\ell$th pair of time series. For simplicity, we assume that the $p$ bivariate sequences are mutually independent. We propose to test $H_0$ by constructing an order-$2$ U-statistic for each of the $p$ pairs, that is, for $\ell\in[p]$,
\begin{align}
\label{eq:U_test}
U_{n,\ell} \define {n\choose 2}^{-1} \sum_{1\leq i<j\leq n}h^\ell(X^\ell_i,X^\ell_j),
\end{align}
where the sequence of kernels $\{h^\ell\}_{\ell=1}^p:\RR^2\times\RR^2\rightarrow\RR$ is assumed to be symmetric and nondegenerate. For each $\ell\in[p]$, we will use the notation 
\begin{align*}
&h^\ell_1(x) \define \E_{H_0}\bbrace*{h^\ell\parr*{x,\widetilde{X}^\ell_2}}, \quad\theta_\ell \define \E_{H_0}\bbrace*{h^\ell\parr*{\widetilde{X}^\ell_1,\widetilde{X}^\ell_2}},\\
&\sigma_\ell^2 \define \var_{H_0}\bbrace*{h^\ell_1(X^\ell_1)} + 2\sum_{i>1}\cov_{H_0}\bbrace*{h^\ell_1\parr*{X^\ell_1},h^\ell_1\parr*{X^\ell_i}},
\end{align*}
where $\widetilde{X}^\ell_1,\widetilde{X}^\ell_2\in\RR^2$ are two i.i.d. copies of $X^\ell_1$, and $\E_{H_0}, \var_{H_0},\cov_{H_0}$ are expectation, variance, and covariance under the null hypothesis $H_0$. For each $\ell\in[p]$, we define 
\begin{align*}
\widetilde{U}_{n,\ell} \define \frac{\sqrt{n}}{2\sigma_\ell}\parr*{U_{n,\ell}-\theta_\ell}
~~~{\rm and}~~~
S_n \define \max_{1\leq\ell\leq p}\abs*{\widetilde{U}_{n,\ell}}. 
\end{align*}
The following result gives an asymptotically valid test of $H_0$. The Cram\'er-type moderate deviation in Proposition \ref{prop:MDP_U_general} plays a central role in its proof.
\begin{thm}
\label{thm:U_test}
Suppose that for each $\ell\in[p]$, $\{X^\ell_i\}_{i=1}^n$ is part of a stationary sequence $\{X^\ell_i\}_{i\in\mathbb{Z}}$ that is geometrically $\alpha$-mixing with coefficient $\alpha(i)\leq \gamma_1\exp(-\gamma_2 i)$ for all $i\geq 1$ and some positive constants $\gamma_1,\gamma_2$. Suppose the kernel sequence $\{h^\ell\}_{\ell=1}^p$ in \eqref{eq:U_test} is chosen such that each $\widetilde{U}_{n,\ell}$ satisfies \eqref{eq:MDP} with $x\in[0,\gamma\sqrt{\log n}]$. Then, for $p = O\parr*{n^{\gamma^2/2}}$,  the test based on the following rejection region,
\begin{align}
\label{eq:reject_region}
I &\define \bbrace*{S_n^2 - 2\log p + \log\log p \geq q_\alpha},
\end{align}
has asymptotic size $\alpha$, where $q_\alpha = -\log \pi - 2\log\log(1-\alpha)^{-1}$ is the $1-\alpha$ quantile of the Gumbel distribution with distribution function $\text{exp}\{-\pi^{-1/2}\text{exp}(-y/2)\}$. 
\end{thm}

With a similar optimality argument as in \cite{han2017distribution}, one can readily show that under similar conditions to that in Section 4.3 therein, the test based on \eqref{eq:reject_region} is rate-optimal.

\subsubsection{Example: testing of a bivariate AR(1) sequence}
\label{subsubsec:test_example}
We now exemplify Theorem \ref{thm:U_test} with the special case when each component of $\{X^\ell_i\}_{i\in\mathbb{Z}} = \{X^\ell_{i,1}, X^\ell_{i,2}\}_{i\in\mathbb{Z}}$ for $\ell\in[p]$ is an $AR(1)$ sequence that takes the form
\begin{align}
\label{eq:AR}
X^\ell_{i+1,j} = \alpha^\ell_jX^\ell_{i,j} + \varepsilon^\ell_{i+1,j},\qquad i\geq 1
\end{align}
for $j=1,2$, where $|\alpha^\ell_j| < 1$ and $\{\varepsilon^\ell_{i,j}\}_{i=1}^n$ is the innovation sequence. Assuming that $\{\varepsilon^\ell_{i,1}\}_{i=1}^n$ and $\{\varepsilon^\ell_{i,2}\}_{i=1}^n$ are absolutely continuous with respect to the Lebesgue measure, Theorem 1 of \cite{mokkadem1988mixing} guarantees that $\{X^\ell_{i,1}\}_{i\in\mathbb{Z}}$ and $\{X^\ell_{i,2}\}_{i\in\mathbb{Z}}$ are both geometrically $\alpha$-mixing for $\ell\in[p]$. Suppose $\{X^\ell_{i,1}\}_{i=1}^n$ and $\{X^\ell_{i,2}\}_{i=1}^n$ are stationary with distributions $F^\ell_1$ and $F^\ell_2$, respectively for $\ell\in[p]$. Following the independence test of two i.i.d. samples in \cite{kendall1938new}, for each $\ell\in[p]$, we choose $h^\ell$ to be Kendall's tau statistic:
\begin{align*}
h^\ell(x,y) \define \sgn(x_1-y_1)\sgn(x_2-y_2).
\end{align*}
We now, for each $\ell\in[p]$, calculate the parameters $\theta_\ell$ and $\sigma^2_\ell$ in the definition of $\widetilde{U}_{n,\ell}$. Under $H_0$, it holds that
\begin{align*}
\theta_\ell &= \E\bbrace*{\sgn\parr*{\widetilde{X}^\ell_{1,1} - \widetilde{X}^\ell_{2,1}}\sgn\parr*{\widetilde{X}^\ell_{1,2}-\widetilde{X}^\ell_{2,2}}}\\
&= \E\bbrace*{\sgn\parr*{\widetilde{X}^\ell_{1,1} - \widetilde{X}^\ell_{2,1}}}\E\bbrace*{\sgn\parr*{\widetilde{X}^\ell_{1,2}-\widetilde{X}^\ell_{2,2}}}= 0.
\end{align*}
Moreover, $h^\ell_1$ can be readily calculated as
\begin{align*}
h^\ell_1(x) = \E_{H_0}\bbrace*{\sgn\parr*{x_1-\widetilde{X}^\ell_{1,1}}\sgn\parr*{x_2-\widetilde{X}^\ell_{1,2}}} = \bbrace*{2F^\ell_1(x_1)-1}\bbrace*{2F^\ell_2(x_2)-1}.
\end{align*}
Therefore, it holds that
\begin{align*}
\sigma_\ell^2 &= \var\bbrace*{h^\ell_1(X_1)} + 2\sum_{i>1}\cov\bbrace*{h^\ell_1(X_1),h^\ell_1(X_i)}\\
&= \frac{1}{9} + 2\sum_{i>1}\E\bbrace*{\parr*{2F^\ell_1(X_{1,1})-1}\parr*{2F^\ell_1(X_{i,1})}}\E\bbrace*{\parr*{2F^\ell_2(X_{1,2})-1}\parr*{2F^\ell_2(X_{i,2})}}\\
&= \frac{1}{9}+2\sum_{i>1}\fence*{4\E\bbrace*{F^\ell_1(X_{1,1})F^\ell_1(X_{i,1})}-1}\fence*{4\E\bbrace*{F^\ell_2(X_{1,2})F^\ell_2(X_{i,2})}-1},
\end{align*}
where the second summand can be readily estimated given the model \eqref{eq:AR}. 

The following result guarantees the validity of the test \eqref{eq:reject_region} in this special case. The Cram\'er-type moderate deviation in Corollary \ref{cor:MDP_U} plays a central role in its proof. For each $\ell\in[p]$, we will use the notation $\{\widetilde{X}^\ell_{i}\}_{i=1}^n = \{(\widetilde{X}^\ell_{i,1}, \widetilde{X}^\ell_{i,2})\}_{i=1}^n$ to represent $n$ i.i.d. copies of $X^\ell_1$.

\begin{proposition}
\label{prop:test_example}
Suppose that for each $\ell\in[p]$, $X^\ell_{1,1}$ and $X^\ell_{1,2}$ both have finite $\eta_1$th moment for some $\eta_1 > 0$. Moreover, suppose that for each $\ell\in[p]$, $X^\ell_1$ is absolutely continuous, and the density of $X^\ell_{1,k}$, $X^\ell_{i,k} - X^\ell_{j,k}$, $\widetilde{X}^\ell_{i,k} - \widetilde{X}^\ell_{j,k}$ uniformly over $1\leq i \neq j \leq n$ and $k\in[2]$ is upper bounded by some $D = D(n) > 0$ such that $D = O(n^{\eta_2})$ for some $\eta_2 > 0$. Then, the test based on the rejection region \eqref{eq:reject_region} has asymptotic size $\alpha$ for $p = O(n^{\gamma^2/2})$ with arbitrarily large $\gamma$.
\end{proposition}

\section{Discussion}
\label{sec:discussion}




In this section we discuss extensions of our main results to $\tau$-mixing sequence. Recall the following definition of the $\tau$-coefficient introduced in \cite{dedecker2004coupling}. Given a random variable $X$ that takes value in a metric space $(\cal{X}, d)$ and a $\sigma$-algebra $\cal{M}$, denote $\Lambda_1(\cal{X})$ as the family of $1$-Lipschitz functions from $\cal{X}$ to $\RR$ with respect to the metric $d$. When $X$ is integrable, the function
\begin{align*}
W\parr*{\P_{X\mid\cal{M}}} \define \sup_{f\in\Lambda_1(\cal{X})}\bbrace*{\abs*{\int f(x)\P_{X\mid\cal{M}}(dx) - \int f(x)\P_{X}(dx)}}
\end{align*}
can be shown to be $\cal{M}$-measurable. The $\tau$-coefficient is defined as 
\begin{align*}
\tau(\cal{M},X;d) \define \|W\parr*{\P_{X\mid\cal{M}}}\|_1.
\end{align*}
For a stationary sequence of $\cal{X}$-valued random variables $\{X_i\}_{i\in\mathbb{Z}}$ and $\sigma$-algebras $\{\cal{M}_i\}_{i\in\mathbb{Z}}$ (usually taken to be $\sigma\parr*{\{X_j,j\leq i\}}$), $\Lambda_1(\cal{X}^k)$ is defined as the family of $1$-Lipschitz functions on the product space $\cal{X}^k$ equipped with the metric $d_{1,k}(x,y)\define \sum_{i=1}^k d(x_i,y_i)$ for $x = (x_1,\ldots,x_k)$ and $y = (y_1,\ldots,y_k)$ in $\cal{X}^k$. The $\tau$-mixing coefficient is then defined as
\begin{align*}
\tau_m(i) \define \max_{1\leq k\leq m}\frac{1}{k}\sup\bbrace*{\tau(\cal{M}_0, (X_{j_1},\ldots,X_{j_k});d_{1,k})},
\end{align*}
where the supremum is taken over all $(j_1,\ldots,j_k)$ such that $1\leq i\leq j_1<\ldots<j_k$, and
\begin{align*}
\tau(i)  \define \sup_{m\geq 1}\tau_m(i) .
\end{align*}
The sequence is said to be $\tau$-mixing if $\tau(i)\rightarrow 0$ as $i\rightarrow\infty$. Lemma \ref{lemma:mixing_compare} in the supplement shows that, under quite mild moment conditions, $\alpha$-mixing implies $\tau$-mixing. See \cite{dedecker2007weak} for examples of time series models that are $\tau$-mixing. In this section, we make the following model assumption:
\begin{itemize}
\item[{\bf (M$^{\prime}$)}] $\{X_i\}_{i=1}^n$ in \eqref{eq:U} is part of a stationary sequence $\{X_i\}_{i\in\Z}$ in $\RR^d$, which is assumed to be geometrically $\tau$-mixing with coefficient
\begin{align*}
\tau(i) \leq \gamma_1\exp(-\gamma_2 i)~~~\text{for all }i\geq 1,
\end{align*}
for some positive constants $\gamma_1,\gamma_2$.  
\end{itemize}
The first result extends the Bernstein-type inequality in Proposition \ref{prop:bern_U_alpha_general}.

\begin{propbisbis}{prop:bern_U_alpha}
\label{prop:bern_U_tau}
Suppose that $n\geq 2$, Condition ($M^{\prime}$) holds, and for any $t > 0$, $f$ in \eqref{eq:U} satisfies (A) with some set $\cal{C}$. Moreover, suppose that the bases $\{e_j(\cdot)\}_{j=1}^K$ in the expansion \eqref{eq:expansion} satisfy the following Condition (A$^\prime$) with set $\cal{C}$:
\begin{enumerate}
\item[{\bf(A$^{\prime}$)}] there exists a positive constant $L = L(t,\cal{C})$ such that 
\begin{align*}
\abs*{e_j(x) - e_j(y)} \leq L\sum_{\ell=1}^d\abs*{x_\ell-y_\ell}
\end{align*}
for all $x,y\in\RR^d$ and all $j\in[K]$.
\end{enumerate}
Then, the tail bound in Proposition \ref{prop:bern_U_alpha_general} holds with the same constants $\{A_p\}_{p=1}^m$, $\{M_p\}_{p=1}^m$, except now $\sigma^2$ in the definition of $\{A_p\}_{p=1}^m$ takes the value
\begin{align*}
\sigma^2 \define \frac{12(\gamma_1 L)^{\frac{\delta}{1+\delta}}}{1-\exp\bbrace*{-\gamma_2 \delta/(1+\delta)}}\mu_{2+\delta}^{\frac{2+\delta}{1+\delta}}. 
\end{align*}

\end{propbisbis}
Proposition \ref{prop:bern_U_tau} is almost identical to its $\alpha$-mixing version in Proposition \ref{prop:bern_U_alpha_general}, except that a slightly different upper bound on the variance is now used for the $\tau$-mixing sequence $\{X_i\}_{i=1}^n$. 
Next, we develop an analogous version of Theorem \ref{thm:random_fourier_alpha}, which helps verify the conditions in Proposition \ref{prop:bern_U_tau}.
%

\begin{thmbis}{thm:random_fourier_alpha}[Smooth kernels]
\label{thm:random_fourier_tau}
For any given $M>0$, let $\cal{C} = [-M,M]^{md}$. Suppose that $f$ in \eqref{eq:U} satisfies Condition (B1) in Theorem \ref{thm:random_fourier_alpha} with set $\cal{C}$, $\bar{f}$,  and some $q\geq 1$. Then, for any $t > 0$, 
Conditions (A) and (A$^{\prime}$) 
are satisfied with set $\cal{C}$ and constants
\begin{equation}
\begin{aligned}
\label{eq:tau_constant}
&F = 2^m\nm*{\hat{\bar{f}}}_{L^1}, \quad B = 1, \quad \mu_a = 1, \\
&L = \Big[\frac{\mu_q^q\parr*{\hat{\bar{f}}}md\nm*{\hat{\bar{f}}}_{L^1}^2}{t^2}\log\Big\{\frac{48\pi Mmd\nm*{\hat{\bar{f}}}_{L^1}^{1-1/q}\mu_q\parr*{\hat{\bar{f}}}}{t}\Big\}\Big]^{1/q}
\end{aligned}
\end{equation}
for any $a\geq 1$. 
\end{thmbis}

%
Note that in Theorem \ref{thm:random_fourier_tau}, the constant $L$ depends polynomially on $1/t$ under Condition \eqref{eq:fourier_moment} in (B1). This dependence can be improved to $\log(1/t)$ if the Fourier transform of $\bar{f}$ has exponential moment, that is,
\begin{align*}
\int_{\RR^{md}} \abs*{\hat{\bar{f}}(u)}\exp\parr*{\lambda_0 \|u\|}du < \infty
\end{align*}
for some positive $\lambda_0$. This exponential moment condition is satisfied, for example, by the Gaussian kernel and the Cauchy kernel discussed after Corollary \ref{cor:random_fourier_alpha}. Starting from Theorem \ref{thm:random_fourier_tau}, we can readily extend Theorems \ref{thm:Lipschitz_alpha} and \ref{thm:discontinuous_alpha} to the $\tau$-mixing case with the same smoothing technique. We omit them here for the purpose of brevity. \\

\bibliography{reference}

\newpage{}

\appendix

\section{Proofs of results in Section \ref{sec:main}}
\label{proof:main}

We will use the following extra notation. For any real-valued function $f$ on $\RR^d$, $\nabla_xf$ is the gradient of $f$. For a set $A$, $|A|$ indicates its cardinality. For two sequences $\{a_n\}$ and $\{b_n\}$, $a_n = \Omega(b_n)$ if there exists some positive constant $C$ such that $b_n \leq Ca_n$ for all $n\geq 1$.

\subsection{Proof of Proposition \ref{prop:bern_U_alpha}}
\begin{proof}
Throughout the proof, $C_i$'s are positive constants, and we will use the shorthand $f_{j_{a:b}}$ for $f_{j_a,\ldots,j_b}$ for positive integers $a < b$. 

\textbf{Step I:} When $f$ is degenerate of level $r - 1$, its Hoeffding decomposition takes the form
\begin{align*}
f(x_1,\ldots,x_m)  - \theta = \sum_{1\leq i_1<\ldots <i_r\leq m}f_r\parr*{x_{i_1},\ldots,x_{i_r}} + \ldots + f_m(x_1,\ldots,x_m),
\end{align*}
where $\{f_p\}_{p=r}^m$ are defined in \eqref{eq:Hoeffding_decomposition} in the main paper. Given any $t> 0$, let $\widetilde{f}$ be the approximating kernel in Condition (A) such that
\begin{align*}
\abs*{f(x_1,\ldots,x_m) - \widetilde{f}(x_1,\ldots,x_m)} \leq  t
\end{align*}
uniformly over $\supp(\P^m)$. For each $p$, let $f_p$ and $\widetilde{f}_p$ be the $p$th term in the Hoeffding decomposition of $f$ and $\widetilde{f}$, respectively. Then, by definition, there exists some positive constant $C_1 = C_1(m)$ such that
\begin{align*}
\abs*{f_p(x_1,\ldots,x_p) - \widetilde{f}_p(x_1,\ldots,x_p)} \leq C_1t
\end{align*}
uniformly over $\supp(\P^p)$. For each $p\in[m]$, define the V-statistics
\begin{align*}
V_{n,p} \define \sum_{i_1,\ldots,i_p=1}^n f_p(X_{i_1},\ldots,X_{i_p}), \quad \widetilde{V}_{n,p} \define \sum_{i_1,\ldots,i_p=1}^n \widetilde{f}_p(X_{i_1},\ldots,X_{i_p}).
\end{align*}
Then, it holds that
\begin{align*}
\abs*{V_{n,p} - \widetilde{V}_{n,p}} \leq C_1n^pt
\end{align*}
almost surely with respect to the joint distribution of $\{X_i\}_{i=1}^n$. Then, for any $x > 0$, there exists large enough $C_2$ such that
\begin{align*}
\P\bbrace*{\abs*{V_n - \theta}\geq C_2(x+C_1t)} \leq\sum_{p=r}^m \P\bbrace*{n^{-p}\abs*{V_{n,p}}\geq (x+C_1t)} \leq \sum_{p=r}^m\P\parr*{n^{-p}\abs*{\widetilde{V}_{n,p}}\geq x}.\end{align*}

\textbf{Step II:} We now upper bound each summand in the last inequality in Step I. 
For the set of bases $\{e_j(\cdot)\}_{j=1}^K$ in the expansion of $\widetilde{f}$, define $\widetilde{e}_j\define e_j - \E\bbrace*{e_j(X_1)}$ for $j\in[K]$. Since $\widetilde{f}$ is symmetric, for any $(x_1^\top,\ldots,x_m^\top)^\top\in\RR^{md}$, $\widetilde{f}(x_1,\ldots,x_m) = \widetilde{f}(\pi(x_1),\ldots,\pi(x_m))$ for any permutation $\pi$ of $\{x_1,\ldots,x_m\}$. By the definition of $\{\widetilde{f}_p\}_{p=1}^m$ in \eqref{eq:Hoeffding_decomposition} in the main paper, one can readily check that under expansion \eqref{eq:expansion} in the main paper, it holds that
\begin{align*}
\widetilde{f}_p(x_1,\ldots,x_p) = \sum_{j_1,\ldots,j_m=1}^K f_{j_1,\ldots,j_m} \E\parr*{e_{j_1}}\ldots\E\parr*{e_{j_{m-p}}}\widetilde{e}_{j_{m-p+1}}(x_1)\ldots\widetilde{e}_{j_{m}}(x_p),
\end{align*}
for $r\leq p\leq m$. Thus, we have
\begin{align*}
\widetilde{V}_{n,p} = \sum_{j_1,\ldots,j_m=1}^K \fm \E\parr*{e_{j_1}}\ldots\E\parr*{e_{j_{m-p}}}\bbrace*{\sum_{i=1}^n \widetilde{e}_{j_{m-p+1}}(X_i)}\ldots\bbrace*{\sum_{i=1}^n \widetilde{e}_{j_{m}}(X_i)}.
\end{align*}
Define for each $j\in[K]$, $S_{n,j} \define \sum_{i=1}^n \widetilde{e}_{j}(X_i)$ and $\psi_{S_{n,j}}(\lambda) \define \log\fence*{\E\bbrace*{\text{exp}\parr*{\lambda S_{n,j}}}}$. We now control each even order moment of $\widetilde{V}_{n,p}$. To this end, we first note that for each $j\in[K]$, $\{\widetilde{e}_{j}(X_i)\}_{i=1}^n$ is also geometrically $\alpha$-mixing. Thus by Lemma \ref{lemma:sample_mean_Bern}, we have
\begin{align*}
\psi_{S_{n,j}}(\lambda) \leq \frac{\lambda^2\nu_{j}/2}{1-c\lambda} \leq \frac{\lambda^2\nu/2}{1-c\lambda}
\end{align*}
where 
\begin{align*}
\nu_{j} = C_3(n\sigma^2_{j}+B^2), \quad \nu = C_3(n\sigma^2+B^2), \quad c = C_4B(\log n)^2
\end{align*}
and $\sigma^2 = \sup_{j\in[K]}\sigma_j^2$, with
\begin{align*}
\sigma_j^2 \define \var\bbrace*{\widetilde{e}_j(X_1)} + 2\sum_{i>1}\abs*{\cov\bbrace*{\widetilde{e}_j\parr*{X_1},\widetilde{e}_j\parr*{X_i}}}.
\end{align*}
Therefore, Lemma \ref{lemma:bern_moment} implies that for any positive integer $N$, it holds that
\begin{align*}
\E\parr*{S_{n,j}^{2pN}} \leq (pN)!(8\nu)^{pN} + (2pN)!(4c)^{2pN}
\end{align*}
for $j\in[K]$. Therefore, employing a similar argument as in \cite{borisov2015note}, $\E\parr*{\widetilde{V}_{n,p}^{2N}}$ equals
\begin{align*}
&\ms\sum_{j_1,\ldots,j_{2mN}=1}^K f_{j_{1:m}}\ldots f_{j_{(2N-1)m+1:2mN}}\E\parr*{e_{j_1}}\ldots\E\parr*{e_{j_{m-p}}}\ldots\E\parr*{e_{j_{(2N-1)m+1}}}\ldots\E\parr*{e_{j_{2Nm - p}}}\cdot\\
&\ms \E\parr*{S_{n,j_{m-p+1}}\ldots S_{n,j_{m}}\ldots S_{n,j_{2Nm-p+1}}\ldots S_{n,j_{2Nm}}}\\
&\leq \mu_1^{2N(m-p)}\sum_{j_1,\ldots,j_{2mN}=1}^K \abs*{f_{j_{1:m}}}\ldots \abs*{f_{j_{(2N-1)m+1:2mN}}} \E\parr*{\abs*{S_{n,j_{m+1-p}}}\ldots \abs*{S_{n,j_{2mN}}}}\\
&\leq \mu_1^{2N(m-p)}\sum_{j_1,\ldots,j_{2mN}=1}^K \abs*{f_{j_{1:m}}}\ldots\abs*{f_{j_{(2N-1)m+1:2mN}}}\bbrace*{\E\parr*{S_{n,j_{m+1-p}}^{2pN}}}^{\frac{1}{2pN}}\ldots\bbrace*{\E\parr*{S_{n,j_{2mN}}^{2pN}}}^{\frac{1}{2pN}}\\
&\leq \mu_1^{2N(m-p)}F^{2N}\bbrace*{(pN)!(8\nu)^{pN} + (2pN)!(4c)^{2pN}},
\end{align*}
where in the third line we use the generalized H\"older's inequality. By Stirling's approximation formula $\sqrt{2\pi}n^{n+1/2}e^{-n} \leq n! \leq en^{n+1/2}e^{-n}$, it holds that
\begin{align*}
\bbrace*{(pN)!}^{1/p} &\leq e^{1/p} (pN)^{N+1/2p} e^{-N} \leq C_5^N N^{N+1/2} e^{-N}\leq C_6^NN!.
\end{align*}
Similarly, we have $\bbrace*{(2pN)!}^{1/p}\leq C_6^{2N}(2N)!$. Thus we have
\begin{align*}
\E\bbrace*{\abs*{\widetilde{V}_{n,p}}^{\frac{2N}{p}}} \leq \fence*{\E\bbrace*{\parr*{\widetilde{V}_{n,p}}^{2N}}}^{\frac{1}{p}} \leq \mu_1^{\frac{2N(m-p)}{p}}F^{\frac{2N}{p}}\bbrace*{C_7^NN!\nu^N+C_8^{2N}(2N)!c^{2N}}.
\end{align*}
Now we control the Laplace transform of $\abs*{\widetilde{V}_{n,p}}^{1/p}$, 
\begin{align}
\label{eq:bern_moment}
&\ms\E\parr*{e^{\lambda \abs*{\widetilde{V}_{n,p}}^{1/p}}}\notag\\
&= \sum_{N=0}^\infty \frac{\lambda^N}{N!}\E\bbrace*{\parr*{\abs*{\widetilde{V}_{n,p}}^{1/p}}^{N}}\notag\\
&\leq 3\sum_{N=0}^\infty \frac{\lambda^{2N}}{(2N)!}\E\bbrace*{\parr*{\abs*{\widetilde{V}_{n,p}}^{1/p}}^{2N}}\notag\\
&\leq 3\bbrace*{\sum_{N=0}^\infty \frac{\lambda^{2N}}{(2N)!}C_7^NN!\mu_1^{\frac{2N(m-p)}{p}}F^{\frac{2N}{p}}\nu^N + \sum_{N=0}^\infty \lambda^{2N}C_8^{2N}\mu_1^{\frac{2N(m-p)}{p}}F^{\frac{2N}{p}}c^{2N}},
\end{align}
where in the second line we use only the even moments with an absolute constant 3. For the first summand in \eqref{eq:bern_moment}, we have
\begin{align*}
\sum_{N=0}^\infty \lambda^{2N}\frac{N!}{(2N)!}C_7^N\mu_1^{2N(m-p)/p}F^{2N/p}\nu^N &\leq \sum_{N=0}^\infty \frac{\lambda^{2N}}{N!}2^{-N}C_{9}^N\mu_1^{2N(m-p)/p}F^{2N/p}\nu^N\\
&= \text{exp}\bbrace*{C_{10}\lambda^2\mu_1^{2(m-p)/p}F^{2/p}\nu},
\end{align*}
where in the first line we use the relation $N!/(2N)!\leq 2^{-N}/N!$. For the second summand, we have
\begin{align*}
\sum_{N=0}^\infty \lambda^{2N}C_8^{2N}\mu_1^{2N(m-p)/p}F^{2N/p}c^{2N} &= 1 + \frac{\lambda^2C_8^2\mu_1^{2(m-p)/p}F^{2/p}c^2}{1-\lambda^2C_8^2\mu_1^{2(m-p)/p}F^{2/p}c^2}\\
&\leq 1 + \frac{\lambda^2C_8^2\mu_1^{2(m-p)/p}F^{2/p}c^2}{1-\lambda C_8\mu_1^{(m-p)/p}F^{1/p}c}
\end{align*}
for $\lambda \leq 1/\bbrace*{C_8\mu_1^{(m-p)/p}F^{1/p}c}$. Now using the relation $e^{x} + 1 + y \leq 2e^{x+y}$ which holds for all positive $x,y$, we have
\begin{align*}
\E\parr*{e^{\lambda \abs*{V_{n,p}}^{1/p}}} &\leq 6\exp\fence*{\frac{\lambda^2C_{11}\mu_1^{\frac{2(m-p)}{p}}F^{\frac{2}{p}}(\nu + c^2)}{2\bbrace*{1-\lambda C_8c\mu_1^{\frac{(m-p)}{p}}F^{\frac{1}{p}}}}} = 6\text{exp}\bbrace*{\frac{\lambda^2C_{11}nA_p^{1/p}}{2\parr*{1-\lambda C_8M_p^{1/p}}}}.
\end{align*}
A standard exponential Chebyshev argument gives that for any $x>0$
\begin{align*}
\P\parr*{\abs*{\widetilde{V}_{n,p}}^{1/p}\geq x}\leq 6\exp\parr*{-\frac{C_{12}x^2}{nA_p^{1/p}+xM_p^{1/p}}},
\end{align*}
or equivalently,
\begin{align*}
\P\parr*{n^{-p}\abs*{\widetilde{V}_{n,p}}\geq x} \leq 6\exp\parr*{-\frac{C_{12}nx^{2/p}}{A_p^{1/p}+x^{1/p}M_p^{1/p}}}.
\end{align*}
Moreover, we have by Lemma \ref{lemma:alpha_covariance} that
\begin{align*}
&\ms\var\bbrace*{\widetilde{e}_j(X_1)} + 2\sum_{i>1}\abs*{\cov\bbrace*{\widetilde{e}_j(X_1), \widetilde{e}_j(X_i)}}\\
&\leq 2\sum_{i\geq 1}\abs*{\cov\bbrace*{\widetilde{e}_j(X_1), \widetilde{e}_j(X_i)}}\\
&\leq 2\bbrace*{\sum_{n=0}^\infty 8\alpha^{\delta/(2+\delta)}(n)}\|\widetilde{e}_j(X_1)\|_{2+\delta}\|\widetilde{e}_j(X_1)\|_{2+\delta}\\
&\leq 64\bbrace*{\sum_{n=0}^\infty \alpha^{\delta/(2+\delta)}(n)}\|e_j(X_1)\|_{2+\delta}\|e_j(X_1)\|_{2+\delta}\\
&\leq 64\gamma_1^{\delta/(2+\delta)}\mu_{2+\delta}^2\bbrace*{\sum_{n=0}^\infty \exp\parr*{-\gamma_2\frac{\delta}{2+\delta}n}}\\
&= \frac{64\gamma_1^{\delta/(2+\delta)}}{1-\exp\bbrace*{-\gamma_2\delta/(2+\delta)}}\mu_{2+\delta}^2.
\end{align*}
Putting together the pieces completes the proof. 
\end{proof}

\subsection{Proof of Proposition \ref{prop:MDP_U}}
\label{proof:MDP}
\begin{proof}
The proof is standard once exponential inequalities are established. See, for example, Chapter 8.2 in \cite{korolyuk2013theory}. In detail, by Hoeffding decomposition, 
\begin{align*}
\frac{\sqrt{n}}{m\nu}V_n = \frac{1}{\sqrt{n}\nu}S_n + R_n,
\end{align*}
where $S_n = \sum_{i=1}^n\fence*{f_1(X_i) - \E\bbrace*{f_1(X_i)}}$, $R_n = \frac{\sqrt{n}}{m\nu}{m\choose p}\sum_{p=2}^m V_{n,p}$ and each $V_{n,p}$ is a fully degenerate V-statistic of order $p$. Thus, for any positive $x_n$ and $\varepsilon_n$ that depend on $n$, we have
\begin{align*}
\P\parr*{\frac{\sqrt{n}}{m\nu}V_n\geq x_n} \leq \P\bbrace*{\frac{1}{\sqrt{n}\nu}S_n\geq \parr*{x_n-\varepsilon_n}} + \P\parr*{\abs*{R_n}\geq \varepsilon_n}
\end{align*}
and
\begin{align*}
\P\parr*{\frac{\sqrt{n}}{m\nu}V_n\geq x_n} \geq \P\bbrace*{\frac{1}{\sqrt{n}\nu}S_n \geq \parr*{x_n + \varepsilon_n}} - \P\parr*{\abs*{R_n}\geq \varepsilon_n}.
\end{align*}
Since $\{X_i\}_{i=1}^n$ is geometrically $\alpha$-mixing, $\{f_1(X_i)\}_{i=1}^n$ is also geometrically $\alpha$-mixing. Moreover, letting $p = 2 + \gamma^2 + \eta$, we have
\begin{align*}
&\ms\E\bbrace*{\abs*{f_1(X)}^p}\\
&= \E\parr*{\abs*{\frac{1}{m}\fence*{\E\bbrace*{f\parr*{X,\widetilde{X}_2,\ldots,\widetilde{X}_m} + \ldots + f\parr*{\widetilde{X}_1,\ldots,\widetilde{X}_{m-1},X}\mid \bbrace*{\widetilde{X}_i}_{i=1}^m}^p}}}\\ 
&\leq m^{-1}\parr*{\E\fence*{\abs*{\E\bbrace*{f\parr*{X,\widetilde{X}_2,\ldots,\widetilde{X}_m}}\mid \bbrace*{\widetilde{X}_i}_{i=1}^m}^p} + \ldots + \E\fence*{\abs*{\E\bbrace*{f\parr*{\widetilde{X}_1,\ldots,\widetilde{X}_{m-1},X}}\mid \bbrace*{\widetilde{X}_i}_{i=1}^m}^p}}\\
&< m^{-1}\fence*{\E\bbrace*{\abs*{f\parr*{X,\widetilde{X}_2,\ldots,\widetilde{X}_m}}^p} + \ldots + \E\bbrace*{\abs*{f\parr*{\widetilde{X}_1,\ldots,\widetilde{X}_{m-1},X}}^p}}\\
&= \E\bbrace*{\abs*{f\parr*{\widetilde{X}_1,\ldots,\widetilde{X}_m}}^p}\\
&< \infty.
\end{align*}
Therefore the moment condition in Theorem 1.1 of \cite{babu1978probabilities} is satisfied with the positive constant $\gamma$. Choosing $\varepsilon_n = (\log n)^{-2}$, we have
\begin{align*}
\P\fence*{\frac{1}{\sqrt{n}\nu}S_n \geq \bbrace*{x_n-(\log n)^{-2}}} = \bbrace*{1-\Phi\parr*{x_n-(\log n)^{-2}}}\bbrace*{1 + o(1)}
\end{align*}
for $0\leq x_n\leq \gamma\sqrt{\log n}$. Moreover, by the property of normal distribution, it holds that
\begin{align*}
1-\Phi\parr*{x_n-(\log n)^{-2}} = \bbrace*{1-\Phi(x_n)}\bbrace*{1+o((\log n)^{-1})}
\end{align*}
uniformly for $0\leq x_n \leq \gamma\sqrt{\log n}$. Therefore, it holds that
\begin{align*}
\P\bbrace*{\frac{1}{\sqrt{n}\nu}S_n \geq \parr*{x_n-\varepsilon_n}} = \bbrace*{1-\Phi(x_n)}\bbrace*{1 + o(1)}.
\end{align*}
By a similar calculation of $\P\bbrace*{S_n/(\sqrt{n}\nu)\geq \parr*{x_n+\varepsilon_n}}$, it remains to show that $\P\parr*{|R_n|\geq \varepsilon_n} = \bbrace*{1-\Phi(x_n)}o(1)$ uniformly over $x_n\in[0,\gamma\sqrt{\log n}]$. To this end, applying Proposition \ref{prop:bern_U_alpha} with $r = 2$ and $t\asymp x\asymp 1/\bbrace*{\sqrt{n}(\log n)^2}$, we obtain
\begin{align*}
\P\parr*{|R_n|\geq\varepsilon_n}\leq C_1\sum_{p=2}^m \exp\bbrace*{-\frac{C_2nt^{2/p}}{A_p^{1/p}+t^{1/p}M_p^{1/p}}}.
\end{align*}
Note that, in the first summand, the leading term is the one with $p = 2$. Therefore, using the relation that for any $x > 0$, 
\begin{align*}
\frac{1}{x+1/x} \frac{1}{\sqrt{2\pi}}\exp(-x^2/2)\leq 1-\Phi(x) \leq \frac{1}{x} \frac{1}{\sqrt{2\pi}}\text{exp}(-x^2/2),
\end{align*}
it suffices to show that 
\begin{align*}
\log n = o\bbrace*{\frac{n^2t}{n\mu_1^{m-2}F\sigma^2} \wedge \frac{n^2t}{\mu_1^{m-2}B^2F(\log n)^4} \wedge \frac{n^2t}{nt^{1/2}\mu_1^{(m-2)/2}BF^{1/2}(\log n)^2}},
\end{align*}
which holds under the condition
\begin{align*}
\mu_1^{m-2}F\sigma^2 = o\bbrace*{n^{1/2}(\log n)^{-3}} \text{ and } \mu_1^{m-2}B^2F = o\bbrace*{n^{3/2}(\log n)^{-8}}
\end{align*}
with $\sigma^2$ given in Proposition \ref{prop:bern_U_alpha}. This completes the proof.
\end{proof}


\subsection{Proof of Theorem \ref{thm:random_fourier_alpha}}
\begin{proof}
This proof adapts from that of Claim 1 in \cite{rahimi2008random}. 
Throughout the proof, $x_1,\ldots,x_m$ and $u_1,\ldots,u_m$ are real vectors in $\RR^d$, $dx = dx_1\ldots dx_d$ for any $x\in\RR^d$ and $x,u$ will be real vectors in $\RR^{md}$. Let $\hat{\bar{f}}:\RR^{md}\rightarrow\mathbb{C}$ be the Fourier transform of $\bar{f}$, that is,
\begin{align*}
\hat{\bar{f}}(u_1,\ldots,u_m) = \int_{\RR^{md}} \bar{f}(x_1,\ldots,x_m)e^{-2\pi i(u_1^\top x_1+\ldots+u_m^\top x_m)}dx_1\ldots dx_m.
\end{align*}
Clearly, Condition \eqref{eq:fourier_moment} in the main paper implies that $\hat{\bar{f}}\in L^1\parr*{\RR^{md}}$. Since $\bar{f}$ is continuous, by the Fourier inversion formula (see, for example, Chapter 6 of \cite{stein2011fourier}), we have
\begin{align*}
\bar{f}(x_1,\ldots,x_m) = \int_{\RR^d} \hat{\bar{f}}(u_1,\ldots,u_m)e^{2\pi i(u_1^\top x_1+\ldots+u_m^\top x_m)}du_1\ldots du_m.
\end{align*}
Note that without continuity of $\bar{f}$, the above equation only holds almost surely with respect to the Lebesgue measure. Let $\hat{\bar{f}} = \hat{\bar{g}} + i\hat{\bar{h}}$ for real-valued functions $\hat{\bar{g}},\hat{\bar{h}}$, then since $\bar{f}$ is real-valued, we have $\bar{f} = I - II$, where
\begin{align*}
I &\define \int_{\RR^{md}} \hat{\bar{g}}(u_1,\ldots,u_m)\cos\bbrace*{2\pi\parr*{u_1^\top x_1+\ldots + u_m^\top x_m}}du_1\ldots du_m,\\
II &\define \int_{\RR^{md}} \hat{\bar{h}}(u_1,\ldots,u_m)\sin\bbrace*{2\pi\parr*{u_1^\top x_1+\ldots+u_m^\top x_m}}du_1\ldots du_m.
\end{align*}
We now approximate $I$ and $II$ separately. $I$ can be further written as $I = I_+ - I_-$, where
\begin{align*}
I_+ &\define \int_{[\hat{\bar{g}}>0]}\hat{\bar{g}}(u_1,\ldots,u_m)\cos\bbrace*{2\pi\parr*{u_1^\top x_1+\ldots + u_m^\top x_m}}du_1\ldots du_m,\\
I_- &\define \int_{[\hat{\bar{g}}<0]}-\hat{\bar{g}}(u_1,\ldots,u_m)\cos\bbrace*{2\pi\parr*{u_1^\top x_1+\ldots + u_m^\top x_m}}du_1\ldots du_m.
\end{align*}
Let $A^+_g \define \int_{[\hat{\bar{g}} > 0]}\hat{\bar{g}}(u)du$ and $A^-_g \define \int_{[\hat{\bar{g}} <0]}\parr*{-\hat{\bar{g}}(u)}du$, then it can be verified that $A^+_g$ and $A^-_g$ are both nonnegative and satisfy $A^+_g + A^-_g = \|\hat{\bar{g}}\|_{L^1} < \infty$ and $A^+_g - A^-_g = f(0)$, where we use the fact that $\hat{\bar{g}}\in L^1\parr*{\RR^{md}}$ since $\hat{\bar{f}}\in L^1\parr*{\RR^{md}}$. Then, we have
\begin{align*}
I &= A^+_g\cdot\E_u\fence*{\cos\bbrace*{2\pi\parr*{u_1^\top x_1 + \ldots u_m^\top x_m}}} - A^-_g\cdot\E_v\fence*{\cos\bbrace*{2\pi\parr*{v_1^\top x_1 + \ldots v_m^\top x_m}}}\\
&\define A^+_g\cdot k^+_g(x_1,\ldots,x_m) - A^-_g\cdot k^-_g(x_1,\ldots,x_m),
\end{align*}
where $(u_1^\top,\ldots,u_m^\top)^\top$ follows the distribution $\hat{\bar{g}}\mathbbm{1}\bbrace*{\hat{\bar{g}} > 0}/A^+_g$, and $(v_1^\top,\ldots,v_m^\top)^\top$ follow the distribution $-\hat{\bar{g}}\mathbbm{1}\bbrace*{\hat{\bar{g}}<0}/A^-_g$. Assume without loss of generality that $A^+_g > 0$ and $A^-_g > 0$. We now focus on $I_+$. For any $\cal{M} \subset\RR^{md}$, let $\diam(\cal{M})$ be the diameter of $\cal{M}$. Then, there exist $T$ Euclidean balls with radius $r$ that cover $\cal{M}$, where $T \leq \bbrace*{c\diam(\cal{M})/r}^{md}$ with $c = 3\sqrt{md/\pi}$. Denote $\{d_1,\ldots, d_T\}$ as the centers of these balls in $\RR^{md}$. Now choose an i.i.d. sample $\{(u_{i1}^\top,\ldots,u_{im}^\top)^\top\}_{i=1}^{D_1}$ from the distribution $\hat{\bar{g}}\mathbbm{1}\bbrace*{\hat{\bar{g}} > 0}/A^+_g$ with the sample size $D_1$ to be specified later. Then, for each center $d=(d_1^\top,\ldots,d_m^\top)^\top$ and any $t > 0$, it holds by Hoeffding's inequality that
\begin{align*}
\P\bbrace*{\abs*{\frac{1}{D_1}\sum_{i=1}^{D_1}\cos\bbrace*{2\pi\parr*{u_{i1}^\top d_1+\ldots +u_{im}^\top d_m}} - k^+_g(d_1,\ldots,d_m)}\geq \frac{t}{8}}\leq \exp\parr*{-\frac{D_1t^2}{128}}.
\end{align*}
Let $s_{D_1}(x_1,\ldots,x_m) \define \sum_{i=1}^{D_1}\cos\bbrace*{2\pi\parr*{u_{i1}^\top x_1+\ldots +u_{im}^\top x_m}}/D_1$ so that $k^+_g(x_1,\ldots,x_m) = \E_u\bbrace*{s_{D_1}(x_1,\ldots,x_m)}$. Then, for any $q\geq 1$, it holds that
\begin{equation}
\begin{aligned}
\label{eq:lip_constant}
\E\fence*{\sup_{x}\|\nabla_x \bbrace*{s_{D_1}(x)-k_g^+(x)}\|^q} &= \E\fence*{\sup_x\|\nabla_x s_{D_1}(x)-\E\nabla_xs_{D_1}(x)\|^q}\\
&\leq \E\fence*{\sup_x\bbrace*{\|\nabla_x s_{D_1}(x)\| + \E\parr*{\|\nabla_x s_{D_1}(x)\|}}^q}\\
&\leq 2^{q-1}\E\fence*{\sup_x\|\nabla_xs_{D_1}(x)\|^q + \sup_x\bbrace*{\E\parr*{\|\nabla_xs_{D_1}(x)\|}}^q}\\
&\leq 2^q \E\parr*{\sup_x\|\nabla_xs_{D_1}(x)\|^q},
\end{aligned}
\end{equation}
where in the first line we use the finiteness of $\int_{\RR^{md}}\abs*{\hat{\bar{f}}(u)}\|u\|du$ (guaranteed by Condition \eqref{eq:fourier_moment} in the main paper) and dominated convergence theorem to exchange the derivative with expectation. Moreover, 
\begin{align*}
\E\parr*{\sup_x\|\nabla_xs_{D_1}(x)\|^q} &= \E\fence*{\sup_x\nm*{\frac{1}{D_1}\sum_{i=1}^{D_1}2\pi u_i\cos\bbrace*{2\pi\parr*{u_i^\top x}}}^q}\\
&\leq (2\pi)^q\E\parr*{\nm*{\frac{1}{D_1}\sum_{i=1}^{D_1}u_i}^q}\\
&\leq (2\pi)^q \E\bbrace*{\parr*{\frac{1}{D_1}\sum_{i=1}^{D_1}\nm*{u_i}}^q}\\
&\leq (2\pi)^q \E\bbrace*{\frac{1}{D_1}\sum_{i=1}^{D_1}\nm*{u_i}^q}\\
&= (2\pi)^q\E\parr*{\nm*{u_1}^q},
\end{align*}
where we have used the finiteness of $\E\parr*{\|u_1\|^q}$ since $\int_{\RR^{md}}\abs*{\hat{\bar{f}}(u)}\|u\|^qdu < \infty$. Therefore, it holds that
\begin{align*}
\E\bbrace*{\sup_x\|\nabla_x \parr*{s_{D_1}(x) - k^+_g(x)}\|^q} \leq (4\pi)^q\E\parr*{\|u_1\|^q} 
\end{align*}
and thus by Markov's inequality,
\begin{align*}
\P\parr*{\sup_x\|\nabla_x\bbrace*{s_{D_1}(x) - k^+_g(x)}\|\geq \frac{t}{8r}} \leq \parr*{\frac{32\pi r}{t}}^q\E\parr*{\|u_1\|^q}. 
\end{align*}
By triangular inequality, the event $\bbrace*{\sup_{x\in \cal{M}}\abs*{s_{D_1}(x)-k^+_g(x)}\leq t/4}$ has greater probability than the following event
\begin{align*}
\Big\{\abs*{s_{D_1}(d) - k^+_g(d)}\leq t/8,\forall d\in\{d_1,\ldots,d_T\}\Big\} \bigcap \Big\{\sup_x\|\nabla_x\bbrace*{s_{D_1}(x) - k^+_g(x)}\|\leq t/(8r)\Big\}.
\end{align*}
Therefore, we have
\begin{align*}
\P\bbrace*{\sup_{x\in \cal{M}}\abs*{s_{D_1}(x) -k^+_g(x)}\geq \frac{t}{4}} \leq \parr*{\frac{c\diam(\cal{M})}{r}}^{md}\exp\parr*{-\frac{D_1t^2}{128}} + \parr*{\frac{32\pi r}{t}}^q\E\parr*{\|u_1\|^q}.
\end{align*}
Letting the right side of the above inequality be of the form $\kappa_1r^{-md} + \kappa_2r^q$, and $r = (\kappa_1/\kappa_2)^{1/(q+md)}$, we have
\begin{align*}
\P\bbrace*{\sup_{x\in \cal{M}}\abs*{s_{D_1}(x) -k_g^+(x)}\geq \frac{t}{4}} \leq 2\bbrace*{\frac{32\pi\parr*{\E\|u_1\|^q}^{1/q}c\diam(\cal{M})}{t}}^{\frac{qmd}{q+md}}\exp\parr*{-\frac{D_1t^2}{128}\frac{q}{q+md}}.
\end{align*}
Now, using the fact
\begin{align*}
\E\parr*{\|u_1\|^q} = \int_{\RR^{md}} \|u\|^q\frac{\hat{\bar{g}}(u)\mathbbm{1}\bbrace*{\hat{\bar{g}}(u)>0}}{A^+_g}du \leq \frac{1}{A^+_g}\int_{\RR^{md}}\|u\|^q\abs*{\hat{\bar{g}}(u)}du \leq \frac{1}{A_g^+}\int_{\RR^{md}}\|u\|^q\abs*{\hat{\bar{f}}(u)}du,
\end{align*}
we conclude that there exists $\{u_i\}_{i=1}^{D_1}\in\RR^{md}$ such that uniformly over $\cal{M}$, it holds that
\begin{align*}
A^+_g\cdot\abs*{s_{D_1}(x) - k_g^+(x)} = \abs*{\frac{A^+_g}{D_1}\sum_{i=1}^{D_1}\cos\bbrace*{2\pi\parr*{u_i^\top x}} - A^+_g\cdot k^+_g(x)} \leq A^+_g\frac{t}{4}
\end{align*}
when $D_1$ is chosen to be larger than
\begin{align*}
D_1 = \Omega\fence*{\frac{md}{t^2}\log\bbrace*{\frac{\pi c\diam(\cal{M})\mu_q\parr*{\hat{\bar{f}}}}{(A^+_g)^{1/q}t}}}.
\end{align*}
Equivalently, it holds that $\abs*{A^+_g\cdot s_{D_1}(x) - A^+_g\cdot k^+_g(x)}\leq t/4$ when $D_1$ is chosen to be larger than
\begin{align*}
D_1 = \Omega\fence*{\frac{md(A^+_g)^2}{t^2}\log\bbrace*{\frac{\pi c\diam(\cal{M})(A^+_g)^{1-1/q}\mu_q\parr*{\hat{\bar{f}}}}{t}}}.
\end{align*}
Similarly, it can be shown that there exists $\{v_i\}_{i=1}^{D_2}\in\RR^{md}$ 
such that $\abs*{A^-_g\cdot s_{D_2}(x)-A^-_g\cdot k^-_g(x)}\leq t/4$ uniformly over $x\in\cal{M}$, where
\begin{align*}
s_{D_2}(x) = \frac{1}{D_2}\sum_{i=1}^{D_2}\cos\bbrace*{2\pi\parr*{v_i^\top x}},
\end{align*}
and $D_2$ is chosen to be larger than
\begin{align*}
D_2 = \Omega\fence*{\frac{md(A^-_g)^2}{t^2}\log\bbrace*{\frac{8\pi c\diam(\cal{M})(A^-_g)^{1-1/q}\mu_q\parr*{\hat{\bar{f}}}}{t}}}.
\end{align*}
Repeating this procedure for the approximation of $II$, then with $A^+_h,A^-_h,k^+_h,k^-_h$ similarly defined as $A^+_g,A^-_g,k^+_h,k^-_h$, we can find $s_{D_3}$ and $s_{D_4}$ which are sample means of sine functions such that $\abs*{A^+_h\cdot s_{D_3}(x) - A^+_h\cdot k^+_h(x)}\leq t/4$ and $\abs*{A^-_h\cdot s_{D_4}(x) - A^-_h\cdot k^-_h(x)}\leq t/4$ uniformly over all $x\in\cal{M}$, when the sample sizes $D_3$ and $D_4$ are respectively chosen to be larger than
\begin{align*}
D_3 &= \Omega\fence*{\frac{md(A^+_h)^2}{t^2}\log\bbrace*{\frac{\pi c\diam(\cal{M})(A^+_h)^{1-1/q}\mu_q\parr*{\hat{\bar{f}}}}{t}}},\\
D_4 &= \Omega\fence*{\frac{md(A^-_h)^2}{t^2}\log\bbrace*{\frac{\pi c\diam(\cal{M})(A^-_h)^{1-1/q}\mu_q\parr*{\hat{\bar{f}}}}{t}}}.
\end{align*}
Putting together the pieces, we obtain that
\begin{align*}
\abs*{s_D(x) - \bar{f}(x)} \define \abs*{\bbrace*{A^+_g\cdot s_{D_1}(x) - A^-_g\cdot s_{D_2}(x) - A^+_h\cdot s_{D_3}(x) + A^-_h\cdot s_{D_4}(x)} - \bar{f}(x)}
\end{align*}
is smaller than $t$ when $D_1$-$D_4$ are chosen as above. Since 
\begin{align*}
A^+_g + A^-_g + A^+_h + A^-_h = \int_{\RR^{md}}\abs*{\hat{\bar{g}}} + \abs*{\hat{\bar{h}}} \leq \sqrt{2}\int_{\RR^{md}}\sqrt{\abs*{\hat{\bar{g}}}^2+\abs*{\hat{\bar{h}}}^2} = \sqrt{2}\nm*{\hat{\bar{f}}}_{L_1}
\end{align*}
and for each $u$, $\cos\bbrace*{2\pi\parr*{u^\top x}}$ can be written as at most $2^{m-1}$ linear combinations of the term $z_{u_1}\parr*{2\pi u_1^\top x_1}\ldots z_{u_m}\parr*{2\pi u_m^\top x_m}$, where $\{z_{u_i}(\cdot)\}_{i=1}^m$ is either the cosine or sine function.\\
 Therefore, with the choice $\cal{M}\define [-M,M]^{md}$, $s_D$ satisfies Condition (A) with constants $F = 2^m\nm*{\hat{\bar{f}}}_{L^1}, B = \mu_a = 1$ for any $a \geq 1$. Moreover, define the symmetrized version of $s_D$ to be
\begin{align*}
\widetilde{s}_D(x_1,\ldots,x_m) \define \frac{1}{m!}\sum_{\pi}s_D(\pi(x_1),\ldots,\pi(x_m)),
\end{align*}
where the summation is taken over all $m!$ permutations of $(x_1,\ldots,x_m)$. Then, due to the symmetry of $\bar{f}$ and $\cal{M}$, $\abs*{\widetilde{s}_D - \bar{f}}\leq t$ uniformly over $\cal{M}$ and $\widetilde{s}_D$ satisfies expansion \eqref{eq:expansion} in the main paper with the same constants as $s_D$. 
\end{proof}

\subsection{Proof of Corollary \ref{cor:random_fourier_m2}}
\begin{proof}
When $m = 2$ and the kernel is shift invariant with $f(x,y) = f_0(x-y)$, if, for any given $M>0$ and $\cal{C}_0 = [-2M,2M]^d$, $f_0$ satisfies Condition (B1) with set $\cal{C}_0$ and $\bar{f}_0$, then for any given $t>0$, the proof of Theorem \ref{thm:random_fourier_alpha} guarantees the existence of $\widetilde{f}_0$ with expansion in the cosine bases $\{\cos\parr*{2\pi u^\top x}\}$ such that $\abs*{\bar{f}_0-\widetilde{f}_0}\leq t$ uniformly over $[-2M,2M]^d$ and $\widetilde{f}_0$ satisfies the expansion \eqref{eq:expansion} in the main paper with constants
\begin{align*}
F = 2\nm*{\hat{\bar{f}}_0}_{L^1}, \quad B = 1, \quad \mu_a = 1, \quad a\geq 1.
\end{align*}
Since $f_0$ and $\bar{f}_0$ coincides on $\cal{C}_0$, it also holds that $\abs*{f_0 - \widetilde{f}_0}\leq t$ uniformly over $\cal{C}_0$. Define $\widetilde{f}(x, y) \define \widetilde{f}_0(x-y)$. Note that for any $(x,y)\in\cal{C} = [-M,M]^{2d}$, $x-y\in[-2M,2M]^d$. Thus, it holds uniformly over $(x,y)\in\cal{C}$ that
\begin{align*}
\abs*{f(x,y) - \widetilde{f}(x,y)} = \abs*{f_0(x-y) - \widetilde{f}_0(x-y)} \leq t.
\end{align*}
Moreover, using the trigonometric identity
\begin{align*}
\cos\bbrace*{2\pi u^\top (x-y)} = \cos\parr*{2\pi u^\top x}\cos\parr*{2\pi u^\top y} + \sin\parr*{2\pi u^\top x}\sin\parr*{2\pi u^\top y},
\end{align*}
we obtain that $\widetilde{f}$ admits an expansion with constants 
\begin{align*}
F = 4\nm*{\hat{\bar{f}}_0}_{L^1}, \quad B = 1, \quad \mu_a = 1
\end{align*}
for all $a\geq 1$ in \eqref{eq:expansion} in the main paper.
\end{proof}

\subsection{Proof of Corollary \ref{cor:schwartz}}
\begin{proof}
We first prove part (a). Using the relation that
\begin{align*}
\|u\|^k \leq \|u\|_1^k \leq (md)^{k-1}\sum_{i=1}^{md}|u_i|^k,
\end{align*}
we have
\begin{align*}
\parr*{1 + \|u\|^k}\abs*{\hat{f}(u)} &\leq \abs*{\hat{f}(u)} + (md)^{k-1}\sum_{i=1}^{md}\abs*{u_i^k\hat{f}(u)}\\
&\leq \|f\|_{L^1} + (md)^{k-1}\sum_{i=1}^{md}\abs*{\widehat{\frac{\partial^k}{\partial x_i^k}f}(u)}\\
&\leq \|f\|_{L^1} + (md)^{k-1}\sum_{i=1}^{md}\nm*{\frac{\partial^k}{\partial x_i^k}f}_{L^1},
\end{align*}
where in the second line $\widehat{\partial^kf/\partial x_i^k}(u)$ should be understood as the Fourier transform of $\partial^kf/\partial x_i^k$ evaluated at $u$. Denote $C\define \|f\|_{L^1} + (md)^{k-1}\sum_{i=1}^{md}\|\partial^kf/\partial x_i^k\|_{L^1}$. Then $\abs*{\hat{f}(u)}\leq C/\parr*{1+\|u\|^k}$. Therefore, using polar coordinates, we have
\begin{align*}
\int_{\RR^{md}}\abs*{\hat{f}(u)}\|u\|du &\leq C\int_{\RR^{md}}\frac{\|u\|}{1+\|u\|^k}du\\
&= CC'\int_0^\infty \frac{r}{1+r^k}\cdot r^{md-1}dr\\
&\leq CC' \parr*{\int_0^1 1dr + \int_1^{\infty} \frac{1}{r^2}dr}\\
&= 2CC',
\end{align*}
where $C'$ is some positive constant that depends on $m$ and $d$. Therefore, for any given $M > 0$ and $\cal{C} = [-M,M]^{md}$, Condition (B1) in Theorem \ref{thm:random_fourier_alpha} is satisfied with set $\cal{C}$, $\bar{f} = f$, and $q = 1$. In particular, when $f\in\cal{S}(\RR^{md})$, $f$ is indefinitely differentiable by definition and is in $L^1(\RR^{md})$. Moreover, (mixed) derivatives of arbitrary order of $f$ are still Schwartz on $\RR^{md}$ and thus in $L^1(\RR^{md})$. Since Schwartz functions dominate polynomials of arbitrary order, $f$ thus satisfies Condition (B1) in Theorem \ref{thm:random_fourier_alpha} with set $\cal{C}$, $\bar{f} = f$, and arbitrary $q\geq 1$. 

Now we prove part (b). First note that $h_\ell \in L^1(\RR^m)$ for $\ell\in[d]$ implies that $f\in L^1(\RR^{md})$. Moreover, using a similar argument as part (a), we have 
\begin{align*}
\abs*{\hat{h}_\ell(u)} \leq \frac{C_\ell}{1+\|u\|^{m+2}}
\end{align*}
whenever $h_\ell\in C^{m+2}(\RR^m)\bigcap L^1(\RR^m)$ and all $\partial^{m+2}h_\ell/\partial x_i^{m+2} \in L^1(\RR^{m})$ for all $\ell\in[d]$ and $i\in[m]$. Apparently $\hat{h}_\ell(u)\in L^1(\RR^m)$ for all $\ell \in [d]$. For any $u\in\RR^{md}$, write $u = (u_1^\top,\ldots,u_d^\top)^\top$ with each $u_i\in\RR^m$. Then, it holds that
\begin{align*}
\int_{\RR^{md}}\abs*{\hat{f}(u)}\|u\|du &\leq \sum_{\ell=1}^d \int_{\RR^{md}} \abs*{\hat{h}_1(u_1)}\ldots\abs*{\hat{h}_d(u_d)}\|u_\ell\|du.
\end{align*}
For each $\ell\in[d]$, using a similar argument as in part (a), it holds that $\int_{\RR^m} \abs*{\hat{h}_\ell(u)}\|u\|du \leq CC_\ell$
for some constant $C$ that only depends on $m$ and $d$. Therefore, it holds that
\begin{align*}
\int_{\RR^{md}}\abs*{\hat{f}(u)}\|u\|du \leq C\parr*{\prod_{\ell=1}^d\nm*{\hat{h}_\ell}_{L^1}}\parr*{\sum_{\ell=1}^d \frac{C_\ell}{\nm*{\hat{h}_\ell}_{L^1}}}.
\end{align*}
The case where $h_\ell \in \cal{S}(\RR^m)$ for all $\ell\in[d]$ follows trivially.
\end{proof}

\subsection{Proof of Corollary \ref{cor:random_fourier_alpha}}
\begin{proof}
When $\bar{f}_0$ is PD, we have by definition that $\bar{f}_0(0) \geq 0$ and for each $x,y\in\RR^d$, $\bar{f}_0(x-y) = \bar{f}_0(y-x)$, therefore $\bar{f}_0(x) = \bar{f}_0(-x)$ for any $x\in\RR^d$. This implies that the Fourier transform $\hat{\bar{f}}_0$ of $\bar{f}_0$ is real-valued, and $\hat{\bar{h}} = 0$ in the proof of Theorem \ref{thm:random_fourier_alpha}. Moreover, since $\bar{f}_0\in L^1(\RR^d)$ as it satisfies Condition (B1), $\bar{f}$ equals the inverse Fourier transform of $\hat{\bar{f}}$. Thus, by Lemma \ref{lemma:bochner}, $\hat{\bar{f}}_0$ is nonnegative and we have $\bar{f}_0 = I = I_+$ with $m = 1$ in the proof of Theorem \ref{thm:random_fourier_alpha}. By definition, we have
\begin{align*}
\bar{f}_0(x) = \int_{\RR^d} \hat{\bar{f}}_0(u)e^{2\pi ix^\top u}du = \int_{\RR^d} \abs*{\hat{\bar{f}}_0(u)}e^{2\pi ix^\top u}du.
\end{align*}
Letting $x = 0$ in the above equation, we obtain $\nm*{\hat{\bar{f}}_0}_{L^1} = \bar{f}_0(0)$. 

Now consider the case where $\hat{\bar{f}}_0$ only has fractional moment. Let $\widetilde{\bar{f}}_0 \define \bar{f}_0/\bar{f}_0(0)$ and denote the Lipschitz constants of $\widetilde{\bar{f}}_0$ and $\bar{f}_0$ as $L_{\widetilde{\bar{f}}_0}$ and $L_{\bar{f}_0}$, respectively. Then, $L_{\widetilde{\bar{f}}_0} = L_{\bar{f}_0}/\bar{f}_0(0)$. Now we proceed with the proof of Theorem \ref{thm:random_fourier_alpha} until \eqref{eq:lip_constant}, and replace it with
\begin{align*}
\E\bbrace*{\sup_x \nm*{\nabla_x s_D(x) - \widetilde{\bar{f}}_0(x)}^q} &\leq \E\bbrace*{\sup_x \parr*{\nm*{\nabla_x s_D(x)}^q + \nm*{\nabla \widetilde{\bar{f}}_0(x)}^q}}\\
&\leq \E\bbrace*{\sup_x\nm*{\nabla_x s_D(x)}^q} + \sup_x\nm*{\nabla_x\widetilde{\bar{f}}_0(x)}^q\\
&\leq \E\bbrace*{\sup_x\nm*{\nabla_x s_D(x)}^q} + L_{\widetilde{\bar{f}}_0}^q,
\end{align*}
where $s_D(x) = \sum_{i=1}^D \cos\parr*{2\pi u_i^\top x}/D$ (here we use the notation $s_{D}$ instead of $s_{D_1}$ since in the PD case we only need to approximate the term $I_+$ as argued in part (a)). Note that the original \eqref{eq:lip_constant} in the proof of Theorem \ref{thm:random_fourier_alpha} no longer holds as mere fractional moment does not guarantee the exchange of derivative and expectation in its first step. For the first term in the above inequality, we have
\begin{align*}
\E\bbrace*{\sup_x\|\nabla_xs_{D}(x)\|^q} &= \E\fence*{\sup_x\nm*{\frac{1}{D}\sum_{i=1}^{D}2\pi u_i\cos\bbrace*{2\pi\parr*{u_i^\top x}}}^q}\\
&\leq (2\pi)^q\E\parr*{\nm*{\frac{1}{D}\sum_{i=1}^Du_i}^q}\\
&\leq (2\pi)^q \E\bbrace*{\parr*{\frac{1}{D}\sum_{i=1}^{D}\nm*{u_i}}^q}\\
&\leq (2\pi)^q \E\bbrace*{D^{-q}\sum_{i=1}^{D}\nm*{u_i}^q}\\
&= (2\pi)^qD^{1-q}\E\parr*{\nm*{u_1}^q}.
\end{align*}
Therefore, it holds that
\begin{align*}
\E\bbrace*{\sup_x \nm*{\nabla_x s_D(x) - \widetilde{\bar{f}}_0(x)}^q} \leq 
 (2\pi)^qD^{1-q}\E\parr*{\|u\|^q} + L_{\widetilde{\bar{f}}_0}^q.
\end{align*}
Markov inequality now gives
\begin{align*}
\P\bbrace*{\sup_x\nm*{\nabla_x\parr*{s_D(x) - \widetilde{\bar{f}}_0(x)}}\geq \frac{t}{2r}} \leq \parr*{\frac{2r}{t}}^q\bbrace*{(2\pi)^qD^{1-q}\E\parr*{\|u\|^q} + L_{\widetilde{\bar{f}}_0}^q}. 
\end{align*}
Proceeding with the proof of Theorem \ref{thm:random_fourier_alpha}, we obtain 
\begin{align*}
\P\bbrace*{\sup_{x\in\cal{M}}\abs*{s_D(x) - \widetilde{\bar{f}}_0}\geq t} \leq \parr*{\frac{2r}{t}}^q\bbrace*{(2\pi)^qD^{1-q}\E\parr*{\|u\|^q} + L_{\widetilde{\bar{f}}_0}^q} + \parr*{\frac{c\diam(\cal{M})}{r}}^{md}\exp\parr*{-\frac{Dt^2}{8}}.
\end{align*}
Writing the right side of the above inequality in the form $\kappa_1r^{-md} + \kappa_2r^q$ and letting $r = (\kappa_1/\kappa_2)^{1/(q+md)}$, we obtain
\begin{align*}
\P\bbrace*{\sup_{x\in\cal{M}}\abs*{s_D(x) - \widetilde{\bar{f}}_0}\geq t} \leq 2\parr*{\frac{2c\diam(\cal{M})}{\varepsilon}}^{\frac{qmd}{q+md}}\bbrace*{(2\pi)^qD^{1-q}\E\parr*{\|u_1\|^q} + L_{\widetilde{\bar{f}}_0}^q}^{\frac{md}{q+md}}\exp\parr*{-\frac{D\varepsilon^2}{8}\frac{q}{q+md}}.
\end{align*}
For any $t > 0$, we can choose large enough $D = D(t)$ such that the right side of the above inequality is arbitrarily small. Lastly, since $\bar{f}_0$ coincides with $f_0$ on $[-2M,2M]^d$ for any given $M > 0$, the proof is complete.
\end{proof}

\subsection{Proof of Theorem \ref{thm:Lipschitz_alpha}}
\begin{proof}
First consider the case under (B2). Let $K:\RR^{md}\rightarrow\RR$  be the standard $md$-variate Gaussian density defined as $K(x) \define \exp(-\|x\|^2/2)(2\pi)^{-md/2}$, and $K_h(x) = K(x/h)h^{-md}$ for some positive constant $h$. Define $\bar{f}_h(x) \define (\bar{f}*K_h)(x)$. Then, it holds that
\begin{align*}
\abs*{\bar{f}_h(x) - \bar{f}(x)} &= \abs*{\int_{\RR^{md}} (2\pi)^{-md/2}\exp\parr*{-\frac{\|y\|^2}{2}}\bbrace*{\bar{f}(x-yh) - \bar{f}(x)}dy}\\
&\leq \int_{\RR^{md}} (2\pi)^{-md/2}\exp\parr*{-\frac{\|y\|^2}{2}}\abs*{\bar{f}(x-yh) - \bar{f}(x)}dy.
\end{align*}
Denote the upper bound of $\bar{f}$ as $M_{\bar{f}}$. Then, for any $t > 0$, there exists some positive constant $A = A(M_{\bar{f}},m,d,t)$ such that
\begin{align*}
&\ms\int_{\parr*{[-A,A]^{md}}^c}(2\pi)^{-md/2}\exp\parr*{-\frac{\|y\|^2}{2}}\abs*{\bar{f}(x-yh) - \bar{f}(x)}dy\\
&\leq 2M_{\bar{f}}\int_{\parr*{[-A,A]^{md}}^c}(2\pi)^{-md/2}\exp\parr*{-\frac{\|y\|^2}{2}}dy\\
&\leq t/4.
\end{align*}
Inside $[-A,A]^{md}$, using the uniform continuity of $\bar{f}$, there exists some $h = h(M_{\bar{f}},m,d,t)$, such that
\begin{align*}
\int_{[-A,A]^{md}}(2\pi)^{-md/2}\exp\parr*{-\frac{\|y\|^2}{2}}\abs*{\bar{f}(x-yh) - \bar{f}(x)}dy \leq t/4.
\end{align*}
Putting together the pieces, it holds that for any $t > 0$, there exists some $h = h(M_{\bar{f}},m,d,t)$ such that $\|\bar{f}_h-\bar{f}\|_\infty\leq t/2$.
 Since both $\bar{f}$ and $K_h$ belong to $L^1(\RR^{md})$, their Fourier transforms exist. It can be readily checked that $\hat{K}_h(u) = \exp\parr*{-2\pi^2h^2\|u\|^2}$, and thus
\begin{align*}
\hat{\bar{f}}_h(u) = \hat{\bar{f}}(u)\cdot \hat{K}_h(u) = \hat{\bar{f}}(u)\exp\parr*{-2\pi^2h^2\|u\|^2}.
\end{align*}
Using the relation $\|f*g\|_{L^q} \leq \|f\|_{L^q}\|g\|_{L^1}$ for any $q\geq 1$ and $f\in L^q(\RR^{md}), g\in L^1(\RR^{md})$ and the fact that $K_h\in L^1(\RR^{md})$, it holds that $\bar{f}_h\in L^1(\RR^{md})$. Moreover, it can readily checked that 
\begin{align*}
\mu_q^q\parr*{\hat{\bar{f}}_h} = \int_{\RR^{md}}\abs*{\hat{\bar{f}}_h(u)}\|u\|^qdu = \int_{\RR^{md}}\abs*{\hat{\bar{f}}(u)}\|u\|^q\exp\parr*{-2\pi^2h^2\|u\|^2}du < \infty
\end{align*}
for any $q\geq 1$. Therefore, by Theorem \ref{thm:random_fourier_alpha}, for any given $M > 0$ and $\cal{C} = [-M,M]^{md}$, and any given $t > 0$, we can find an approximating kernel $\widetilde{\bar{f}}_h = \widetilde{\bar{f}}_h(t)$ such that $\abs*{\widetilde{\bar{f}}_h-\bar{f}_h}\leq t/2$ uniformly over $\cal{C}$, and $\widetilde{\bar{f}}_h$ further satisfies expansion \eqref{eq:expansion} in the main paper with constants 
\begin{align*}
F = 2^m\nm*{\hat{\bar{f}}_h}_{L^1}, \quad B = 1, \quad \mu_a = 1
\end{align*}
for all $a\geq 1$. Choosing $h = h(M_{\bar{f}},m,d,t/2)$, by triangular inequality, we have
\begin{align*}
\abs*{\bar{f} - \widetilde{\bar{f}}_h} \leq t/2 + t/2 = t
\end{align*}
uniformly over $\cal{C}$. Furthermore, since $\bar{f}$ and $f$ coincides on $\cal{C}$, we have
\begin{align*}
\abs*{f - \widetilde{\bar{f}}_h} \leq t
\end{align*}
uniformly over $\cal{C}$. Now we upper bound the term $\nm*{\hat{\bar{f}}_h}_{L^1}$. To this end, we have 
\begin{align*}
\nm*{\hat{\bar{f}}_h}_{L^1} &= \int_{\RR^{md}} \abs*{\hat{\bar{f}}(u)}\exp(-2\pi^2h^2\|u\|^2)du \leq L_F\int_{\RR^{md}} \frac{1}{1+\|u\|^{md+\varepsilon}}\exp(-2\pi^2h^2\|u\|^2)du.
\end{align*}
Using polar coordinates, it holds that
\begin{align*}
\nm*{\hat{\bar{f}}_h}_{L^1} &\leq \Gamma_1(md)L_F\int_0^\infty \frac{r^{md-1}}{1+r^{md+\varepsilon}}\exp(-2\pi^2h^2r^2)dr\\
&\leq \Gamma_1(md)L_F\parr*{1 + \int_1^\infty \frac{1}{r^{1+\varepsilon}}dr} \\
&= (1+\varepsilon^{-1})\Gamma_1(md)L_F,
\end{align*}
where $\Gamma_1(\cdot)$ is defined in \eqref{eq:polar_constant} in the main paper.

Next, we consider the case under Condition (B3). It suffices to recalculate $\|\bar{f}_h-\bar{f}\|_\infty$ over $\cal{C} = [-M,M]^{md}$ and $\nm*{\hat{\bar{f}}_h}_{L^1}$ for any given $M > 0$. For the first quantity, we have by the $L$-Lipschitz continuity of $\bar{f}$ that for any $x\in\RR^{md}$
\begin{align*}
\abs*{\bar{f}_h(x) - \bar{f}(x)} &= \abs*{\int_{\RR^{md}} (2\pi)^{-md/2}\exp\parr*{-\frac{\|t\|^2}{2}}\bbrace*{\bar{f}(x-th) - \bar{f}(x)}dt}\\
&\leq \int_{\RR^{md}} (2\pi)^{-md/2}\exp\parr*{-\frac{\|t\|^2}{2}}\abs*{\bar{f}(x-th) - \bar{f}(x)}dt\\
&\leq Lh\int_{\RR^{md}} (2\pi)^{-md/2}\exp\parr*{-\frac{\|t\|^2}{2}}\|t\|dt\\
&= Lh\Gamma_2(md),
\end{align*}
where in the last step we integrate using polar coordinates and $\Gamma_2(\cdot)$ is defined in \eqref{eq:polar_constant} in the main paper. Therefore, $\abs*{\bar{f}_h(x) - \bar{f}(x)}\leq \Gamma_2(md)Lh$ uniformly over all $x\in\RR^{md}$. For the second quantity, we have
\begin{align*}
\nm*{\hat{\bar{f}}_h}_{L^1} &\leq \Gamma_1(md)L_F\int_{0}^\infty \frac{r^{md-1}}{1+r^{md}}\exp\parr*{-2\pi^2h^2r^2}dr\\
&= \Gamma_1(md)L_F\parr*{\int_0^1 1dr + \int_{1}^{1/h}\frac{1}{r}dr + \int_{1/h}^\infty \frac{1}{r}\exp(-2\pi^2h^2r^2)dr}\\
&\leq 2\Gamma_1(md)L_F\log(1/h)
\end{align*}
for $h$ small enough. The result follows by plugging in $h = t/(2\Gamma_2(md)L)$. 
\end{proof}

\subsection{Proof of Corollary \ref{cor:Lipschitz_shift_invariant}}
\begin{proof}
The proof follows directly from that of Corollary \ref{cor:random_fourier_m2}.
\end{proof}

\subsection{Proof of Corollary \ref{cor:Lipschitz_sufficient}}
\begin{proof}
Part (a) follows essentially from the proof of Corollary \ref{cor:schwartz}. For part (b), we follow the proof of Theorem \ref{thm:Lipschitz_alpha} and recalculate $\nm*{\hat{f}_h}_{L^1}$ under the product structure. Note that when every $h_\ell\in C^{m+1}(\RR^m)\bigcap L^1(\RR^m)$ and $\partial^{m+1}h_\ell/\partial x_i^{m+1}\in L^1(\RR^m)$ for each $i\in[m]$, then $\abs*{\hat{h}_\ell(u)} \leq C_\ell/\parr*{1+\|u\|^{m+1}}$, where $C_\ell = \|h_\ell\|_{L^1} + m^m\sum_{i=1}^m \nm*{\partial^{m+1}h_\ell/\partial x_i^{m+1}}_{L^1}$. Therefore, it holds that 
\begin{align*}
\int_{\RR^m} \abs*{\hat{h}_\ell(u)}\exp(-2\pi^2h^2\|u\|^2)du &\leq C_\ell\int_{\RR^m} \frac{1}{1+\|u\|^{m+1}}\exp(-2\pi^2h^2\|u\|^2)du\\
&= CC_\ell\int_0^\infty \frac{1}{1+r^{m+1}}\exp(-2\pi^2r^2h^2)r^{m-1}dr\\
&\leq 2CC_\ell,
\end{align*}
where $C = \Gamma_2(m)$, with $\Gamma_2(\cdot)$ defined in \eqref{eq:polar_constant} in the main paper. Therefore, with $u = (u_1^\top,\ldots,u_m^\top)^\top$ in $\RR^{md}$, we have
\begin{align*}
\nm*{\hat{f}_h}_{L^1} &= \int_{\RR^{md}} \abs*{\hat{f}(u)}\exp(-2\pi^2h^2\|u\|^2)du = \prod_{\ell=1}^d\bbrace*{\int_{\RR^m} \abs*{\hat{h}_\ell(u_\ell)}\exp(-2\pi^2h^2\|u_\ell\|^2)du_\ell} \\
&\leq (2C)^d\prod_{\ell=1}^d C_\ell.
\end{align*}
The rest of part (b) follows similarly. This completes the proof. 
\end{proof}

\subsection{Proof of Theorem \ref{thm:discontinuous_alpha}}
\begin{proof}
Using a similar smoothing technique as Theorem \ref{thm:Lipschitz_alpha}, define $\bar{f}_{0,h}\define (\bar{f}_0*K_h)(x)$. Then, it holds that
\begin{align*}
\bar{f}_{0,h}(x) &= \int_{\RR^d} \bar{f}_0(y)(2\pi)^{-\frac{d}{2}}h^{-d}\exp\parr*{-\frac{\|x-y\|^2}{2h^2}}dy\\
&= \prod_{\ell=1}^d \int_{\RR} \bar{h}_{0,\ell}(y_\ell)\frac{1}{\sqrt{2\pi}h}\exp\bbrace*{-\frac{(x_\ell-y_\ell)^2}{2h^2}}dy_\ell.
\end{align*}
Define $\bar{g}_{\ell,h}$ to be the $\ell$th term in the above product. Since for each $\ell\in[d]$, $\|\bar{h}_{0,\ell}\|_\infty \leq \Delta$, thus 
\begin{align*}
\abs*{g_{\ell,h}(x)} &= \abs*{\int_{-\infty}^\infty \bar{h}_{0,\ell}(y)\frac{1}{\sqrt{2\pi}h}\exp\bbrace*{-\frac{(x-y)^2}{2h^2}}dy}\\
&=\abs*{\int_{-\infty}^\infty \bar{h}_{0,\ell}(x+uh)\frac{1}{\sqrt{2\pi}}\exp\parr*{-\frac{u^2}{2}}du}\\
&\leq \Delta\int_{-\infty}^\infty \frac{1}{\sqrt{2\pi}}\exp\parr*{-\frac{u^2}{2}}du\\
&= \Delta.
\end{align*}
Thus $\|\bar{g}_{\ell,h}\|_\infty \leq \Delta$ for all $\ell\in[d]$. By telescoping, it holds that
\begin{align*}
\abs*{\bar{f}_{0,h}(x) - \bar{f}_0(x)} &= \abs*{\prod_{\ell=1}^d \bar{g}_{\ell,h}(x_\ell) - \prod_{\ell=1}^d \bar{h}_{0,\ell}(x_\ell)}\\
&\leq \Delta^{d-1}\sum_{\ell=1}^d\abs*{\bar{g}_{\ell,h}(x_\ell) - \bar{h}_{0,\ell}(x_\ell)}\\
&= \Delta^{d-1}\sum_{\ell=1}^d\abs*{\int_{-\infty}^\infty\frac{1}{\sqrt{2\pi}}\exp(-t^2/2)\bbrace*{\bar{h}_{0,\ell}(x_\ell - th) - \bar{h}_{0,\ell}(x_\ell)}dt}.
\end{align*}
By Condition (B4), for each $\ell\in[d]$, $\bar{h}_{0,\ell}$ has $J_\ell$ jump points denoted as $y_{\ell,1},\ldots,y_{\ell,J_\ell}$. Define $y_{\ell,0} \define -\infty$ and $y_{\ell,J_{\ell+1}} \define \infty$. Define 
\begin{align*}
\widetilde{S}\define \{[-2M_1,2M_1]^d\} \bigcap \{x\in\RR^d:\abs*{x_\ell - y_{\ell,k}} \geq M_2, \ell\in[d], k\in[J_\ell]\}.
\end{align*}
For any $x\in\widetilde{S}$, $x_\ell$ does not take any value in $\{y_{\ell,1},\ldots,y_{\ell,J_\ell}\}$ for all $\ell\in[d]$. Thus, there exists some integer $0\leq i_\ell\leq J_\ell$ such that $y_{\ell,i_\ell} < x_\ell < y_{\ell,i_{\ell}+1}$. Note that \\
\begin{align*}
\bbrace*{\bar{h}_{0,\ell}(x_\ell-th) \neq \bar{h}_{0,\ell}(x)} &\subset \bbrace*{x_\ell-th\leq y_{\ell,i_\ell}} \bigcup \bbrace*{x_\ell -th \geq y_{\ell,i_\ell+1}}\\
&= \bbrace*{t \geq (x_\ell-y_{\ell,i_\ell})/h} \bigcup \bbrace*{t \leq (x_\ell - y_{\ell,i_\ell+1})}\\
&\subset \bbrace*{t \geq M_2/h} \bigcup \bbrace*{t \leq -M_2/h}.
\end{align*}
Therefore, for each $\ell\in[d]$,
\begin{align*}
&\ms\abs*{\int_{-\infty}^\infty\frac{1}{\sqrt{2\pi}}\exp(-t^2/2)\bbrace*{\bar{h}_{0,\ell}(x_\ell - th) - \bar{h}_{0,\ell}(x_\ell)}dt} \\
&\leq \Delta\bbrace*{\int_{M_2/h}^\infty \frac{1}{\sqrt{2\pi}}\exp(-t^2/2)dt + \int_{-\infty}^{-M_2/h} \frac{1}{\sqrt{2\pi}}\exp(-t^2/2)dt}\\
&= 2\Delta\P\parr*{N(0,1)\geq M_2/h}\\
&\leq \Delta\exp\parr*{-\frac{M_2^2}{2h^2}},
\end{align*}
where in the last line we use the standard Gaussian tail bound. Putting together the pieces, it holds that
\begin{align*}
\abs*{\bar{f}_{0,h}(x) - \bar{f}_0(x)} \leq d\Delta^d\exp\parr*{-\frac{M_2^2}{2h^2}}
\end{align*}
uniformly over all $x\in \widetilde{S}$. 
Proceeding with the proof of Theorem \ref{thm:Lipschitz_alpha}, for any given $t > 0$, there exists an approximating kernel $\widetilde{\bar{f}}_{0,h}$ such that $\abs*{\widetilde{\bar{f}}_{0,h} - \bar{f}_{0,h}}\leq t/2$ uniformly over $\widetilde{S}$, 
and $\widetilde{\bar{f}}_{0,h}$ satisfies expansion \eqref{eq:expansion} in the main paper with constants
\begin{align*}
F = 4\nm*{\hat{\bar{f}}_{0,h}}_{L^1}, \quad B = 1, \quad \mu_a = 1
\end{align*}
for all $a \geq 1$. Choosing $h = M_2\log^{-1/2}\parr*{2d\Delta^d/t \vee 2}/\sqrt{2}$, by triangular inequality, we have
\begin{align*}
\abs*{\bar{f}_0(x) - \widetilde{\bar{f}}_{0,h}(x)} \leq t/2 + t/2 = t
\end{align*}
uniformly over $x\in\widetilde{S}$. Define $\widetilde{\bar{f}}_h(x,y)\define \widetilde{\bar{f}}_{0,h}(x-y)$. Then for any $(x,y)$ in the support \eqref{eq:discontinuous_support} in the main paper, $x-y \in\widetilde{S}$ and thus
\begin{align*}
\abs*{\bar{f}(x,y) - \widetilde{\bar{f}}_h(x,y)} = \abs*{\bar{f}_0(x-y) - \widetilde{\bar{f}}_{0,h}(x-y)} \leq t.
\end{align*}
Furthermore, since $\bar{f}$ coincides with $f$ on \eqref{eq:discontinuous_support} in the main paper, it holds that 
\begin{align*}
\abs*{f(x,y) - \widetilde{\bar{f}}_h(x,y)} \leq t
\end{align*}
uniformly over $\cal{C}$ defined in \eqref{eq:discontinuous_support} in the main paper. Now we upper bound the term $\nm*{\hat{\bar{f}}_{0,h}}_{L^1}$. For any $u = (u_1,\ldots,u_d)$, using the relation
\begin{align*}
\hat{\bar{f}}_{0,h}(u) = \hat{\bar{f}}_0(u)\hat{K}_h(u) = \prod_{\ell=1}^d \bbrace*{\hat{\bar{h}}_{0,\ell}(u_\ell)\exp(-2\pi^2h^2u_\ell^2)},
\end{align*}
we obtain
\begin{align*}
\nm*{\hat{\bar{f}}_{0,h}}_{L^1} &= \prod_{\ell=1}^d \int_{-\infty}^{\infty} \abs*{\hat{\bar{h}}_{0,\ell}(u)}\exp(-2\pi^2h^2u^2)du\\
&\leq \prod_{\ell=1}^d \bbrace*{\int_{-1}^1\abs*{\hat{\bar{h}}_{0,\ell}(u)}du + 2\int_{1}^\infty \frac{C_\ell}{u}\exp(-2\pi^2h^2u^2)du}\\
&=\prod_{\ell=1}^d \bbrace*{\int_{-1}^1\abs*{\hat{\bar{h}}_{0,\ell}(u)}du + 2C_\ell\int_{1}^{1/h} \frac{1}{u}\exp(-2\pi^2h^2u^2)du + 2C_\ell\int_{1/h}^\infty \frac{1}{u}\exp(-2\pi^2h^2u^2)du}\\
&\leq \prod_{\ell=1}^d \bbrace*{\int_{-1}^1 \abs*{\hat{\bar{h}}_{0,\ell}(u)}du + 4C_\ell\log(1/h\vee 2)}.
\end{align*}
This completes the proof.
\end{proof}


\subsection{Proof of Proposition \ref{prop:stability_sum_alpha}}
\begin{proof}
We will use the notation $\nm*{\cdot}_{\infty,\cal{D}}$ to indicate the supremum norm over set $\cal{D}$. First consider the additive case. Given any $t>0$ and $\{\lambda_i\}_{i=1}^N$, for each $1\leq i\leq N$, choose $t_i$ and $\widetilde{f}^i$ such that $\|\widetilde{f}^i-f^i\|_{\infty,\cal{C}^i} \leq t_i\leq t/(\sum_{i=1}^N|\lambda_i|)$. Let $\widetilde{f}^0\define \sum_{i=1}^N\lambda_i\widetilde{f}^i$. Then, we have
\begin{align*}
\|f^0 - \widetilde{f}^0\|_{\infty,\cal{C}^0} \leq \sum_{i=1}^N|\lambda_i|\|f^i-\widetilde{f}^i\|_{\infty,\cal{C}^i} \leq t.
\end{align*}
By assumption, for each $i\in[N]$, $\widetilde{f}^i$ admits an expansion 
\begin{align*}
\widetilde{f}^i(x^i_1,\ldots,x^i_m) = \sum_{j^i_1,\ldots,j^i_m=1}^{K^i} f^i_{j^i_1,\ldots,j^i_m}e^i_{j^i_1}(x^i_1)\ldots e^i_{j^i_m}(x^i_m)
\end{align*}
with corresponding constants $(F^i(t_i),B^i(t_i),\mu^i_a(t_i))$ for $a\geq 1$. Therefore, the kernel $\widetilde{f}^0$ can be written as
\begin{align*}
\widetilde{f}^0(x^1_1,\ldots,x^1_m,\ldots,x^N_1,\ldots,x^N_m) &= \sum_{i=1}^N \lambda_i\widetilde{f}^i(x^i_1,\ldots,x^i_m)\\
&= \sum_{i=1}^N \lambda_i\sum_{j^i_1,\ldots,j^i_m=1}^{K^i} f^i_{j^i_1,\ldots,j^i_m}e_{j^i_1}(x^i_1)\ldots e_{j^i_m}(x^i_m).
\end{align*}
The constants $(F(t),B(t),\mu_a(t))$ for the expansion of $\widetilde{f}^0$ thus take the value $(\sum_{i=1}^N\abs*{\lambda_i}F^i, B^1\vee\ldots\vee B^N, \mu^1_a\vee\ldots\mu^N_a)$. 

Now consider the product case. Given any $t>0$, for each $i\in[N]$, choose $t_i$ and $\widetilde{f}^i$ such that $\|\widetilde{f}^i-f^i\|_{\infty,\cal{C}^i} \leq t_i\leq t(M^i+t)/\bbrace*{N\prod_{i=1}^N(M^i+t)}$. Since $\|f^i\|_{\infty,\cal{C}^i} \leq M^i$ and $\|\widetilde{f}^i-f^i\|_{\infty,\cal{C}^i} \leq t_i \leq t$, triangular inequality gives $\|\widetilde{f}^i\|_{\infty,\cal{C}^i}\leq M^i+t$ for all $i\in[N]$. Let $\widetilde{f}^0\define \prod_{i=1}^N\widetilde{f}^i$. Then, we have by telescoping 
\begin{align*}
\|\widetilde{f}^0-f^0\|_{\infty,\cal{C}^0} \leq \sum_{i=1}^N\bbrace*{\|\widetilde{f}^i-f^i\|_{\infty,\cal{C}^i}\prod_{j\neq i}\parr*{\|\widetilde{f}^j\|_{\infty,\cal{C}^j}\vee \|f^j\|_{\infty,\cal{C}^j}}} \leq \sum_{i=1}^N \frac{t}{N} = t.
\end{align*}
Moreover, $\widetilde{f}^0$ can be written as 
\begin{align*}
&\ms\widetilde{f}^0(x^1_1,\ldots,x^1_m,\ldots,x^N_1,\ldots,x^N_m)\\
 &= \prod_{i=1}^N \widetilde{f}^i(x^i_1,\ldots,x^i_m)\\
&= \prod_{i=1}^N\parr*{\sum_{j^i_1,\ldots,j^i_m=1}^{K^i} f^i_{j_1,\ldots,j_m} e^i_{j^i_1}(x^i_1)\ldots e^i_{j^i_m}(x^i_m)}\\
&= \sum_{j^1_1,\ldots,j^1_m=1}^{K^1}\ldots\sum_{j^N_1,\ldots,j^N_m=1}^{K^N} f^1_{j^1_1,\ldots,j^1_m}\ldots f^N_{j^N_1,\ldots,j^N_m}e^1_{j^1_1}(x^1_1)\ldots e^1_{j^1_m}(x^1_m)\ldots e^N_{j^N_1}(x^N_1)\ldots e^N_{j^N_m}(x^N_m).
\end{align*}
The constants $(F(t),B(t),\mu_a(t))$ for the expansion of $\widetilde{f}^0$ thus take the value $(\prod_{i=1}^NF^i, B^1\vee\ldots\vee B^N, \mu^1_a\vee\ldots\mu^N_a)$. 

Lastly, note that for both cases, the symmetrized version of $\widetilde{f}^0$, denoted as $\widetilde{f}^0_\circ$, satisfies that
\begin{align*}
\nm*{f^0_\circ - \widetilde{f}^0_\circ}_{\infty,\cal{C}^0_\circ}\leq t
\end{align*}
when restricted to the symmetric set $\cal{C}^0_\circ$. Moreover, it can be readily checked that $\widetilde{f}^0_\circ$ admits an expansion with the same constants as $\widetilde{f}^0$. This completes the proof.
\end{proof}

\subsection{Proof of Proposition \ref{prop:hoeffding}}
\begin{proof}
By definition of $\{f_p\}_{p=1}^m$, $\{\widetilde{f}_p\}_{p=1}^m$ and $\{\cal{C}_p\}_{p=1}^m$, in order to upper bound $\abs*{f_p - \widetilde{f}_p}$ over $\cal{C}_p$, it suffices to derive upper bounds on $\abs*{\theta - \widetilde{\theta}}$ and $\abs*{g_p - \widetilde{g}_p}$ over $\cal{C}_p$, where $\{g_p\}_{p=1}^m$ are defined as
\begin{align*}
g_p(x_1,\ldots,x_p) \define \E\bbrace*{f(x_1,\ldots,x_p, \widetilde{X}_{p+1},\ldots, \widetilde{X}_m)}
\end{align*}
for $p\in[m-1]$ and $g_m\define f$. We first derive an upper bound on $\abs*{\theta - \widetilde{\theta}}$. We have
\begin{align*}
\abs*{\theta - \widetilde{\theta}} &= \abs*{\E\bbrace*{f(\widetilde{X}_1,\ldots,\widetilde{X}_m) - \widetilde{f}(\widetilde{X}_1,\ldots,\widetilde{X}_m)}}\\
&\leq \E\bbrace*{\abs*{f - \widetilde{f}}\mathbbm{1}\bbrace*{(\widetilde{X}_1^\top,\ldots,\widetilde{X}_m^\top)^\top \in \cal{C}}} + \E\bbrace*{\abs*{f - \widetilde{f}}\mathbbm{1}\bbrace*{(\widetilde{X}_1^\top,\ldots,\widetilde{X}_m^\top)^\top \notin \cal{C}}}\\
&\leq t + \bbrace*{\E\parr*{\abs*{f - \widetilde{f}}^2}}^{1/2}\bbrace*{\P\parr*{(\widetilde{X}_1^\top,\ldots,\widetilde{X}_m^\top)^\top\notin \cal{C}}}^{1/2}\\
&\leq t + \sqrt{2}\fence*{\bbrace*{\E f^2(\widetilde{X}_1,\ldots,\widetilde{X}_m)}^{1/2} + FB^m}\bbrace*{\P\parr*{(\widetilde{X}_1^\top,\ldots,\widetilde{X}_m^\top)^\top\notin \cal{C}}}^{1/2},
\end{align*}
where in the last line we use the upper bound $FB^m$ on $\widetilde{f}$. For each $p\in[m-1]$, we have for all $(x_1^\top,\ldots,x_p^\top)^\top \in\cal{C}_p$,
\begin{align*}
&\ms\abs*{g_p(x_1,\ldots,x_p) - \widetilde{g}_p(x_1,\ldots,x_p)}\\
&= \abs*{\E\bbrace*{f(x_1,\ldots,x_p,\widetilde{X}_{p+1},\ldots,\widetilde{X}_m) - \widetilde{f}(x_1,\ldots,x_p,\widetilde{X}_{p+1},\ldots,\widetilde{X}_m)}}\\
&\leq \E\bbrace*{\abs*{f - \widetilde{f}}\mathbbm{1}\bbrace*{(\widetilde{X}_{p+1}^\top,\ldots,\widetilde{X}_m^\top)^\top\in \cal{C}^{(x_1,\ldots,x_p)}}} + \E\bbrace*{\abs*{f - \widetilde{f}}\mathbbm{1}\bbrace*{(\widetilde{X}_{p+1}^\top,\ldots,\widetilde{X}_m^\top)^\top \notin \cal{C}^{(x_1,\ldots,x_p)}}}\\
&\leq t + \sqrt{2}\fence*{\sup_{\cal{C}_p}\bbrace*{\E f^2(x_1,\ldots,x_p,\widetilde{X}_{p+1},\ldots,\widetilde{X}_m)}^{1/2} + FB^m}\sup_{\cal{C}_p}\bbrace*{\P\parr*{(\widetilde{X}_{p+1}^\top,\ldots,\widetilde{X}_m^\top)^\top\notin \cal{C}^{(x_1,\ldots,x_p)}}}^{1/2},
\end{align*}
where in the last line we again use the upper bound on $\widetilde{f}$. For $p = m$, it holds that
\begin{align*}
\abs*{g_m(x_1,\ldots,x_m) - \widetilde{g}_m(x_1,\ldots,x_m)} = \abs*{f(x_1,\ldots,x_m) - \widetilde{f}(x_1,\ldots,x_m)} \leq t
\end{align*}
uniformly over $(x_1^\top,\ldots,x_m^\top)^\top\in\cal{C}$. With the definition of $\{s_i\}_{i=0}^{m-1}$ and $\{v_i\}_{i=0}^{m-1}$, there exists some positive constant $C = C(m)$ such that
\begin{align*}
\sup_{p\in[m]}\sup_{(x_1^\top,\ldots,x_p^\top)^\top \in\cal{C}_p}\abs*{f_p(x_1,\ldots,x_p) - \widetilde{f}_p(x_1,\ldots,x_p)} \leq C(t + \sum_{i=0}^{m-1}s_iv_i + FB^m\sum_{i=0}^{m-1}v_i).
\end{align*} 
This completes the proof.
\end{proof}

\subsection{Proof of Proposition \ref{prop:bern_U_alpha_general}}
\begin{proof}
It suffices to reprove Step I in the proof of Proposition \ref{prop:bern_U_alpha} and Step II therein remains the same. By Proposition \ref{prop:hoeffding} and Condition (A) with set $\cal{C}$, for any $t > 0$, there exists a symmetric approximating kernel $\widetilde{f}$ such that 
\begin{align*}
\abs*{f(x_1,\ldots,x_m) - \widetilde{f}(x_1,\ldots,x_m)} \leq t
\end{align*}
uniformly over $\cal{C}$ and $\widetilde{f}$ satisfies expansion \eqref{eq:expansion} in the main paper with constants $(F(t),B(t),\mu_a(t))$. Moreover, with $t'$ defined in \eqref{eq:bias} in the main paper, it holds for all $p\in[m]$ that
\begin{align*}
\abs*{f_p(x_1,\ldots,x_p) - \widetilde{f}_p(x_1,\ldots,x_p)} \leq t'
\end{align*}
uniformly over $\cal{C}_p$ defined in \eqref{eq:section} in the main paper by Proposition \ref{prop:hoeffding}. Now following the proof of Proposition \ref{prop:bern_U_alpha}, we have again that on the event
\begin{align*}
\cal{E}_p\define \{\text{for all } (i_1,\ldots,i_p) \text{ s.t. }1\leq i_1,\ldots,i_p\leq n, \text{it holds that }(X_{i_1}^\top,\ldots,X_{i_p}^\top)^\top \in \cal{C}_p\},
\end{align*}
it holds that
\begin{align*}
\abs*{V_{n,p} - \widetilde{V}_{n,p}} \leq n^pt'.
\end{align*}
Then, for any $x > 0$, there exists large enough $C_2$ such that
\begin{align*}
&\ms\P\bbrace*{\abs*{V_n - \theta}\geq C_2(x+t')} \\
&\leq \P\bbrace*{\abs*{V_n - \theta}\geq C_2(x+t') \bigcap \cal{E}_r \bigcap \ldots \bigcap \ldots \cal{E}_m} + \P\parr*{\bigcup_{p=r}^m \cal{E}_p^c}\\
&\leq\sum_{p=r}^m \P\bbrace*{n^{-p}\abs*{V_{n,p}}\geq (x+t')\bigcap \cal{E}_p} + R_{n,r}\\
&\leq \sum_{p=r}^m\P\parr*{n^{-p}\abs*{\widetilde{V}_{n,p}}\geq x} + R_{n,r}.
\end{align*}
This completes the proof.
\end{proof}

\subsection{Proof of Proposition \ref{prop:MDP_U_general}}
\begin{proof}
The proof is essentially the same as that of Proposition \ref{prop:MDP_U}; the two extra conditions, $t' = O(t)$ and $R_{n,2} = o(n^{-\gamma^2/2}/\sqrt{\log n})$, guarantee that $R_n$ defined therein satisfies that $R_n = (1-\Phi(x_n))o(1)$ uniformly over $x\in[0,\gamma\sqrt{\log n}]$.
\end{proof}

\subsection{Proof of Corollary \ref{cor:continuous_general}}
\begin{proof}
The proof is straightforward and thus omitted. 
\end{proof}

\subsection{Proof of Corollary \ref{cor:discontinuous_tail}}
\begin{proof}
The proof of (b) is trivial. We now prove (a). Following the proof of Theorem \ref{thm:discontinuous_alpha}, we can 
show that there exists an even function $\widetilde{f}_{0,h}$ and $\widetilde{f}_{h}(x,y)\define \widetilde{f}_{0,h}(x-y)$ such that $\widetilde{f}_h$ satisfies \eqref{eq:expansion} in the main paper with the same constants $(F(t),B(t),\mu_a(t))$ as in Theorem \ref{thm:discontinuous_alpha} and $\abs*{f(x, y) - \widetilde{f}_h(x, y)} \leq t$ uniformly over $(x,y)\in \cal{C}$, with $\cal{C}$ defined in \eqref{eq:discontinuous_support} in the main paper. However, it is not necessarily true that $\abs*{f(x,x) - \widetilde{f}_{h}(x,x)}\leq t$ unless $\abs*{y_{\ell,k}} \geq M_2$ for all $\ell\in[d]$ and $k\in[J_\ell]$. To this end, consider an auxiliary kernel $\check{f}$ defined to be $\check{f}(x,y) \define f(x,y)$ for all $x\neq y$ and $\check{f}(x,x) \define \widetilde{f}_h(x,x) = \widetilde{f}_{0,h}(0)$ for all $x$. Apparently, $\check{f}$ is still symmetric. Moreover, on the event
\begin{align*}
\cal{E}\define \{\abs*{X_i}\in [-M_1, M_1]^d, \forall i\in[n]\}\bigcap \{\abs*{(X_{i,\ell} - X_{j,\ell}) - y_{\ell,k}}\geq M_2, 1\leq i<j\leq n, \ell\in[d],k\in[J_\ell]\},
\end{align*}
it holds that $\abs*{\check{f}(X_i,X_j) - \widetilde{f}_h(X_i,X_j)}\leq t$ for all $1\leq i, j\leq n$. Now, using the relation
\begin{align*}
n^{-2}\sum_{i,j=1}^n f(X_i,X_j) = n^{-2}\sum_{i,j=1}^n \check{f}(X_i,X_j) + \frac{1}{n}\parr*{f_0(0) - \widetilde{f}_{0,h}(0)},
\end{align*}
we have for all $x$ such that $x> \abs*{\widetilde{f}_0(0) - f_0(0)}/n$ and $t'$ defined in \eqref{eq:bias} in the main paper, 
\begin{align*}
\P\parr*{\abs*{V_n  - \theta} \geq x + C_1t'} \leq \P\parr*{\abs*{\check{V}_n  - \theta} \geq x + C_1t' - \frac{1}{n}\abs*{\widetilde{f}_{0,h}(0) - f_0(0)}},
\end{align*}
where $\check{V}_n$ is the V-statistic generated by $\check{f}$. Since $X_1$ is absolutely continuous, we have
\begin{align*}
\check{\theta} \define \E\bbrace*{\check{f}(\widetilde{X}_1,\widetilde{X}_2)} = \theta \text{ and } \E\bbrace*{\check{f}(x,\widetilde{X}_2)} = \E\bbrace*{f(x,\widetilde{X}_2)}
\end{align*}
for any $x\in\RR^d$. Therefore, $\check{f}$ and $f$ have the same mean value and degeneracy level. By triangle inequality and definition of the constants $F$ and $B$, we have
\begin{align*}
\abs*{\widetilde{f}_{0,h}(0) - f_0(0)} \leq \abs*{\widetilde{f}_{0,h}(0)} + \abs*{f_0(0)} \leq FB + \abs*{f_0(0)} =  F + \abs*{f_0(0)}.
\end{align*} 
Define $\widetilde{V}_{n,1}$ and $\widetilde{V}_{n,2}$ to be the V-statistics generated by $\widetilde{f}_{h,1}$ and $\widetilde{f}_{h,2}$, respectively, where for $r=1,2$, $\widetilde{f}_{h,r}$ is the $r$th degenerate term in the Hoeffding decomposition of $\widetilde{f}_h$. Define similarly the V-statistics $\check{V}_{n,1}$ and $\check{V}_{n,2}$ generated by $\check{f}_1$ and $\check{f}_2$, respectively. We now show that $\abs*{\widetilde{V}_{n,1} - \check{V}_{n,1}} \leq t^\prime$ and $\abs*{\widetilde{V}_{n,2} - \check{V}_{n,2}} \leq t^\prime$ on the event $\cal{E}$. Define $\widetilde{\theta}$ to be the mean value of $\widetilde{f}$, $\widetilde{g}_2 \define \widetilde{f}$ and $\widetilde{g}_1(x) = \E\parr*{\widetilde{f}(x,\widetilde{X})}$, and similarly for $g_1$, $g_2$, $\check{g}_1$, and $\check{g}_2$. Then, by Proposition \ref{prop:hoeffding}, it holds that 
\begin{align*}
\abs*{\check{\theta} - \widetilde{\theta}} = \abs*{\theta - \widetilde{\theta}} \leq t^\prime
\end{align*}
and
\begin{align*}
\abs*{\check{g}_1(x) - \widetilde{g}_1(x)} = \abs*{g_1(x) - \widetilde{g}_1(x)} \leq t^\prime
\end{align*}
uniformly over $\cal{C}_1$, with $\cal{C}_1=[-M_1,M_1]^d$ defined in \eqref{eq:section} in the main paper. Thus by definition, $\abs*{f_1(x) - \widetilde{f}_1(x)}\leq t^\prime$ uniformly over $\cal{C}_1$. On the event $\cal{E}$, all $\{X_i\}_{i=1}^n$ take values in $\cal{C}_1$, therefore $\abs*{\widetilde{V}_{n,1} - \check{V}_{n,1}} \leq t^\prime$ on the event $\cal{E}$. By definition, we have
\begin{align*}
&\check{f}_2(X_i,X_j) = \check{f}(X_i,X_j) - \check{f}_1(X_i) - \check{f}_1(X_j) - \check{\theta},\\
&\widetilde{f}_2(X_i,X_j) = \widetilde{f}(X_i,X_j) - \widetilde{f}_1(X_i) - \widetilde{f}_1(X_j) - \widetilde{\theta}.
\end{align*}
When $i\neq j$, clearly $\abs*{\check{f}_2(X_i,X_j) - \widetilde{f}_2(X_i,X_j)}\leq t^\prime$ on the event $\cal{E}$. When $i = j$, since $\check{f}(X_i,X_i) = \widetilde{f}(X_i,X_i)$, we have the same conclusion. This implies that $\abs*{\widetilde{V}_{n,2} - \check{V}_{n,2}} \leq t^\prime$ on the event $\cal{E}$. Then, for both nondegenerate $f$ ($r = 1$) and degenerate $f$ ($r = 2$), it holds for any $x > (\abs*{f_0(0)} + F)/n$ that
\begin{align*}
&\ms\P\parr*{\abs*{V_n - \theta}\geq x+C_1t'}\\
&\leq \P\parr*{\abs*{\check{V}_n - \check{\theta}}\geq y+C_1t'}\\
&\leq \sum_{p=r}^2\P\parr*{\abs*{\widetilde{V}_{n,r}}\geq y} + \P(\cal{E}^c)\\
&\leq 2\sum_{p=r}^2 \exp\parr*{-\frac{C_2ny^{2/p}}{A_p^{1/p} + y^{1/p}M_p^{1/p}}} + n^2\parr*{\sum_{\ell=1}^d J_\ell}M_2D + n\sum_{\ell=1}^d \P\parr*{\abs*{X_{1,\ell}}\geq M_1}
\end{align*}
where $y = x - (\abs*{f_0(0)} + F)/n$, $t'$ is defined in \eqref{eq:bias} in the main paper, $C_1 = C_1(m), C_2 = C_3(m,\gamma_2)$, and $\{A_p\}_{p=1}^m$ and $\{M_p\}_{p=1}^m$ take the values $A_p \asymp F(t)^2, M_p \asymp F(t)(\log n)^{2p}$, with $F$ specified in Theorem \ref{thm:discontinuous_alpha}.
\end{proof}

\subsection{Supporting lemmas}
\label{subsec:preliminaries}

\begin{lemma}[Theorem 2, \cite{merlevede2009bernstein}]
\label{lemma:sample_mean_Bern}
Let $\{X_i\}_{i=1}^n$ be a stationary sequence of centered real-valued random variables. Suppose that the sequence satisfies either a geometric $\alpha$-mixing condition:
\begin{align*}
\alpha(n) \leq \textup{exp}(-\gamma n)
\end{align*}
or a geometric $\tau$-mixing condition:
\begin{align*}
\tau(n) \leq \textup{exp}(-\gamma n)
\end{align*}
for some positive constant $\gamma$, and there exists a positive $B$ such that $\sup_{i\geq 1}\|X_i\|_\infty \leq B$. Then there are positive constants $C_1$ and $C_2$ depending only on $\gamma$ such that for all $n\geq 2$ and positive $t$ satisfying $t < 1/\fence*{C_1B(\log n)^2}$, the following inequality holds:
\begin{align*}
\log\fence*{\E\bbrace*{\exp(tS_n)}} \leq \frac{C_2t^2(n\sigma^2+B^2)}{1-C_1tB(\log n)^2},
\end{align*}
where $S_n = \sum_{i=1}^n X_i$ and $\sigma^2$ is defined by 
\begin{align*}
\sigma^2 \define \textup{Var}(X_1) + 2\sum_{i>1}\abs*{\textup{Cov}(X_1,X_i)}.
\end{align*}
\end{lemma}

\begin{lemma}[Theorem 2.3, \cite{boucheron2013concentration}]
\label{lemma:bern_moment}
Let $X$ be a centered random variable. If for some $v>0$
\begin{align*}
\P\parr*{X>\sqrt{2vt}+ct} \vee \P\parr*{-X>\sqrt{2vt}+ct} \leq e^{-t}
\end{align*}
for every $t>0$, then for every integer $q\geq 1$,
\begin{align*}
\E\parr*{X^{2q}} \leq q!(8v)^q + (2q)!(4c)^{2q}.
\end{align*}
Conversely, if for some positive constants $A$ and $B$, 
\begin{align*}
\E\parr*{X^{2q}} \leq q!A^q + (2q)!B^{2q},
\end{align*}
then $X$ satisfies
\begin{align*}
\psi_X(\lambda) = \frac{v\lambda^2}{2(1-c\lambda)}
\end{align*}
with $(v,c) = (4(A+B^2), 2B)$, where $\psi_X$ is the log Laplace transform of $X$.
\end{lemma}

\begin{lemma}[Theorem 3, \cite{doukhan1994lecture}]
\label{lemma:alpha_covariance}
Assume $X$ and $Y$ are two random variables which are $\cal{U}$- and $\cal{V}$- measurable, respectively. Suppose $X\in L^p$ and $Y\in L^q$ with $p,q\geq 1$, and further denote the $\alpha$-mixing coefficient between $\cal{U}$ and $\cal{V}$ as $\alpha(\cal{U},\cal{V})$, then it holds that 
\begin{align*}
\abs*{\textup{Cov}(X,Y)} \leq 8\alpha^{1/r}(\cal{U},\cal{V})\|X\|_p\|Y\|_q,
\end{align*}
for any $p,q,r\geq 1$ and $1/p + 1/q + 1/r = 1$.
\end{lemma}

\begin{lemma}[Bochner's Theorem, Section 1.4.3, \cite{rudin2017fourier}]
\label{lemma:bochner}
A continuous kernel $f(x,y) = f_0(x-y)$ on $\RR^d$ is positive definite if and only if $f_0(\cdot)$ is the Fourier transform of a non-negative measure.
\end{lemma}

\section{Proofs of results in Section \ref{subsec:PLR}}


In this section, we use the following notation. $\mathbb{B}(r)$ and $S^{p-1}$ denote the Euclidean ball of radius $r$ and the unit sphere in $\RR^p$, respectively. For a real matrix $A$, $\eigmin(A)$, $\rho(A)$ and $\|A\|_\infty$ denote its minimum eigenvalue, spectral radius (the largest absolute value of all eigenvalues) and element-wise infinity norm, respectively. For a pair of two continuous random variables $(D_i, D_j)$, we will use $\E, \P, F, f$ to represent the ordinary expectation, probability, distribution function and density function, and $\EE, \PP, \FF, \ff$ to represent their counterparts under the product measure of $(D_i,D_j)$, that is, for any integrable real function $f$,
\begin{align*}
\E\bbrace*{f(D_i,D_j)} &= \int\int f(d_i,d_j) dF_{D_i,D_j}(d_i,d_j),\\
\EE\bbrace*{f(D_i,D_j)} &= \int\int f(d_i, d_j) dF_{D_i}(d_j)dF_{D_j}(d_j).
\end{align*}
and similarly for $\PP$, $\FF$ and $\ff$. For any real measurable function $f$, $\E^i\bbrace*{f(D_i,D_j)}$ denotes the measurable function of $D_j$ by taking expectation only with respect to $D_i$. We will use the definition $\widetilde{D}_{ij} \define D_i - D_j$, and $\widetilde{D}$ as a shorthand when there is no confusion about the pair to take the difference. Lastly, for any $S\subset [p]$, $\cal{C}(S,\alpha)$ stands for $\{\beta\in\RR^p:\|\beta_{S^c}\|_1\leq \alpha\|\beta_S\|_1\}$.

\label{proof:PLR}

\subsection{Proof of Theorem \ref{thm:error_bound}}
\begin{proof}
We adopt the general $M$-estimator framework introduced in Theorem 2.1 of \cite{han2017adaptive} by verifying Assumptions 2 and 3 therein. Define
\begin{align*}
\Gamma_n(\beta) \define \num\sum_{i<j}\frac{1}{h}K\parr*{\frac{\WW}{h}}\parr*{\YY-\XX^\top\beta}^2.
\end{align*}
Then, by definition, $\hat{\beta} \in\argmin_\beta \bbrace*{\Gamma_n(\beta) + \lambda_n\|\beta\|_1}$. By direct calculation, we have $\nabla_k \Gamma_n(\beta^*) = -2U_{1,k} - 2U_{2,k}$,
where
\begin{align*}
U_{1,k} &\define \num\sum_{i<j} \frac{1}{h}K\parr*{\frac{\WW}{h}}\widetilde{X}_{ijk}\bbrace*{g(W_i)-g(W_j)},\\
U_{2,k} &\define \num\sum_{i<j} \frac{1}{h}K\parr*{\frac{\WW}{h}}\widetilde{X}_{ijk}\widetilde{\varepsilon}_{ij}.
\end{align*}
By independence of $\{\varepsilon_i\}_{i=1}^n$ and $\{X_i,W_i\}_{i=1}^n$, it holds that $\EE\parr*{U_{2,k}} = 0$. Define the corresponding V-statistics to be
\begin{align*}
V_{1,k} &\define n^{-2}\sum_{i,j=1}^n \frac{1}{h}K\parr*{\frac{\WW}{h}}\widetilde{X}_{ijk}\bbrace*{g(W_i)-g(W_j)},\\
V_{2,k} &\define n^{-2}\sum_{i,j=1}^n \frac{1}{h}K\parr*{\frac{\WW}{h}}\widetilde{X}_{ijk}\widetilde{\varepsilon}_{ij}.
\end{align*}
Since $U_{1,k}$ and $U_{2,k}$ take value zero on the diagonal set, we have
\begin{align*}
\abs*{\nabla_k\Gamma_n(\beta^*)} &\leq c\abs*{V_{1,k} + V_{2,k}}\\
&\leq c\bbrace*{\abs*{V_{1,k}-\EE\parr*{U_{1,k}}} + \abs*{V_{2,k}-\EE\parr*{U_{2,k}}} + \abs*{\EE\parr*{U_{1,k}}}},
\end{align*}
where $c$ is some absolute constant. Firstly, by Assumption \ref{ass:dist_W}, $W$ has finite $(4+\delta_2)$th moment, therefore we have
\begin{align*}
\E\parr*{\fence*{g(W_i) - \E\bbrace*{g(W_i)}}^{4+\delta_2}} &= \E\parr*{\fence*{g(W_i) - \E\bbrace*{g(W_i')}}^{4+\delta_2}}\\
&= \E\parr*{\fence*{\E\bbrace*{g(W_i) - g(W_i')}}^{4+\delta_2}}\\
&\leq \E\fence*{\bbrace*{g(W_i)-g(W_j)}^{4+\delta_2}}\\
&\leq \E\bbrace*{\parr*{L|W_i-W_j|}^{4+\delta_2}}\\
&\leq L^{4+\delta_2}2^{4+\delta_2}\E\parr*{|W_i|^{4+\delta_2}},
\end{align*}
where in the first line we introduce an independent copy $W'$ of $W$, and in the fourth line we use the Lipschitz property of $g(\cdot)$ guaranteed by Assumption \ref{ass:g}. Therefore, the variable $V_i\define g(W_i) - \E\bbrace*{g(W_i)}$ has finite $(4+\delta_2)$th moment. Application of Lemma \ref{lemma:perturbation_g} gives
\begin{align*}
\abs*{V_{1,k} - \EE\parr*{U_{1,k}}} \leq c_1\bbrace*{\frac{\log(np)}{n}}^{1/2},\qquad \text{for all }k\in[p]
\end{align*}
with probability at least $1-\text{exp}\bbrace*{-c_1'\log(np)} - c_1''n^{-\delta_2/(8+\delta_2)}$ for positive constants $c_1, c_1',c_1''$. Next, applying Lemma \ref{lemma:perturbation_indep} with $D_i = (X_i, W_i,\varepsilon_i)$ with conditions verified by the assumptions, we have
\begin{align*}
\abs*{V_{2,k} - \EE\parr*{U_{2,k}}} \leq c_2\bbrace*{\frac{\log(np)}{n}}^{1/2}, \qquad \text{for all }k\in[p]
\end{align*}
with probability at least $1-\text{exp}\bbrace*{-c_2'\log(np)} - c_2''n^{-\delta_1/(4+\delta_1)}$ for some positive constants $c_2,c_2',c_2''$.  Lastly, for $\EE\parr*{U_{1,k}}$, we have
\begin{align*}
\abs*{\EE\parr*{U_{1,k}}} &\leq L\EE\bbrace*{\frac{1}{h}K\parr*{\frac{\WW}{h}}\abs*{\widetilde{X}_{ij,k}\WW}}\\
&\leq L\int\int K(u)\abs*{xuh}\ff_{\WW\mid\widetilde{X}_{ij,k}}(uh,x)dud\FF_{\widetilde{X}_{ij,k}}(x)\\
&= Lh\int\int K(u)|xu|\bbrace*{\ff_{\WW\mid\widetilde{X}_{ij,k}}(0,x) + uh\cdot\frac{\partial \ff_{\WW\mid\widetilde{X}_{ij,k}}(w, x)}{\partial w}|_{(\tau uh,x)}}dud\FF_{\widetilde{X}_{ij,k}}(x)\\
&\leq Lh\bbrace*{\int K(u)|u|du}\bbrace*{\int |x|\ff_{\widetilde{X}_{ij,k}\mid\WW}(x,0)dx}\cdot \ff_{\WW}(0) + \\
&\ms Lh^2\bbrace*{\int K(u)u^2du}\bbrace*{\int M_W|x|d\FF_{\widetilde{X}_{ij,k}}(x)}\\
&\leq LhM_WM_K\EE\parr*{\abs*{\widetilde{X}_{ij,k}}\mid \WW = 0} + Lh^2M_WM_K\EE\parr*{\abs*{\widetilde{X}_{ij,k}}}\\
&\leq LhM_WM_K\bbrace*{\EE\parr*{\widetilde{X}_{ij,k}^2\mid \WW = 0}}^{1/2} + LhM_WM_KC_0\bbrace*{\EE\parr*{\widetilde{X}_{ij,k}^2}}^{1/2}\\
&\leq \sqrt{2}LM_WM_K\kappa_x(C_0+1)h,
\end{align*}
where in the first line we use the Lipschitz property of $g(\cdot)$, in the third line we apply Taylor expansion to $\ff_{\WW\mid\widetilde{X}_{ij,k}}(w,x)$ around $(0,x)$ for each $x$ with $\tau\in[0, 1]$, in the fifth line we use Assumption \ref{ass:kernel_alpha} and the upper bound on $\ff_{\WW}(0)$ by Lemma A4.16 in \cite{Han2018SupplementT}, and in the seventh line we use the conditional and unconditional sub-Gaussianity of $\widetilde{X}_{ijk}$. Putting together the pieces, we have
\begin{align*}
\max_{1\leq k\leq p}\abs*{\nabla_k \Gamma_n(\beta^*)} \leq c\fence*{\bbrace*{\frac{\log(np)}{n}}^{1/2} + h}
\end{align*}
with probability at least $1-\text{exp}\bbrace*{-c'\log(np)}-c''n^{-\delta_1/(4+\delta_1)} - c''n^{-\delta_2/(8+\delta_2)}$ for some positive constants $c,c’,c''$, thus Assumption 2 in Theorem 2.1 in \cite{han2017adaptive} holds. Moreover, Assumption 3 holds by Lemma \ref{lemma:empirical_RE}. The proof is complete.
\end{proof}

\begin{lemma}
\label{lemma:sparse_eigenvalue}
Suppose Assumptions \ref{ass:alpha_mixing}, \ref{ass:kernel_alpha}, \ref{ass:density_W}, \ref{ass:sub-Gaussian-1} in Section \ref{subsubsec:assumptions} hold. Define
\begin{align*}
T_n \define {n\choose 2}^{-1}\sum_{i<j}\frac{1}{h_n}K\parr*{\frac{\widetilde{W_{ij}}}{h_n}}\widetilde{X}_{ij}\widetilde{X}_{ij}^\top.
\end{align*}
Denote the $q$-sparse eigenvalue of a generic real matrix $A\in\RR^{p\times p}$ as 
\begin{align*}
\|A\|_{2,q}\define \max_{v\in \mathbb{S}^{p-1},\|v\|_0\leq q}\abs*{v^\top Av},
\end{align*}
where $\mathbb{S}^{p-1}$ is the unit sphere in $\RR^p$. Choose $h$ such that $C_1(\log p/n)^{1/2} \leq h_n\leq C_2$ for some positive constant $C_1,C_2$. Moreover, given some $q\in[p]$, suppose that 
\begin{align*}
n\geq C_3\bbrace*{(\log p)(\log n)^4 \vee q^2(\log p)},
\end{align*}
where $C_3$ depends only on $C_1,C_2, \gamma_1,\gamma_2, M_K, M_W,\kappa_x$. Then, with probability at least $1-\exp\bbrace*{-c'\log(np)}$,
\begin{align*}
\|T_n-\EE\parr*{T_n}\|_{2,q} \leq cq\bbrace*{\frac{\log (np)}{n}}^{1/2},
\end{align*}
where $c,c'$ are positive constants that only depend on $C_1,C_2, \gamma_1,\gamma_2, M_K, M_W,\kappa_x$.
\end{lemma}
\begin{proof}
Throughout the proof, $c$ and $c'$ represent two generic constants that do not depend on $n$. Define $\theta \define \EE\parr*{T_n}$ and $\widetilde{T}_n \define T_n - \theta$. Then it holds that
\begin{align*}
\sup_{v\in \mathbb{S}^{p-1},\|v\|_0\leq q} \abs*{v^\top \widetilde{T}_n v} \leq \sup_{v\in \mathbb{S}^{p-1},\|v\|_0\leq q} \|v\|_1^2\|\widetilde{T}_n\|_\infty \leq q\|\widetilde{T}_n\|_\infty.
\end{align*}
Next, we upper bound $\|\widetilde{T}_n\|_\infty$. For each $(k,\ell)\in[p]\times [p]$, $\{X_{ik},X_{i\ell},W_i\}_{i=1}^n$ is stationary and satisfies the geometric $\alpha$-mixing condition by Assumption \ref{ass:alpha_mixing}. Applying Lemma \ref{lemma:perturbation} to $D_i = \parr*{X_{ik},X_{i\ell},W_i}$ with its conditions verified by Assumptions \ref{ass:alpha_mixing}, \ref{ass:kernel_alpha}, \ref{ass:density_W}, \ref{ass:sub-Gaussian-1}, we obtain that when $n$ satisfies $\bbrace*{(\log p)(\log n)^4 \vee q^2(\log p)} = O(n)$,\begin{align*}
\abs*{\parr*{\widetilde{T}_n}_{k,\ell}} \leq c\bbrace*{\frac{\log(np)}{n}}^{1/2}
\end{align*}
with probability at least $1-\exp\bbrace*{-c'\log(np)}$. Taking union bound over all $(k,\ell)\in[p]\times [p]$, we obtain 
\begin{align*}
\|\widetilde{T}_n\|_\infty \leq c\bbrace*{\frac{\log(np)}{n}}^{1/2}
\end{align*}
with probability at least $1-\exp\bbrace*{-c'\log(np)}$. Therefore, we obtain
\begin{align*}
\sup_{v\in \mathbb{S}^{p-1},\|v\|_0\leq q} \abs*{v^\top \widetilde{T}_n v} \leq cq\bbrace*{\frac{\log(np)}{n}}^{1/2}
\end{align*}
with probability at least $1-\text{exp}\bbrace*{-c'\log(np)}$.
\end{proof}

\begin{lemma}
\label{lemma:empirical_RE}
Suppose Assumptions \ref{ass:alpha_mixing} - \ref{ass:sub-Gaussian-1} in Section \ref{subsubsec:assumptions}. Assume $C_1(\log p/n)^{1/2}\leq h_n \leq C_2$ for some positive constants $C_1, C_2$, and moreover,
\begin{align*}
n\geq C_3\bbrace*{(\log p)(\log n)^4 \vee s^2(\log p)},
\end{align*}
where $C_3$ depends only on $C_1,C_2, \gamma_1,\gamma_2, M_K, M_W,M_\ell,\kappa_\ell,\kappa_x$. Then, with probability at least $1-\exp\bbrace*{-c\log(np)}$, Assumption 3 in Theorem 2.1 in \cite{han2017adaptive} holds with 
\begin{align*}
\delta \Gamma_n(\Delta) \geq \frac{1}{4}\kappa_\ell M_\ell\|\Delta\|^2
\end{align*}
uniformly over all $\Delta \in \cal{C}(S_n,3)$, where $M_\ell,\kappa_\ell$ are from Assumptions \ref{ass:informative} and \ref{ass:RE} and  
\begin{align*}
\Gamma_n(\beta) = \num\sum_{i<j}\frac{1}{h}K\parr*{\frac{\WW}{h}}\parr*{\YY-\XX^\top\beta}^2.
\end{align*}
\end{lemma}

\begin{proof}
Throughout the proof, $c$ is a generic constant. Recall the definition of $T_n$ from Lemma \ref{lemma:sparse_eigenvalue}:
\begin{align*}
T_n = \num\sum_{i<j}\frac{1}{h}K\parr*{\frac{\WW}{h}}\XX\XX^\top,
\end{align*}
$\theta = \EE\parr*{T_n}$ and $\tilde{T}_n = T_n - \theta$. It can be readily checked that $\delta\Gamma_n(\Delta) = \Delta^\top T_n\Delta$ for any $\Delta\in\RR^p$.  Denote $\cal{K}(s) \define \{v\in\RR^p: \|v\|_0\leq s, \|v\| = 1\}$, then by Lemma F.1 and F.3 in \cite{basu2015regularized}, it holds that
\begin{align*}
\sup_{v\in\cal{C}(S_n, 3)\bigcap \mathbb{B}(1)}\abs*{v^\top \widetilde{T}_n v} \leq 25\sup_{v\in\text{cl}\bbrace*{\text{conv}(\cal{K}(s))}}\abs*{v^\top\widetilde{T}_nv}\leq 75\sup_{v\in\cal{K}(2s)}\abs*{v^\top\widetilde{T}_nv},
\end{align*}
where $\text{cl}(A)$ and $\text{conv}(A)$ are the closure and convex hull of set $A$. Applying Lemma \ref{lemma:sparse_eigenvalue} with $q = 2s$ with its scaling requirement satisfied, it holds with probability at least $1-\text{exp}\bbrace*{-c\log(np)}$ that
\begin{align*}
\sup_{v\in\cal{C}(S_n,3)\bigcap\mathbb{B}(1)}\abs*{v^\top\widetilde{T}_nv} &\leq \frac{1}{4}\kappa_\ell M_\ell.
\end{align*}
Thus we have uniformly for $v\in\cal{C}(S_n,3)\bigcap \mathbb{B}(1)$ that
\begin{align*}
v^\top T_nv &= v^\top \widetilde{T}_n v  + v^\top\theta v \geq v^\top \theta v - \abs*{v^\top\widetilde{T}_n v}\geq v^\top\theta v-\kappa_\ell M_\ell/4.
\end{align*}
For the first term, we have
\begin{align*}
v^\top\theta v &= \EE\bbrace*{\frac{1}{h}K\parr*{\frac{\WW}{h}}\parr*{\XX^\top v}^2}\\
&\geq \EE\bbrace*{\parr*{\XX^\top v}^2\mid W_i=W_j}\cdot \ff_{\WW}(0) - M_WM_Kh\EE\bbrace*{\parr*{\XX^\top v}^2}\\
&\geq \kappa_\ell M_\ell - 2M_WM_Kh\parr*{2\kappa_x^2}\\
&\geq \frac{1}{2}\kappa_\ell M_\ell
\end{align*}
for sufficiently large $n$, where in the third line we invoke Assumption \ref{ass:informative}, \ref{ass:RE} and the sub-Gaussianity of $\XX^\top v$. Putting together the pieces, we have
\begin{align*}
\delta \Gamma_n(\Delta) = \abs*{\Delta^\top T_n\Delta} \geq \kappa_\ell M_\ell/4
\end{align*}
uniformly over all $\Delta\in\cal{C}(S_n,3)\bigcap\mathbb{B}(1)$ with probability at least $1-\text{exp}\bbrace*{-c\log(np)}$.
\end{proof}

\begin{lemma}
\label{lemma:perturbation}
Let $D_i = (X_i, V_i, W_i) \in\RR\times\RR\times\RR$ be generated from a stationary sequence for $i = 1,\ldots, n$, satisfying a geometric $\alpha$-mixing condition with coefficient $\alpha(i) \leq \text{exp}(-\gamma i)$ for all $i\geq 1$ and some positive $\gamma$. Let $K(\cdot)$ be a positive density kernel such that $\int K(u)du = 1$ and 
\begin{align*}
\max\bbrace*{\int |u|K(u)du, \sup_{u\in\RR}K(u)}\leq M_K
\end{align*}
for some positive constant $M_K$. Assume that conditional on $W_i$ and unconditionally, $X_i$ and $V_i$ are sub-Gaussian with constants $\kappa_x^2$ and $\kappa_v^2$, respectively. Assume that there exists some positive absolute constant $M_W$, such that
\begin{align*}
\max\bbrace*{\abs*{\frac{\partial f_{W\mid X,V}(w,x,v)}{\partial w}}, f_{W\mid X,V}(w,x,v)}\leq M_W
\end{align*}
for any $(w,x,v)$ such that the densities are defined. Take $h_n\geq K_1\{\log(np)/n\}^{1/2}$ for some positive absolute constant $K_1$, and further assume that $h_n \leq C_0$ for some constant $C_0$. Consider the V-statistic
\begin{align*}
V_n \define n^{-2}\sum_{i,j=1}^n \frac{1}{h}K\parr*{\frac{W_i-W_j}{h}}(X_i-X_j)(V_i-V_j)
\end{align*}
and concentration parameter
\begin{align*}
\theta \define \EE\bbrace*{\frac{1}{h}K\parr*{\frac{W_i-W_j}{h}}(X_i-X_j)(V_i-V_j)},
\end{align*}
then under the scaling $n = \Omega\bbrace*{(\log p)^3(\log n)^4}$, the following holds:
\begin{align*}
\abs*{V_n - \theta} \leq c\bbrace*{\frac{\log(np)}{n}}^{1/2}
\end{align*}
with probability at least $1-\textup{exp}\bbrace*{-c'\log(np)}$, where $c,c'$ are positive constants that do not depend on $n$.
\end{lemma}
\begin{proof}
Throughout the proof, $c_1,c_2$ are generic positive constants. Define the following quantities for the truncated version of the V-statistic:
\begin{align*}
&\widetilde{h}(D_i,D_j) \define \frac{1}{h}K\parr*{\frac{W_i-W_j}{h}}(X_i-X_j)(V_i-V_j)M_1(X_i)M_1(X_j)M_1(V_i)M_1(V_j),\\
&\widetilde{V}_n \define n^{-2}\sum_{i,j=1}^n\tilde{h}(D_i, D_j),\quad \widetilde{\theta} \define \EE\bbrace*{\widetilde{h}(D_i, D_j)},
\end{align*}
where $M_1(t) = \mathbbm{1}\bbrace*{|t|\leq \kappa_1\sqrt{\log(np)}}$, and $\kappa_1$ is a truncation constant to be specified later. Further define $Z_i \define (X_i, V_i)$ and $Z_{ij} \define (X_i - X_j)(V_i-V_j)$. \\
\textbf{Step 1:} Upper bound the truncated V-statistic.\\
Via Hoeffding decomposition, we have
\begin{align*}
\widetilde{V}_n - \widetilde{\theta} = \frac{2}{n}\sum_{i=1}^n g_1(D_i) + n^{-2}\sum_{i,j=1}^ng_2(D_i,D_j),
\end{align*}
where $g_1(D_i) = \E^j\bbrace*{\tilde{h}(D_i,D_j)}-\widetilde{\theta}, g_2(D_i,D_j) = \tilde{h}(D_i,D_j) - g_1(D_i) - g_1(D_j) - \tilde{\theta}$. It can be readily checked that $g_1$ and $g_2$ are both degenerate kernels. For any $t>0$, it holds that
\begin{align*}
\P\parr*{\abs*{\widetilde{V}_n-\widetilde{\theta}} > t} \leq \P\bbrace*{\abs*{\frac{1}{n}\sum_{i=1}^ng_1(D_i)}\geq \frac{t}{4}} + \P\bbrace*{\abs*{n^{-2}\sum_{i,j=1}^ng_2(D_i,D_j)}\geq \frac{t}{2}}.
\end{align*}
We now upper bound the two summands separately. Define $f(D_i) = \E^j \bbrace*{\widetilde{h}(D_i,D_j)}$ so that $g_1(D_i) = f(D_i) - \widetilde{\theta} = f(D_i) - \E\bbrace*{f(D_i)}$. Applying Lemma A3.3 in \cite{Han2018SupplementT} with $M_1 = M_W$ and $M_2 = M_K$, we have
\begin{align*}
\abs*{f(D_i) - f_1(D_i)} \leq f_2(D_i),
\end{align*}
where, with definition $\varphi(Z_i,Z_j) = Z_{ij}M(X_i)M(X_j)M(V_i)M(V_j)$,
\begin{equation}
\begin{aligned}
\label{eq:f1f2_sub}
f_1(D_i) &= \E^j\bbrace*{\varphi(Z_i,Z_j)\mid W_i = W_j}\cdot f_{W}(W_i)\leq M_W\E^j\bbrace*{\varphi(Z_i,Z_j)\mid W_i = W_j},\\
f_2(D_i) &= MM_Kh\E^j\bbrace*{|\varphi(Z_i,Z_j)|} \leq C_0M_WM_K\E^j\bbrace*{|\varphi(Z_i,Z_j)|}.
\end{aligned}
\end{equation}
Therefore, since $\|\varphi\|_\infty = O(\log(np))$, we have $\|f\|_\infty = O(\log(np))$. Furthermore, we have
\begin{align*}
\E\bbrace*{f_1^2(D_i)} &\leq M_W^2\EE\parr*{Z_{ij}^2}\\
&= M_W^2\EE\bbrace*{(X_i-X_j)^2(V_i-V_j)^2}\\
&\leq M_W^2 \fence*{\EE\bbrace*{(X_i-X_j)^4}}^{1/2}\fence*{\EE\bbrace*{(V_i-V_j)^4}}^{1/2}\\
&\leq M_W^2\parr*{3\cdot 4\kappa_x^4}^{1/2}\parr*{3\cdot 4\kappa_v^4}^{1/2}\\
&= 12M_W^2\kappa_x^2\kappa_v^2,
\end{align*}
where in the fourth line we use the fact that under the product measure, $(X_i-X_j)$ and $(V_i-V_j)$ are both sub-Gaussian with constants at most $2\kappa_x^2$ and $2\kappa_v^2$. Similarly, we have 
\begin{align*}
\E\bbrace*{f_2^2(D_i)} \leq C_0^2M_W^2M_K^2\EE\bbrace*{\varphi^2(Z_i,Z_j)} \leq C_0^2M_W^2M_K^2\EE\parr*{Z_{ij}^2}\leq 12C_0^2M_W^2M_K^2\kappa_x^2\kappa_v^2.
\end{align*}
Therefore, 
\begin{align*}
\E\bbrace*{f^2(D_i)} \leq 2\E\bbrace*{f_2^2(D_i)} + 2\E\bbrace*{f_1^2(D_i)}\leq 24M_W^2\kappa_x^2\kappa_v^2(1+C_0^2M_K^2).
\end{align*}
Since $\{D_i\}_{i=1}^n$ is geometrically $\alpha$-mixing, so is $\{g_1(D_i)\}_{i=1}^n$, thus by Lemma \ref{lemma:sample_mean_Bern} it holds that 
\begin{align*}
\P\bbrace*{\abs*{\frac{1}{n}\sum_{i=1}^n g_1(D_i)}>\frac{t}{4}} &= \P\bbrace*{\abs*{\frac{1}{n}\sum_{i=1}^n f(D_i) - \frac{1}{n}\sum_{i=1}^n\E\fence*{f(D_i)}}\geq \frac{t}{4}}\\
&\leq \exp\bbrace*{-\frac{C_3n^2t^2}{n\sigma^2+B^2+ntB(\log n)^2}},
\end{align*}
where $B = \|g_1(D_i)\|_\infty \leq 2\|f(D_i)\|_\infty = O(\log(np))$, and 
\begin{align*}
\sigma^2 = \var\bbrace*{g_1(D_1)} + 2\sum_{i>1}^{\infty} \abs*{\cov\bbrace*{g_1(D_1),g_1(D_i)}}.
\end{align*}
By previous calculation, it holds that the first summand $\var\bbrace*{g_1(D_1)}\leq \E\bbrace*{f^2(D_1)}$ is upper bounded by an absolute constant. Thus in order to show that $\sigma^2$ is also upper bounded by an absolute constant, it suffices to show that the second summation is also bounded. Since $\{g_1(D_i)\}_{i=1}^n$ is geometrically $\alpha$-mixing, applying Lemma \ref{lemma:alpha_covariance} with $p = q = 2+\delta$ and $r = (2+\delta)/\delta$ for some positive number $\delta$ gives
\begin{align*}
\sum_{i>1}^{\infty} \abs*{\cov\bbrace*{g_1(D_1),g_1(D_i)}} \leq \bbrace*{\sum_{i=1}^\infty \alpha^{\delta/(2+\delta)}(i)}\fence*{\E\bbrace*{|g_1(D_1)|^{2+\delta}}}^{2/(2+\delta)}.
\end{align*}
By the condition on the $\alpha$-mixing coefficient, it suffices to show that $\E\bbrace*{|g_1(D_1)|^{2+\delta}} < \infty$. To this end, we have
\begin{align*}
\E\bbrace*{\abs*{g_1(D_i)}^{2+\delta}} &= \E\fence*{\abs*{f(D_i)-\E\bbrace*{f(D_i)}}^{2+\delta}}\\
&\leq 2^{(2+\delta)-1}\fence*{\E\bbrace*{\abs*{f(D_i)}^{2+\delta}} + \abs*{\E\bbrace*{f(D_i)}}^{2+\delta}}\\
&\leq 2^{2+\delta}\E\bbrace*{\abs*{f(D_i)}^{2+\delta}}\\
&\leq 2^{2+\delta}\E\fence*{\bbrace*{\abs*{f_1(D_i)} + \abs*{f_2(D_i)}}^{2+\delta}}\\
&\leq 2^{3+2\delta}\fence*{\E\bbrace*{\abs*{f_1(D_i)}^{2+\delta}} + \E\bbrace*{\abs*{f_2(D_i)}^{2+\delta}}},
\end{align*}
where $f_1,f_2$ are defined in \eqref{eq:f1f2_sub}. For the first summand, it holds that
\begin{align*}
\E\bbrace*{\abs*{f_1(D_i)}^{2+\delta}} &\leq M^{2+\delta}\EE\parr*{|Z_{ij}|^{2+\delta}}\\
&= M^{2+\delta}\EE\parr*{\abs*{X_i-X_j}^{2+\delta}\abs*{V_i-V_j}^{2+\delta}}\\
&\leq M^{2+\delta}\fence*{\EE\bbrace*{\abs*{X_i-X_j}^{2(2+\delta)}}}^{1/2}\fence*{\EE\bbrace*{\abs*{V_i-V_j}^{2(2+\delta)}}}^{1/2}\\
&\leq M^{2+\delta}\bbrace*{\sqrt{4+2\delta}\parr*{\sqrt{2}\kappa_x}}^{2+\delta}\bbrace*{\sqrt{4+2\delta}\parr*{\sqrt{2}\kappa_v}}^{2+\delta},
\end{align*}
where in the last step we use the sub-Gaussianity of $(X_i-X_j)$ and $(V_i-V_j)$ under the product measure. Similar calculation shows that the second term is also upper bounded by an absolute constant. This concludes that $\sigma^2$ is upper bounded by an absolute constant. Choosing $t \asymp \{\log(np)/n\}^{1/2}$, then under the scaling $(\log p)^3(\log n)^4 = O(n)$, it holds that 
\begin{align}
\label{eq:tail_1_sub}
\P\fence*{\abs*{\frac{1}{n}\sum_{i=1}^n g_1(D_i)}>c_1\bbrace*{\frac{\log(np)}{n}}^{1/2}} \leq \text{exp}\bbrace*{-c_2\log(np)}.
\end{align}
Next, we consider the tail of the degenerate V-statistic generated by the canonical kernel $g_2(D_i,D_j)$. 
To this end, for arbitrary positive constants $t^\prime$ and $B_W$, Lemma \ref{lemma:expansion}(a) verifies Condition (A$^\prime$) with $\cal{C} = \Big\{[-B_W,B_W]\times (-\infty,+\infty) \times (-\infty, +\infty)\Big\}^2$ and constants:
\begin{align*}
F(t^\prime)\asymp h^{-1} \asymp n^{1/2}(\log p)^{-1/2}, \quad B(t^\prime) \asymp \log p, \quad \mu_a(t^\prime) \asymp 1
\end{align*}
for all $a\geq 1$. Moreover, by choosing $B_W$ to be large enough, $t$ defined in \eqref{eq:bias} in the main paper satisfies $t = O(t^\prime)$. Applying the fully degenerate version of Proposition \ref{prop:bern_U_alpha_general} with $x\asymp t \asymp t^\prime \asymp \bbrace*{(\log np)/n}^{1/2}$, then we have
\begin{align}
\label{eq:tail_3_sub}
\P\fence*{\abs*{n^{-2}\sum_{i,j=1}^n g_2(D_i,D_j)}\geq c_1\bbrace*{\frac{\log(np)}{n}}^{1/2}} \leq \text{exp}\bbrace*{-c_2\log (np)}
\end{align}
under the scaling $(\log p)^3(\log n)^4 = O(n)$. Combining \eqref{eq:tail_1_sub} - \eqref{eq:tail_3_sub}, we have
\begin{align*}
\P\fence*{\abs*{\tilde{V}_n - \tilde{\theta}}\geq c_1\bbrace*{\frac{\log(np)}{n}}^{1/2}} \leq \text{exp}\bbrace*{-c_2\log(np)}.
\end{align*}
\\
\textbf{Step 2:} Next we calculate the difference between $\tilde{\theta}$ and $\theta$. By definition and symmetry, 
\begin{align*}
\abs*{\tilde{\theta} - \theta} &= \EE\bigg[\frac{1}{h}K\parr*{\frac{W_i-W_j}{h}}(X_i-X_j)(V_i-V_j)\cdot \mathbbm{1}\bbrace*{\abs*{X_i}\geq \kappa_1\sqrt{\log(np)}}\cdot\\
&\ms\mathbbm{1}\bbrace*{|X_j|\geq \kappa_1\sqrt{\log(np)}}\cdot\mathbbm{1}\bbrace*{|V_i|\geq \kappa_1\sqrt{\log(np)}}\cdot\mathbbm{1}\bbrace*{|V_i|\geq \kappa_1\sqrt{\log(np)}}\bigg]\\
&\leq 2\EE\fence*{\frac{1}{h}K\parr*{\frac{W_i-W_j}{h}}Z_{ij}\mathbbm{1}\bbrace*{|X_i|\geq \kappa_1\sqrt{\log(np)}}} + \\
&\ms2\EE\fence*{\frac{1}{h}K\parr*{\frac{W_i-W_j}{h}}Z_{ij}\mathbbm{1}\bbrace*{|V_i|\geq \kappa_1\sqrt{\log(np)}}}.
\end{align*}
For the first term, note that
\begin{align*}
\EE\fence*{\frac{1}{h}K\parr*{\frac{\WW}{h}}Z_{ij}\mathbbm{1}\bbrace*{|X_i|\geq \kappa_1\sqrt{\log(np)}}} &\leq \parr*{\EE\fence*{\frac{1}{h}K\parr*{\frac{\WW}{h}}\XX^2\mathbbm{1}\bbrace*{|X_i|\geq \kappa_1\sqrt{\log(np)}}}}^{1/2}\cdot\\
&\ms\parr*{\EE\fence*{\frac{1}{h}K\parr*{\frac{\WW}{h}}\VV^2\mathbbm{1}\bbrace*{|X_i|\geq \kappa_1\sqrt{\log(np)}}}}^{1/2}.
\end{align*}
Applying Lemma A3.2 in \cite{Han2018SupplementT} to the first term with $Z = \XX^2\mathbbm{1}\bbrace*{\abs*{X_i}\geq \kappa_1\sqrt{\log(np)}}, M_1 = M_W, M_2 = M_K$, it holds that
\begin{align*}
&\EE\fence*{\frac{1}{h}K\parr*{\frac{\WW}{h}}\XX^2\mathbbm{1}\bbrace*{\abs*{X_i}\geq \kappa_1\sqrt{\log(np)}}}\\
&\leq \EE\fence*{\XX^2\mathbbm{1}\bbrace*{\abs*{X_i}\geq \kappa_1\sqrt{\log(np)}}\mid\WW=0}\cdot \ff_{\WW}(0) + M_WM_Kh\EE\fence*{\XX^2\mathbbm{1}\bbrace*{\abs*{X_i}\geq \kappa_1\sqrt{\log(np)}}}\\
&\leq c\parr*{\EE\fence*{\XX^2\mathbbm{1}\bbrace*{\abs*{X_i}\geq \kappa_1\sqrt{\log(np)}}\mid\WW=0} + \EE\fence*{\XX^2\mathbbm{1}\bbrace*{\abs*{X_i}\geq \kappa_1\sqrt{\log(np)}}}},
\end{align*}
where in the last line we use the upper bound on $\ff_{\WW}$ calculated by Lemma A4.16 in \cite{Han2018SupplementT}. Since, conditional on $\WW=0$ and unconditionally, $\XX$ is sub-Gaussian with constant at most $2\kappa_x^2$ under the product measure, it holds that
\begin{align*}
\EE\fence*{\XX^2\mathbbm{1}\bbrace*{|X_i|\geq \kappa_1\sqrt{\log(np)}}} &\leq \bbrace*{\EE\parr*{\XX^4}}^{1/2}\fence*{\P\bbrace*{|X_i|\geq \kappa_1\sqrt{\log(np)}}}^{1/2}\\
&\leq \parr*{3\cdot4\kappa_x^4}^{1/2} \cdot \exp\bbrace*{-\frac{\kappa_1^2}{2\kappa_x^2}\log(np)}\\
&\leq c\bbrace*{\frac{\log(np)}{n}}^{1/2}
\end{align*}
for sufficiently large $\kappa_1$. Similarly, it can be shown that 
\begin{align*}
\EE\fence*{\XX^2\mathbbm{1}\bbrace*{|X_i|\geq \kappa_1\sqrt{\log(np)}}\mid \WW=0} \leq c\bbrace*{\frac{\log(np)}{n}}^{1/2}.
\end{align*}
Put together the pieces gives
\begin{align*}
\abs*{\tilde{\theta} - \theta} \leq c\bbrace*{\frac{\log(np)}{n}}^{1/2}.
\end{align*}
\textbf{Step 3:} With the definition $\cal{E}_n \define \bigcap_{i=1}^n\bbrace*{|X_i|\leq \kappa_1\sqrt{\log(np)}, |V_i|\leq\kappa_1\sqrt{\log(np)}}$ and $\cal{E}_n^c$ as its complement, we have
\begin{align*}
\P\parr*{\cal{E}_n^c} &\leq n\cdot\P\bbrace*{|X_i|\geq \kappa_1\sqrt{\log(np)}}+\P\bbrace*{|V_i| \geq \kappa_1\sqrt{\log(np)}}\\
&\leq n\fence*{\exp\bbrace*{-\frac{\kappa_1^2\log(np)}{2\kappa_x^2}}+\exp\bbrace*{-\frac{\kappa_1^2\log(np)}{2\kappa_v^2}}}\\
&\leq \frac{c}{p^2}
\end{align*}
for sufficiently large $\kappa_1$. Lastly, putting together the pieces, we have
\begin{align*}
\P\fence*{\abs*{V_n-\theta} \geq c_1\bbrace*{\frac{\log(np)}{n}}^{1/2}} &\leq \P\fence*{\abs*{V_n-\theta} \geq c_1\bbrace*{\frac{\log(np)}{n}}^{1/2}\bigcap \cal{E}_n} + \P\parr*{\cal{E}_n^c}\\
&\leq \P\fence*{\abs*{\widetilde{V}_n-\theta}\geq c_1\bbrace*{\frac{\log(np)}{n}}^{1/2}} + \frac{c}{p^2}\\
&\leq \P\fence*{\abs*{\widetilde{V}_n-\widetilde{\theta}} + \abs*{\widetilde{\theta}-\theta}\geq c_1\bbrace*{\frac{\log(np)}{n}}^{1/2}} + \frac{c}{p^2}\\
&\leq \P\fence*{\abs*{\widetilde{V}_n-\widetilde{\theta}} \geq c_1\bbrace*{\frac{\log(np)}{n}}^{1/2}} + \frac{c}{p^2}\\
&\leq \exp\bbrace*{-c_2\log(np)}.
\end{align*}
This completes the proof.
\end{proof}

\begin{lemma}
\label{lemma:perturbation_g}
Let $D_i = (X_i, W_i)\in\RR^p\times\RR$ be generated from a stationary sequence for $i = 1,\ldots, n$, satisfying a geometric $\alpha$-mixing condition with coefficient $\alpha(i)\leq \exp(-\gamma i)$ for all $i\geq 1$ and some positive $\gamma$. Let $K(\cdot)$ be a positive density kernel such that $\int K(u)du = 1$ and 
\begin{align*}
\max\bbrace*{\int |u|K(u)du,~\sup_{u\in\RR}K(u)}\leq M_K
\end{align*}
for some positive constant $M_K$. Assume that conditional on $W_i$ and unconditionally, $X_i$ is sub-Gaussian with constants $\kappa_x^2$. Assume that there exists some positive absolute constant $M_W$, such that
\begin{align*}
\max\bbrace*{\abs*{\frac{\partial f_{W\mid X}(w,x)}{\partial w}}, f_{W\mid X}(w,x)}\leq M_W
\end{align*}
for any $(w,x)$ such that the densities are defined. Take $h_n\geq K_1\{\log(np)/n\}^{1/2}$ for some positive absolute constant $K_1$, and further assume that $h_n \leq C_0$ for some constant $C_0$. Let $\{V_i\}_{i=1}^n\define \{g(W_i)\}_{i=1}^n$ be some function of $\{W_i\}_{i=1}^n$ and further assume that $V_1$ has finite $(4+4\delta)$th moment for some positive $\delta$. Consider the V-statistic
\begin{align*}
V_{n,k} \define n^{-2}\sum_{i,j=1}^n \frac{1}{h}K\parr*{\frac{W_i-W_j}{h}}(X_{ik}-X_{jk})(V_i-V_j)
\end{align*}
with concentration parameter 
\begin{align*}
\theta_k \define \EE\bbrace*{\frac{1}{h}K\parr*{\frac{W_i-W_j}{h}}(X_{ik}-X_{jk})(V_i-V_j)},
\end{align*}
then under the scaling 
\begin{align*}
n = \Omega\bbrace*{(\log p)^{\frac{4+2\delta}{1+\delta}}(\log n)^{\frac{8+4\delta}{1+\delta}}} ,
\end{align*}
the following holds:
\begin{align*}
\max_{k\in[p]}\abs*{V_{n,k} - \theta_k} \leq c\bbrace*{\frac{\log(np)}{n}}^{1/2}
\end{align*}
with probability at least $1-\exp\bbrace*{-c’\log(np)} - n^{-\frac{\delta}{2+\delta}}$ for some absolute constants $c,c'$.
\end{lemma}
\begin{proof}
Throughout the proof, $c,c_1,c_2$ will be generic positive constants. We consider each component $X_k$ of $X$ and let $\{D_i\}_{i=1}^n = \{W_i,X_{ik},V_i\}_{i=1}^n$. First define the following quantities for the truncated version of the V-statistic:
\begin{align*}
&\widetilde{h}_k(D_i,D_j) \define \frac{1}{h}K\parr*{\frac{W_i-W_j}{h}}(X_{ik}-X_{jk})(V_i-V_j)M_1(X_{ik})M_1(X_{jk})M_2(V_i)M_2(V_j),\\
&\widetilde{V}_{n,k} \define n^{-2}\sum_{i,j=1}^n\widetilde{h}_k(D_i, D_j), \quad \widetilde{\theta}_k \define \EE\bbrace*{\widetilde{h}_k(D_i, D_j)},
\end{align*}
where $M_1(t) = \mathbbm{1}\bbrace*{|t|\leq \kappa_1\sqrt{\log(np)}}, M_2(t) = \mathbbm{1}\bbrace*{|t|\leq \theta_n}$ and $\kappa_1$, $\theta_n$ are two truncation constants to be specified later. Further define $Z_{i} \define (X_{ik}, V_i)$ and $Z_{ij} \define (X_{ik} - X_{jk})(V_i-V_j)$. \\
\textbf{Step 1:} Upper bound the truncated V-statistic.\\
Via Hoeffding decomposition, we have
\begin{align*}
\widetilde{V}_{n,k} - \widetilde{\theta}_k = \frac{2}{n}\sum_{i=1}^n g_1(D_i) + n^{-2}\sum_{i,j=1}^ng_2(D_i,D_j),
\end{align*}
where $g_1(D_i) = \E^j\bbrace*{\widetilde{h}_k(D_i,D_j)}-\widetilde{\theta}_k, g_2(D_i,D_j) = \widetilde{h}_k(D_i,D_j) - \widetilde{\theta}_k - g_1(D_i) - g_1(D_j)$. It's easy to verify that $g_1$ and $g_2$ are both degenerate kernels. For any $t>0$, it holds that
\begin{align*}
\P\parr*{\abs*{\widetilde{V}_{n,k}-\widetilde{\theta}} > t} \leq \P\bbrace*{\abs*{\frac{1}{n}\sum_{i=1}^ng_1(D_i)}\geq \frac{t}{4}} + \P\bbrace*{\abs*{n^{-2}\sum_{i,j=1}^ng_2(D_i,D_j)}\geq \frac{t}{2}}.
\end{align*}
We now upper bound the two summands separately. Define $f(D_i) = \E^j \bbrace*{\widetilde{h}_k(D_i,D_j)}$ so that $g_1(D_i) = f(D_i) - \widetilde{\theta}_k = f(D_i) - \E\bbrace*{f(D_i)}$. Apply Lemma A3.3 in \cite{Han2018SupplementT} with $M_1 = M_W$ and $M_2 = M_K$, we have
\begin{align*}
|f(D_i) - f_1(D_i)| \leq f_2(D_i),
\end{align*}
where 
\begin{equation}
\begin{aligned}
\label{eq:f1f2_g}
f_1(D_i) &= \E^j\bbrace*{\varphi(Z_i,Z_j)\mid W_i = W_j}\cdot f_{W}(W_i)\leq M_W\E^j\bbrace*{\varphi(Z_i,Z_j)\mid W_i = W_j},\\
f_2(D_i) &= M_WM_Kh\E^j\bbrace*{|\varphi(Z_i,Z_j)|} \leq C_0M_WM_K\E^j\bbrace*{|\varphi(Z_i,Z_j)|}
\end{aligned}
\end{equation}
and $\varphi(Z_i,Z_j) = Z_{ij}M_1(X_{ik})M_1(X_{jk})M_2(V_i)M_2(V_j)$. Therefore, we have
\begin{align*}
\|f(D_i)\|_\infty = O(\bbrace*{\log(np)}^{1/2}\theta_n).
\end{align*}
Furthermore, we have
\begin{equation*}
\begin{aligned}
\E\bbrace*{f_1^2(D_i)} &\leq M_W^2\EE\parr*{Z_{ij}^2}\\
&= M_W^2\EE\bbrace*{(X_i-X_j)^2(V_i-V_j)^2}\\
&\leq M_W^2 \fence*{\EE\bbrace*{(X_{ik}-X_{jk})^4}}^{1/2}\fence*{\EE\bbrace*{(V_i-V_j)^4}}^{1/2}\\
&\leq M_W^2\parr*{2^4\cdot3\kappa_x^4}^{1/2}\bbrace*{2^4\cdot\E\parr*{V_i^4}}^{1/2},
\end{aligned}
\end{equation*}
where in the fourth line we use the fact that under the product measure, $(X_{ik}-X_{jk})$ is sub-Gaussian with constants at most $2\kappa_x^2$ and also $V_i$ has finite $4$th moment by assumption. Similarly, we have 
\begin{align*}
\E\bbrace*{f_2^2(D_i)} \leq C_0^2M_K^2M_W^2\parr*{2^4\cdot3\kappa_x^4}^{1/2}\bbrace*{2^4\cdot\E\parr*{V_i^4}}^{1/2}.
\end{align*}
Therefore, 
\begin{equation*}
\begin{aligned}
\E\bbrace*{f^2(D_i)} &= \E\bbrace*{(f(D_i) - f_1(D_i) + f_1(D_i))^2}\\
&\leq 2\E\bbrace*{f_2^2(D_i)} + 2\E\bbrace*{f_1^2(D_i)}\\
&\leq (C_0^2M_K^2 + 1)M_W^2{2^4\cdot3\kappa_x^4}^{1/2}\bbrace*{2^4\cdot\E\parr*{V_i^4}}^{1/2}.
\end{aligned}
\end{equation*}
Since $\{D_i\}_{i=1}^n$ is geometrically $\alpha$-mixing, so is $\{g_1(D_i)\}_{i=1}^n$. Applying Lemma \ref{lemma:sample_mean_Bern}, it holds that 
\begin{align*}
\P\bbrace*{\abs*{\frac{1}{n}\sum_{i=1}^n g_1(D_i)}>\frac{t}{4}} &= \P\bbrace*{\abs*{\frac{1}{n}\sum_{i=1}^n f(D_i) - \frac{1}{n}\sum_{i=1}^n\E\fence*{f(D_i)}}\geq \frac{t}{4}}\\
&\leq \exp\bbrace*{-\frac{C_3n^2t^2}{n\sigma^2+B^2+ntB(\log n)^2}},
\end{align*}
where $B = \|g_1(D_i)\|_\infty \leq 2\|f(D_i)\|_\infty = O(\sqrt{\log(np)}\theta_n)$, and
\begin{align*}
\sigma^2 = \var\bbrace*{g_1(D_1)} + 2\sum_{i>1}^\infty \abs*{\cov\bbrace*{g_1(D_1),g_1(D_i)}}.
\end{align*}
By the previous calculation, $\var\bbrace*{g_1(D_1)}\leq \E\bbrace*{f_1^2(D_1)}$ is upper bounded by an absolute constant, thus in order to show that $\sigma^2$ is also upper bounded by an absolute constant, it suffices to show that the second summation is also bounded. Apply Lemma \ref{lemma:alpha_covariance} with $p = q = 2+\delta$ and $r = (2+\delta)/\delta$ for the positive constant $\delta$ in the moment condition of $V_i$, it holds that
\begin{align*}
\sum_{j=1}^\infty \abs*{\text{Cov}(g_1(D_1),g_1(D_{1+j}))} \leq \bbrace*{\sum_{i=1}^\infty \alpha^{\delta/(2+\delta)}(i)}\fence*{\E\bbrace*{|g_1(D_1)|^{2+\delta}}}^{2/(2+\delta)}.
\end{align*}
By the condition on the $\alpha$-mixing coefficient, it suffices to show that $\E\bbrace*{|g_1(D_1)|^{2+\delta}} < \infty$. To this end, we have
\begin{align*}
\E\bbrace*{\abs*{g_1(D_i)}^{2+\delta}} &= \E\fence*{\abs*{f(D_i)-\E\bbrace*{f(D_i)}}^{2+\delta}}\\
&\leq 2^{(2+\delta)-1}\fence*{\E\bbrace*{\abs*{f(D_i)}^{2+\delta}} + \abs*{\E\bbrace*{f(D_i)}}^{2+\delta}}\\
&\leq 2^{(2+\delta)}\E\bbrace*{\abs*{f(D_i)}^{2+\delta}}\\
&\leq 2^{(2+\delta)}\E\bbrace*{\abs*{f_1(D_i) + f_2(D_i)}^{2+\delta}}\\
&\leq 2^{(3+2\delta)}\fence*{\E\bbrace*{\abs*{f_1(D_i)}^{2+\delta}} + \E\bbrace*{\abs*{f_2(D_i)}^{2+\delta}}}.
\end{align*}
For the first summand in the parentheses, it holds that
\begin{align*}
\E\bbrace*{\abs*{f_1(D_i)}^{2+\delta}} &\leq \EE\parr*{|Z_{ij}|^{2+\delta}}\\
&= \EE\parr*{\abs*{X_{ik}-X_{jk}}^{2+\delta}\abs*{V_i-V_j}^{2+\delta}}\\
&\leq \bbrace*{\EE\parr*{\abs*{X_{ik}-X_{jk}}^{4+2\delta}}}^{1/2}\bbrace*{\EE\parr*{\abs*{V_i-V_j}^{4+2\delta}}}^{1/2}\\
&\leq 2^{2+\delta}\bbrace*{\sqrt{4+2\delta}\parr*{\sqrt{2}\kappa_x}}^{2+\delta}\E\parr*{\abs*{V_i}^{4+2\delta}}^{1/2},
\end{align*}
where in the last step we use the sub-Gaussianity of $(X_i-X_j)$ under the product measure and finiteness of $(4+2\delta)$th order moment of $V_i$. This concludes that $\sigma^2$ is upper bounded by an absolute constant. Choose $t \asymp \{\log(np)/n\}^{1/2}$, then under the scaling $\theta_n^2(\log p)^2(\log n)^4 = O(n)$, it holds that 
\begin{align}
\label{eq:tail_1_g}
\P\fence*{\abs*{\frac{1}{n}\sum_{i=1}^n g_1(D_i)}>c_1\bbrace*{\frac{\log(np)}{n}}^{1/2}} \leq \exp\bbrace*{-c_2\log(np)}.
\end{align}
Next, consider the tail bound of the degenerate V-statistic generated by the canonical kernel $g_2(D_i,D_j)$. 
To this end, for arbitrary positive constants $t$ and $B_W$, Lemma \ref{lemma:expansion}(b) verifies Condition (A) with $\cal{C} = \Big\{[-B_W,B_W]\times (-\infty, +\infty)\times (-\infty, +\infty)\Big\}^2$ and constants:
\begin{align*}
F(t)\asymp h^{-1} \asymp n^{1/2}(\log p)^{-1/2}, \quad B(t) \asymp \sqrt{\log p}\cdot\theta_n, \quad \mu_a(t) \asymp 1
\end{align*}
for all $a\geq 1$. By choosing $B_W$ to be large enough, $t'$ defined in \eqref{eq:bias} in the main paper satisfies $t'\asymp t$. Applying the fully degenerate version of Proposition \ref{prop:bern_U_alpha_general} with $x\asymp t^\prime \asymp \bbrace*{(\log np)/n}^{1/2}$, then we have
\begin{align}
\label{eq:tail_3_g}
\P\fence*{\abs*{n^{-2}\sum_{i,j=1}^ng_2(D_i,D_j)}\geq c_1\bbrace*{\frac{\log(np)}{n}}^{1/2}} \leq \text{exp}\bbrace*{-c_2\log (np)}
\end{align}
under the scaling $n\geq \theta_n^2(\log p)^2(\log n)^4$. Combining \eqref{eq:tail_1_g} - \eqref{eq:tail_3_g}, we have
\begin{align*}
\P\fence*{\abs*{\widetilde{V}_{n,k} - \widetilde{\theta}_k}\geq c_1\bbrace*{\frac{\log(np)}{n}}^{1/2}} \leq \text{exp}\bbrace*{-c_2\log(np)}.
\end{align*}
\textbf{Step 2:} Next we calculate the difference between $\widetilde{\theta}_k$ and $\theta_k$. By definition and symmetry, 
\begin{align*}
\abs*{\widetilde{\theta}_k - \theta_k} &= \EE\bigg[\frac{1}{h}K\parr*{\frac{W_i-W_j}{h}}(X_{ik}-X_{jk})(V_i-V_j)\cdot \mathbbm{1}\bbrace*{\abs*{X_i}\geq \kappa_1\sqrt{\log(np)}}\cdot\\
& \ms\mathbbm{1}\bbrace*{|X_j|\geq \kappa_1\sqrt{\log(np)}}\cdot\mathbbm{1}\bbrace*{|V_i|\geq \theta_n}\cdot\mathbbm{1}\bbrace*{|V_i|\geq \theta_n}\bigg]\\
&\leq 2\EE\bbrace*{\frac{1}{h}K\parr*{\frac{W_i-W_j}{h}}Z_{ij}\mathbbm{1}\bbrace*{|V_i|\geq \theta_n}} + \\
&\ms2\EE\fence*{\frac{1}{h}K\parr*{\frac{W_i-W_j}{h}}Z_{ij}\mathbbm{1}\bbrace*{|X_i|\geq \kappa_1\sqrt{\log(np)}}}.
\end{align*}
For the first term, applying Lemma A3.2 in \cite{Han2018SupplementT} with $Z = Z_{ij}\mathbbm{1}\bbrace*{|V_i|\geq \theta_n}$ and upper bound of $\ff_{\WW}(0)$, it holds that
\begin{align*}
&\EE\fence*{\frac{1}{h}K\parr*{\frac{\WW}{h}}\tilde{X}_{ij,k}\VV\mathbbm{1}\bbrace*{|V_i|\geq \theta_n}}\\
&\leq M_W\EE\fence*{\widetilde{X}_{ij,k}\VV\mathbbm{1}\bbrace*{|V_i|\geq \theta_n}\mid \WW=0} + M_WM_KC_0\EE\fence*{\abs*{\widetilde{X}_{ij,k}\VV}\mathbbm{1}\bbrace*{|V_i\geq\theta_n|}},
\end{align*}
where
\begin{align*}
\EE\fence*{\widetilde{X}_{ij,k}\VV\mathbbm{1}\bbrace*{|V_i|\geq\theta_n}\mid\WW=0} = \EE\fence*{\widetilde{X}_{ij,k}\bbrace*{g(W_i)-g(W_j)}\mathbbm{1}\bbrace*{|V_i|\geq\theta_n}\mid\WW=0} = 0
\end{align*}
and
\begin{align*}
\EE\fence*{\abs*{\widetilde{X}_{ij,k}\VV}\mathbbm{1}\bbrace*{|V_i|\geq\theta_n}} &\leq \parr*{\EE\fence*{\widetilde{X}_{ij,k}^2\VV^2}}^{1/2}\cdot\bbrace*{\P\parr*{|V_i|\geq \theta_n}}^{1/2}\\
&\leq \bbrace*{\EE\parr*{\widetilde{X}_{ij,k}^4}}^{1/4}\cdot\bbrace*{\EE\parr*{\VV^4}}^{1/4}\parr*{\P\parr*{|V_i|\geq\theta_n}}^{1/2}\\
&\leq \parr*{12\kappa_x^4}^{1/4}\cdot 2\E\parr*{|V_i|^4}^{1/4}\cdot \bbrace*{\theta_n^{-(4+2\delta)}\E\parr*{|V_i|^{4+2\delta}}}^{1/2}\\
&= 2(12)^{1/4}\kappa_x\bbrace*{\E\parr*{|V_i|^4}}^{1/4}\parr*{\E\fence*{|V_i|^{4+2\delta}}}^{1/2}\cdot \theta_n^{-(2+\delta)},
\end{align*}
where in the last inequality we use Markov's inequality and the finitenes of the $(4+2\delta)$th moment of $V_i$. Choosing
\begin{align}
\label{eq:theta_trunc_2}
\theta_n \asymp n^{\frac{1}{4+2\delta}},
\end{align}
then it holds that
\begin{align*}
\EE\fence*{\abs*{\widetilde{X}_{ij,k}\VV}\mathbbm{1}\bbrace*{|V_i|\geq\theta_n}} \leq c\bbrace*{\frac{\log(np)}{n}}^{1/2}.
\end{align*}
Similarly, by choosing $\kappa_1$ to be large enough, we obtain
\begin{align*}
\EE\fence*{\frac{1}{h}K\parr*{\frac{W_i-W_j}{h}}Z_{ij}\mathbbm{1}\bbrace*{|X_{ik}|\geq \frac{\kappa_0}{2}\sqrt{\log(np)}}} \leq c\bbrace*{\frac{\log(np)}{n}}^{1/2}.
\end{align*}
Put together the pieces gives
\begin{align*}
\abs*{\widetilde{\theta}_k - \theta_k} \leq c\bbrace*{\frac{\log(np)}{n}}^{1/2},
\end{align*}
which holds for all triplets $\{X_{ik}, V_i,W_i\}_{i=1}^n$, $k\in[p]$.\\
\textbf{Step 3:} With the definition 
\begin{align*}
\cal{A}_k \define \bigcap_{i=1}^n\bbrace*{|X_{ik}|\leq \kappa_1\sqrt{\log(np)} }, \quad k\in[p], \quad \cal{B} \define \bigcap_{i=1}^n\bbrace*{|V_i|\leq \theta_n},
\end{align*}
and $\cal{A}_k^c$ and $B^c$ as their complements, we have
\begin{align*}
\P\parr*{\cal{A}_k^c} &\leq n\cdot\P\bbrace*{|X_{1k}|\geq \kappa_1\sqrt{\log(np)}} \leq n\cdot\text{exp}\bbrace*{-\frac{\kappa_1^2\log(np)}{2\kappa_x^2}}\leq \exp\bbrace*{-c\log(np)}
\end{align*}
for sufficiently large $\kappa_1$. Taking union bound then gives
$\P\parr*{\bigcup_{k=1}^p\cal{A}_k^c} \leq \exp\bbrace*{-c\log(np)}$. Moreover, we have
\begin{align*}
\P\parr*{\cal{B}^c} \leq n\cdot \P\parr*{|V_1|\geq\theta_n} \leq n\cdot \frac{\E\parr*{|V_i|^{4+4\delta}}}{\theta_n^{4+4\delta}} \leq n^{-\frac{\delta}{2+\delta}}
\end{align*}
for arbitrarily small positive number $\alpha$ under the truncation level \eqref{eq:theta_trunc_2}. Lastly, putting together the pieces, we have
\begin{align*}
&\ms\P\fence*{\max_{k\in[p]}\abs*{V_{n,k}-\theta_k} \geq c_1\bbrace*{\frac{\log(np)}{n}}^{1/2}}\\
&\leq \P\fence*{\max_{k\in[p]}\abs*{V_{n,k}-\theta_k} \geq c_1\bbrace*{\frac{\log(np)}{n}}^{1/2}\bigcap \cal{A}_1\bigcap \ldots\bigcap \cal{A}_p \bigcap \cal{B}} + \P\parr*{\bigcup_{k=1}^p\cal{A}_k^c} + \P\parr*{\cal{B}^c}\\
&\leq \P\fence*{\max_{k\in[p]}\abs*{\widetilde{V}_{n,k}-\theta_k}\geq c_1\bbrace*{\frac{\log(np)}{n}}^{1/2}} + \exp\bbrace*{-c_2\log(np)} + n^{-\frac{\delta}{2+\delta}}\\
&\leq \P\fence*{\max_{k\in[p]}\fence*{\abs*{\widetilde{V}_{n,k}-\tilde{\theta}_k} + \abs*{\tilde{\theta}_k-\theta_k}}\geq c_1\bbrace*{\frac{\log(np)}{n}}^{1/2}} + \exp\bbrace*{-c\log(np)} + n^{-\frac{\delta}{2+\delta}}\\
&\leq \P\fence*{\max_{k\in[p]}\abs{\widetilde{V}_{n,k}-\widetilde{\theta}_k} \geq c_1\bbrace*{\frac{\log(np)}{n}}^{1/2}} + \exp\bbrace*{-c_2\log(np)} + n^{-\frac{\delta}{2+\delta}}\\
&\leq p\cdot\P\fence*{\abs{\widetilde{V}_{n,1}-\tilde{\theta}_1} \geq c_1\bbrace*{\frac{\log(np)}{n}}^{1/2}} + \text{exp}\bbrace*{-c_2\log(np)} + n^{-\frac{\delta}{2+\delta}}\\
&\leq \text{exp}\bbrace*{-c_2\log(np)} + n^{-\frac{\delta}{2+\delta}}.
\end{align*}
This completes the proof.
\end{proof}

\begin{lemma}
\label{lemma:perturbation_indep}
Let $D_i = (X_i, V_i, W_i)\in\RR^p\times\RR\times\RR$ be generated from a stationary sequence for $i = 1,\ldots, n$, satisfying a geometric $\alpha$-mixing condition with coefficient $\alpha(i)\leq \exp(-\gamma i)$ for all $i\geq 1$ and some positive $\gamma$. Let $K(\cdot)$ be a positive density kernel such that $\int K(u)du = 1$ and 
\begin{align*}
\max\bbrace*{\int |u|K(u)du,~\sup_{u\in\RR}K(u)}\leq M_K
\end{align*}
for some positive constant $M_K$. Assume that conditional on $W_i$ and unconditionally, $X_{ik}$ is sub-Gaussian with constants $\kappa_x^2$ for all $k\in[p]$. $\{V_i\}_{i=1}^n$ is assumed to be independent of the sequence $\{X_i,W_i\}_{i=1}^n$ with mean zero and finite $(2+2\delta)$th moment for some positive $\delta$. Assume that there exists some positive absolute constant $M_W$, such that
\begin{align*}
\max\bbrace*{\abs*{\frac{\partial f_{W\mid X}(w,x)}{\partial w}}, f_{W\mid X}(w,x)}\leq M_W
\end{align*}
for any $(w,x)$ such that the densities are defined. Take $h_n\geq K_1\{\log(np)/n\}^{1/2}$ for some positive absolute constant $K_1$, and further assume that $h_n \leq C_0$ for some constant $C_0$. Consider the V-statistic
\begin{align*}
V_{n,k} \define n^{-2}\sum_{i,j=1}^n \frac{1}{h}K\parr*{\frac{W_i-W_j}{h}}(X_{ik}-X_{jk})(V_i-V_j)
\end{align*}
with corresponding concentration parameter
\begin{align*}
\theta_k \define \EE\bbrace*{\frac{1}{h}K\parr*{\frac{W_i-W_j}{h}}(X_{ik}-X_{jk})(V_i-V_j).
}
\end{align*}
Then under the scaling 
\begin{align*}
n = \Omega\bbrace*{(\log p)^{\frac{4+2\delta}{\delta}}(\log n)^{\frac{8+4\delta}{\delta}}},
\end{align*}
the following holds:
\begin{align*}
\max_{k\in[p]}\abs*{V_{n,k} - \theta_k} \leq c\bbrace*{\frac{\log(np)}{n}}^{1/2}
\end{align*}
with probability at least $1-\exp\bbrace*{-c'\log(np)} - n^{-\frac{\delta}{2+\delta}}$ for some absolute positive constants $c_1,c_2$.
\end{lemma}
\begin{proof}
Throughout the proof, $c,c_1,c_2$ will be generic constants that do not depend on $n$. Let $\{D_i\}_{i=1}^n = \{W_i,X_{ik},V_i\}_{i=1}^n$ for each component $X_k$ of $X$. First define the following quantities for the truncated version of the V-statistic:
\begin{align*}
&\widetilde{h}_k(D_i,D_j) \define \frac{1}{h}K\parr*{\frac{W_i-W_j}{h}}(X_{ik}-X_{jk})(V_i-V_j)M_1(X_{ik})M_1(X_{jk})M_2(V_i)M_2(V_j),\\
&\widetilde{V}_{n,k} \define {n\choose 2}^{-1}\sum_{i<j}\widetilde{h}_k(D_i, D_j), \quad \widetilde{\theta}_k \define \EE\bbrace*{\widetilde{h}_k\parr*{D_i, D_j}},
\end{align*}
where $M_1(t) = \mathbbm{1}\bbrace*{|t|\leq \kappa_1\sqrt{\log(np)}}, M_2(t) = \mathbbm{1}\bbrace*{|t|\leq \theta_n}$, and $\kappa_1$, $\theta_n$ are two truncation constants to be specified later. Further define $Z_{i} \define (X_{ik}, V_i)$ and $Z_{ij} \define (X_{ik} - X_{jk})(V_i-V_j)$. \\
\textbf{Step 1:} Upper bound the truncated V-statistic.\\
Via Hoeffding decomposition, we have
\begin{align*}
\widetilde{V}_{n,k} - \widetilde{\theta}_k = \frac{2}{n}\sum_{i=1}^n g_1(D_i) + n^{-2}\sum_{i,j=1}^ng_2(D_i,D_j),
\end{align*}
where $g_1(D_i) = \E^j\bbrace*{\widetilde{h}_k(D_i,D_j)}-\widetilde{\theta}_k, g_2(D_i,D_j) = \widetilde{h}_k(D_i,D_j) - \widetilde{\theta}_k - g_1(D_i) - g_1(D_j)$. It can be readily checked that $g_1$ and $g_2$ are both canonical kernels. Therefore for any $t>0$, it holds that
\begin{align*}
\P\parr*{\abs*{\tilde{V}_{n,k}-\tilde{\theta}_k} > t} \leq \P\bbrace*{\abs*{\frac{1}{n}\sum_{i=1}^ng_1(D_i)}\geq \frac{t}{4}} + \P\bbrace*{\abs*{n^{-2}\sum_{i,j=1}^ng_2(D_i,D_j)}\geq \frac{t}{2}}.
\end{align*}
We now upper bound the two summands separately. Define $f(D_i) = \E^j\bbrace*{\widetilde{h}_k(D_i,D_j)}$ so that $g_1(D_i) = f(D_i) - \widetilde{\theta}_k = f(D_i) - \E\bbrace*{f(D_i)}$. Appling Lemma A3.3 in \cite{Han2018SupplementT} with $M_1 = M_W$ and $M_2 = M_K$, we have
\begin{align*}
|f(D_i) - f_1(D_i)| \leq f_2(D_i),
\end{align*}
where, with definition $\varphi(Z_i,Z_j) = Z_{ij}M_1(X_i)M_1(X_j)M_2(V_i)M_2(V_j)$, 
\begin{equation}
\begin{aligned}
\label{eq:f1f2_indep}
f_1(D_i) &= \E^j\fence*{\varphi(Z_i,Z_j)\mid W_i = W_j}\cdot f_{W_j}(W_i)\leq M_W\E^j\fence*{\varphi(Z_i,Z_j)\mid W_i = W_j},\\
f_2(D_i) &= M_WM_Kh\E^j\fence*{|\varphi(Z_i,Z_j)|} \leq C_0M_WM_K\E^j\fence*{|\varphi(Z_i,Z_j)|}.
\end{aligned}
\end{equation}
Therefore, we have
\begin{align*}
\|f(D_i)\|_\infty = O(\bbrace*{\log(np)}^{1/2}\theta_n).
\end{align*}
Furthermore, we have
\begin{equation}
\begin{aligned}
\label{eq:holder_1}
\E\bbrace*{f_1^2(D_i)} &\leq M_W^2\EE\parr*{Z_{ij}^2}\\
&= M_W^2\EE\bbrace*{(X_{ik}-X_{jk})^2(V_i-V_j)^2}\\
&= M_W^2 \EE\bbrace*{(X_{ik}-X_{jk})^2}\EE\bbrace*{(V_i-V_j)^2}\\
&\leq 4M_W^2\kappa_x^2\cdot \E\parr*{V_i^2}.
\end{aligned}
\end{equation}
where in the third line we use independence of $\{X_i\}_{i=1}^n$ and $\{V_i\}_{i=1}^n$, in the fourth line we use the fact that under the product measure, $(X_i-X_j)$ is sub-Gaussian with constants at most $2\kappa_x^2$  and finite second moment of $V_i$. Similarly, we also have 
\begin{align*}
\E\bbrace*{f_2^2(D_i)} \leq 4C_0^2M_W^2M_K^2\kappa_x^2\E\parr*{V_i^2}.
\end{align*}
Therefore, 
\begin{equation*}
\begin{aligned}
\E\bbrace*{f^2(D_i)} &= \E\fence*{\bbrace*{f(D_i) - f_1(D_i) + f_1(D_i)}^2}\\
&\leq 2\E\bbrace*{f_2^2(D_i)} + 2\E\bbrace*{f_1^2(D_i)}\\
&\leq 8\kappa_x^2M_W^2(1+M_K^2C_0^2)\E\parr*{V_i^2}.
\end{aligned}
\end{equation*}
Since $\{D_i\}_{i=1}^n$ is geometrically $\alpha$-mixing, so is $\{g(D_i)\}_{i=1}^n$. Then, by Lemma \ref{lemma:sample_mean_Bern}, it holds for any $t > 0$ that
\begin{align*}
\P\bbrace*{\abs*{\frac{1}{n}\sum_{i=1}^n g_1(D_i)}>\frac{t}{4}} &= \P\fence*{\abs*{\frac{1}{n}\sum_{i=1}^n f(D_i) - \frac{1}{n}\sum_{i=1}^n\E\bbrace*{f(D_i)}}\geq \frac{t}{4}}\\
&\leq \exp\bbrace*{-\frac{C_3n^2t^2}{n\sigma^2+B^2+ntB(\log n)^2}},
\end{align*}
where $B = \|g_1(D_i)\|_\infty \leq 2\|f(D_i)\|_\infty = O(\sqrt{\log(np)}\theta_n)$, and 
\begin{align*}
\sigma^2 = \var\bbrace*{g_1(D_1)} + 2\sum_{i>1}^\infty \abs*{\cov\bbrace*{g_1(D_1),g_1(D_i)}}.
\end{align*}
By the previous calculation, it holds that the first summand $\var\bbrace*{g_1(D_1)}\leq \E\bbrace*{g_1^2(D_1)}$ is upper bounded by an absolute constant. Thus, in order to show that $\sigma^2$ is also bounded by an absolute constant, it suffices to show that the second summation is also bounded. Applying Lemma \ref{lemma:alpha_covariance} with $p = q = 2+\delta$ and $r = (2+\delta)/\delta$ for some positive number $\delta$, it holds that
\begin{align*}
\sum_{j=1}^\infty \abs*{\text{Cov}(g_1(D_1),g_1(D_{1+j}))} \leq \bbrace*{\sum_{j=1}^\infty \alpha^{\delta/(2+\delta)}(j)}\fence*{\E\bbrace*{\abs*{g_1(D_1)}^{2+\delta}}}^{2/(2+\delta)}.
\end{align*}
By the condition on the $\alpha$-mixing coefficient, it suffices to show that $\E\bbrace*{|g_1(D_1)|^{2+\delta}}$ is bounded. To this end, we have
\begin{align*}
\E\bbrace*{\abs*{g_1(D_i)}^{2+\delta}} &= \E\fence*{\abs*{f(D_i)-\E\bbrace*{f(D_i)}}^{2+\delta}}\\
&\leq 2^{1+\delta}\fence*{\E\bbrace*{\abs*{f(D_i)}^{2+\delta}} + \abs*{\E\bbrace*{f(D_i)}}^{2+\delta}}\\
&\leq 2^{(2+\delta)}\E\bbrace*{\abs*{f(D_i)}^{2+\delta}}\\
&\leq 2^{2+\delta}\E\bbrace*{\abs*{f_1(D_i) + f_2(D_i)}^{2+\delta}}\\
&\leq 2^{3+2\delta}\fence*{\E\bbrace*{\abs*{f_1(D_i)}^{2+\delta}} + \E\bbrace*{\abs*{f_2(D_i)}^{2+\delta}}}.
\end{align*}
For the first summand in the parentheses, it holds that
\begin{align*}
\E\bbrace*{\abs*{f_1(D_i)}^{2+\delta}} &\leq \EE\parr*{\abs*{Z_{ij}}^{2+\delta}}\\
&= \EE\parr*{\abs*{X_i-X_j}^{2+\delta}\abs*{V_i-V_j}^{2+\delta}}\\
&= \EE\parr*{\abs*{X_i-X_j}^{2+\delta}}\EE\parr*{\abs*{V_i-V_j}^{2+\delta}}\\
&\leq 2^{2+\delta}\bbrace*{\sqrt{2+\delta}\parr*{\sqrt{2}\kappa_x}}^{2+\delta}\E\parr*{|V_i|^{2+\delta}},
\end{align*}
where in the third line we use the independence of $\{V_i\}_{i=1}^n$ and $\{X_i\}_{i=1}^n$ and in the last line we use the sub-Gaussianity of $(X_i-X_j)$ under the product measure and finiteness of $(2+\delta)$th order moment of $V_i$. This concludes that $\sigma^2$ is upper bounded by an absolute constant. Choose $t \asymp \{\log(np)/n\}^{1/2}$, then under the scaling $\theta_n^2(\log p)^2(\log n)^4 = O(n)$, it holds that 
\begin{align}
\label{eq:tail_1_indep}
\P\fence*{\abs*{\frac{1}{n}\sum_{i=1}^n g_1(D_i)}>c_1\bbrace*{\frac{\log(np)}{n}}^{1/2}} \leq \exp\bbrace*{-c_2\log(np)}.
\end{align}
Next, consider the tail bound of the degenerate V-statistic generated by the canonical kernel $g_2(D_i,D_j)$. 
To this end, for arbitrary positive constants $t$ and $B_W$, Lemma \ref{lemma:expansion}(c) verifies Condition (A) with $\cal{C} = \Big\{[-B_W,B_W]\times (-\infty, +\infty)\times (-\infty, +\infty)\Big\}^2$ and constants:
\begin{align*}
F(t)\asymp h^{-1} \asymp n^{1/2}(\log p)^{-1/2}, \quad B(t) \asymp \sqrt{\log p}\cdot\theta_n, \quad \mu_a(t) \asymp 1
\end{align*}
for all $a\geq 1$. By choosing $B_W$ to be sufficiently large, we have $t'$ defined in \eqref{eq:bias} in the main paper satisfies that $t'\asymp t$. Applying the fully degenerate version in Proposition \ref{prop:bern_U_alpha_general} with $x\asymp t^\prime \asymp \bbrace*{(\log np)/n}^{1/2}$, we have
\begin{align}
\label{eq:tail_3_indep}
\P\fence*{\abs*{n^{-2}\sum_{i,j=1}^n g_2(D_i,D_j)}\geq c_1\bbrace*{\frac{\log(np)}{n}}^{1/2}} \leq \exp\bbrace*{-c_2\log (np)}
\end{align}
under the scaling $n\geq \theta_n^2(\log p)^2(\log n)^4$. Combining \eqref{eq:tail_1_indep} - \eqref{eq:tail_3_indep} gives
\begin{align*}
\P\fence*{\abs*{\widetilde{V}_{n,k} - \widetilde{\theta}_k}\geq c_1\bbrace*{\frac{\log(np)}{n}}^{1/2}} \leq \exp\bbrace*{-c_2\log(np)}.
\end{align*}
\textbf{Step 2:} Next we calculate the difference between $\widetilde{\theta}_k$ and $\theta_k$ for $k\in[p]$. By definition and symmetry, 
\begin{align*}
\abs*{\widetilde{\theta}_k - \theta_k} &= \EE\bigg[\frac{1}{h}K\parr*{\frac{W_i-W_j}{h}}(X_{ik}-X_{jk})(V_i-V_j)\cdot \mathbbm{1}\bbrace*{\abs*{X_{ik}}\geq \kappa_1\sqrt{\log(np)}}\cdot\\
& \ms\mathbbm{1}\bbrace*{\bigcup |X_{jk}|\geq \kappa_1\sqrt{\log(np)}}\cdot\mathbbm{1}\bbrace*{|V_i|\geq \theta_n}\cdot\mathbbm{1}\bbrace*{|V_i|\geq \theta_n}\bigg]\\
&\leq 2\EE\fence*{\frac{1}{h}K\parr*{\frac{W_i-W_j}{h}}Z_{ij}\mathbbm{1}\bbrace*{|X_{ik}|\geq \kappa_1\sqrt{\log(np)}}} + \\
&\ms2\EE\fence*{\frac{1}{h}K\parr*{\frac{W_i-W_j}{h}}Z_{ij}\mathbbm{1}\bbrace*{|V_i|\geq \theta_n}}.
\end{align*}
For the first term, by independence of $\{V_i\}_{i=1}^n$ with $\{X_i,W_i\}_{i=1}^n$, it holds that
\begin{align*}
&\EE\fence*{\frac{1}{h}K\parr*{\frac{\WW}{h}}\parr*{X_{ik}-X_{jk}}\parr*{V_i-V_j}\mathbbm{1}\bbrace*{|X_{ik}|\geq \kappa_1\sqrt{\log(np)}}}\\
&= \EE\parr*{V_i-V_j}\EE\fence*{\frac{1}{h}K\parr*{\frac{\WW}{h}}\XX\mathbbm{1}\bbrace*{|X_{ik}|\geq \kappa_1\sqrt{\log(np)}}}\\
&= 0.
\end{align*}
For the second term, applying Lemma A3.2 in \cite{Han2018SupplementT} with $Z = Z_{ij}\mathbbm{1}\bbrace*{|V_i|\geq \theta_n}$, it holds that
\begin{align*}
&\EE\fence*{\frac{1}{h}K\parr*{\frac{\WW}{h}}\tilde{X}_{ij,k}\VV\mathbbm{1}\bbrace*{|V_i|\geq \theta_n}}\\
&\leq M_W\EE\fence*{\widetilde{X}_{ij,k}\VV\mathbbm{1}\bbrace*{|V_i|\geq \theta_n}\mid \WW=0} + M_WM_KC_0\EE\fence*{\abs*{\tilde{X}_{ij,k}\VV}\mathbbm{1}\bbrace*{|V_i\geq\theta_n|}},
\end{align*}
where
\begin{align*}
\EE\fence*{\widetilde{X}_{ij,k}\VV\mathbbm{1}\bbrace*{|V_i|\geq\theta_n}\mid\WW=0} = \EE\parr*{\tilde{X}_{ij,k}\mid\WW=0}\cdot \EE\fence*{\VV\mathbbm{1}\bbrace*{|V_i|\geq\theta_n}} = 0
\end{align*}
and
\begin{align*}
\EE\fence*{\abs*{\widetilde{X}_{ij,k}\VV}\mathbbm{1}\bbrace*{|V_i|\geq\theta_n}} &= \EE\parr*{\abs*{\widetilde{X}_{ij,k}}}\cdot\EE\fence*{\abs*{\VV}\mathbbm{1}\bbrace*{|V_i|\geq\theta_n}}\\
&\leq \bbrace*{\EE\parr*{\widetilde{X}_{ij,k}^2}}^{1/2}\cdot\bbrace*{\EE\parr*{\VV^2}}^{1/2}\bbrace*{\P\parr*{|V_i|\geq\theta_n}}^{1/2}\\
&\leq (2\kappa_x^2)^{1/2}\cdot 2\bbrace*{\E\parr*{|V_i|^2}}^{1/2}\cdot \bbrace*{\theta_n^{-(2+\delta)}\E\parr*{|V_i|^{2+\delta}}}^{1/2}\\
&= 2\sqrt{2}\kappa_x\bbrace*{\E\parr*{|V_i|^2}}^{1/2}\bbrace*{\E\parr*{|V_i|^{2+\delta}}}^{1/2}\cdot \theta_n^{-\frac{2+\delta}{2}},
\end{align*}
where in the last inequality we use Markov's inequality and the finitenes of the $(2+\delta)$th moment of $V_i$. Choose 
\begin{align}
\label{eq:theta_trunc}
\theta_n \asymp n^{\frac{1}{2+\delta}}
\end{align}
then it holds that
\begin{align*}
\EE\fence*{\abs*{\widetilde{X}_{ij,k}\VV}\mathbbm{1}\bbrace*{|V_i|\geq\theta_n}} \leq c\bbrace*{\frac{\log(np)}{n}}^{1/2}.
\end{align*}
Putting together the pieces gives for each $k\in[p]$,
\begin{align*}
\abs*{\widetilde{\theta}_k - \theta_k} \leq c\bbrace*{\frac{\log(np)}{n}}^{1/2}.
\end{align*}
\textbf{Step 3:} With the definition 
\begin{align*}
\cal{A}_k &\define \bigcap_{i=1}^n\bbrace*{|X_{ik}|\leq \kappa_1\sqrt{\log(np)} }, \quad k\in[p], \quad \cal{B} \define \bigcap_{i=1}^n\bbrace*{|V_i|\leq \theta_n},
\end{align*}
and $\cal{A}_k^c$ and $B^c$ as their complements, we have
\begin{align*}
\P\parr*{\cal{A}_k^c} &\leq n\cdot\P\bbrace*{|X_{ik}|\geq \kappa_1\sqrt{\log(np)}} \leq n\cdot\exp\bbrace*{-\frac{\kappa_1^2\log(np)}{2\kappa_x^2}} \leq \exp\bbrace*{-c_2\log(np)}
\end{align*}
for sufficiently large $\kappa_1$. Taking union bound gives
\begin{align*}
\P\parr*{\bigcup_{k=1}^p\cal{A}_k^c} \leq p\cdot\P\parr*{\cal{A}_1^c} \leq \text{exp}\bbrace*{-c_2\log(np)}.
\end{align*}
Moreover, we have
\begin{align*}
\P\parr*{\cal{B}^c} \leq n\cdot \P\parr*{|V_i|\geq\theta_n} \leq n\cdot \frac{\E\parr*{|V_i|^{2+2\delta}}}{\theta_n^{2+2\delta}} \leq n^{-\frac{\delta}{2+\delta}}
\end{align*}
for arbitrarily small positive number $\alpha$ under the truncation level \eqref{eq:theta_trunc}. Lastly, putting together the pieces, we have
\begin{align*}
&\ms\P\fence*{\max_{k\in[p]}\abs*{V_{n,k}-\theta_k} \geq c_1\bbrace*{\frac{\log(np)}{n}}^{1/2}}\\
&\leq \P\fence*{\max_{k\in[p]}\abs*{V_{n,k}-\theta_k} \geq c_1\bbrace*{\frac{\log(np)}{n}}^{1/2}\bigcap \cal{A}_1\bigcap \ldots\bigcap \cal{A}_p \bigcap \cal{B}} + \P\parr*{\bigcup_{k=1}^p\cal{A}_k^c} + \P\parr*{\cal{B}^c}\\
&\leq \P\fence*{\max_{k\in[p]}\abs*{\widetilde{V}_{n,k}-\theta_k}\geq c_1\bbrace*{\frac{\log(np)}{n}}^{1/2}} + \exp\bbrace*{-c_2\log(np)} + n^{-\frac{\delta}{2+\delta}}\\
&\leq \P\fence*{\max_{k\in[p]}\fence*{\abs*{\tilde{V}_{n,k}-\widetilde{\theta}_k} + \abs*{\widetilde{\theta}_k-\theta_k}}\geq c_1\bbrace*{\frac{\log(np)}{n}}^{1/2}} + \exp\bbrace*{-c_2\log(np)} + n^{-\frac{\delta}{2+\delta}}\\
&\leq \P\fence*{\max_{k\in[p]}\abs{\widetilde{V}_{n,k}-\widetilde{\theta}_k} \geq c_1\bbrace*{\frac{\log(np)}{n}}^{1/2}} + \text{exp}\fence*{-c\log(np)} + n^{-\frac{\delta}{2+\delta}}\\
&\leq p\cdot\P\fence*{\abs{\widetilde{V}_{n,1}-\tilde{\theta}_1} \geq c_1\bbrace*{\frac{\log(np)}{n}}^{1/2}} + \exp\bbrace*{-c_2\log(np)} + n^{-\frac{\delta}{2+\delta}}\\
&\leq \exp\bbrace*{-c_2\log(np)} + n^{-\frac{\delta}{2+\delta}}.
\end{align*}
This completes the proof.
\end{proof}

\begin{lemma}
\label{lemma:expansion}
Let $\cal{C} \define \Big\{[-B_W, B_W] \times (-\infty, +\infty) \times (-\infty, +\infty)\Big\}^2 \subset \RR^6$ for some positive constant $B_W$.\\
(a) Under the conditions of Lemma \ref{lemma:perturbation}, for any $t > 0$, the kernel $\widetilde{h}(D_i,D_j)$ defined in its proof satisfies Condition (A) with symmetric set $\cal{C}$ and constants:
\begin{align*}
F(t, \cal{C}) \asymp h^{-1}, \quad B(t, \cal{C}) \asymp \log p, \quad \mu_a(t, \cal{C}) \asymp 1
\end{align*}
for all $a\geq 1$.
\\
(b) Under the conditions of Lemma \ref{lemma:perturbation_g},  for any $t > 0$, the kernel $\widetilde{h}(D_i,D_j)$ defined in its proof satisfies Condition (A)  with symmetric set $\cal{C}$ and constants:
\begin{align*}
F(t, \cal{C}) \asymp h^{-1}, \quad B(t, \cal{C}) \asymp \sqrt{\log(np)}\cdot\theta_n, \quad \mu_a(t, \cal{C}) \asymp 1
\end{align*}
for all $1\leq a\leq 4+2\delta$, with $\delta$ being the constant in Lemma \ref{lemma:perturbation_g}.
\\
(c) Under the conditions of Lemma \ref{lemma:perturbation_indep}, for any $t > 0$, the kernel $\widetilde{h}(D_i,D_j)$ defined in its proof satisfies Condition (A) symmetric set $\cal{C}$ and constants:
\begin{align*}
F(t, \cal{C}) \asymp h^{-1}, \quad B(t, \cal{C}) \asymp \sqrt{\log(np)}\cdot\theta_n, \quad \mu_a(t, \cal{C}) \asymp 1\end{align*}
for all $1\leq a\leq 2+\delta$, with $\delta$ being the constant in Lemma \ref{lemma:perturbation_indep}.
\end{lemma}

\begin{proof}
We will only prove part (a); the proofs for part (b) and (c) are essentially the same. Throughout the proof, $c$ will be a generic constant, $D \define (W,X,V)$. Recall that $\widetilde{h}$ takes the form
\begin{align*}
\widetilde{h}(D_i,D_j) = \frac{1}{h}K\parr*{\frac{W_i-W_j}{h}}(X_i-X_j)(V_i-V_j)M_1(X_i)M_1(X_j)M_1(V_i)M_1(V_j),
\end{align*}
where $M_1(t) = \mathbbm{1}\bbrace*{|t|\leq \sqrt{\log(np)}}$ for some $\kappa_1 > 0$ to be chosen. Note that $\widetilde{h}$ can be written as $\widetilde{h} = \frac{1}{h}\cdot f^1\cdot f^2$, where
\begin{align*}
f^1(D_i,D_j) &= K\parr*{\frac{W_i-W_j}{h}},\\
f^2(D_i,D_j) &= (X_i-X_j)(V_i-V_j)M_1(X_i)M_1(X_j)M_1(V_i)M_1(V_j).
\end{align*}
It can be readily checked that $f^2$ has upper bound $M_2\asymp \log(np)$. Since $K\parr*{\frac{W_i-W_j}{h}}$ is continuous, shift-invariant and PD, conditions of Corollary \ref{cor:random_fourier_alpha} are satisfied. Letting $B_W \define p^{\kappa_2}$ for sufficiently large $\kappa_2$ and $t_1\asymp t/M_2$ in Corollary \ref{cor:random_fourier_alpha}, we obtain that there exists a kernel $\widetilde{f}^1$ such that $\|f^1-\widetilde{f}^1\|_\infty\leq t_1$ and $\widetilde{f}^1$ satisfies Condition (A) with symmetric set $\cal{C}$ and constants
\begin{align*}
F = 2K(0),\quad B = 1,\quad \mu_a = 1
\end{align*}
for all $a\geq 1$. Also note that $f^2$ can be expanded in closed form as
\begin{align*}
e_1(D_i)e_4(D_j) - e_2(D_i)e_3(D_j) - e_3(D_i)e_2(D_j) + e_4(D_i)e_1(D_j),
\end{align*}
where the basis $\{e_j(D)\}_{j=1}^4$ are defined by
\begin{align*}
&e_1(D) = M_1(X)M_1(V)XV\\
&e_2(D) = M_1(X)M_1(V)X\\
&e_3(D) = M_1(X)M_1(V)V\\
&e_4(D) = M_1(X)M_1(V).
\end{align*}
Note that $\sup_{1\leq j\leq 4}\|e_j\|_\infty \leq \kappa_1^2\log(np)$. Moreover, by the sub-Gaussianity of $X$ and $V$, we have
\begin{align*}
\sup_{1\leq j\leq 4}\bbrace*{\E\parr*{e_j(D)^a}}^{1/a} \leq \bbrace*{\E\parr*{|X|^{2a}}}^{1/2a}\bbrace*{\E\parr*{|V|^{2a}}}^{1/2a} \leq c\kappa_x\kappa_va
\end{align*}
for all $a\geq 1$. Thus $f^2$ satisfies Condition (A) with symmetric set $\cal{C}$ and constants:
\begin{align*}
F = 4, \quad B = \kappa_1^2\log(np), \quad \mu_a = c\kappa_x\kappa_va
\end{align*}
for all $a\geq 1$. Following the proof of the second part of Proposition \ref{prop:stability_sum_alpha} with $t_1\asymp t/M_2$ and $t_2 = 0$, we obtain that
$\check{h}$ has an approximating kernel that satisfies Condition (A) with symmetric set $\cal{C}$ and constants
\begin{align*}
F \asymp \frac{1}{h},\quad B \asymp \log(np), \quad \mu_a \asymp 1
\end{align*}
for all $a\geq 1$. This completes the proof.
\end{proof}

\section{Proofs of results in Section \ref{subsec:testing}}
\label{proof:testing}

\subsection{Proof of Corollary \ref{cor:MDP_U}}
\begin{proof}
Part (a) follows directly from the proof of Proposition \ref{prop:MDP_U} and the tail bound in Corollary \ref{cor:continuous_general}. Part (b) follows from the proof of Proposition \ref{prop:MDP_U} and Part (a) of Corollary \ref{cor:discontinuous_tail} in the main paper.
\end{proof}

\subsection{Proof of Theorem \ref{thm:U_test}}
\begin{proof}
Applying Theorem 1 in \cite{arratia1989two} with $I = [p]$ and $B_\alpha = \{\alpha\}$, we obtain
\begin{align*}
\abs*{\P\parr*{S_n\leq \sqrt{2\log p-\log\log p + q_\alpha}} - \text{exp}(-\lambda_n)} \leq b_{n,1} + b_{n,2} + b_{n,3},
\end{align*}
where 
\begin{align*}
\lambda_n = \sum_{\ell=1}^p \P\parr*{\abs*{\tilde{U}_{n,\ell}}\geq \sqrt{2\log p - \log\log p + q_\alpha}}.
\end{align*}
Due to the independence of the $p$ sequences, $b_{n,2} = b_{n,3} = 0$.  By Proposition \ref{prop:MDP_U}, it holds that
\begin{align*}
\P\parr*{\abs*{\tilde{U}_{n,\ell}}\geq \sqrt{2\log p - \log\log p + q_\alpha}} \sim 1-\Phi(\sqrt{2\log p - \log\log p + q_\alpha})
\end{align*}
if $\sqrt{2\log p - \log\log p + q_\alpha} \leq \gamma\sqrt{\log n}$, or equivalently, $p = O(n^{\gamma^2/2})$. Using the fact that $1-\Phi(x) \sim 1/(\sqrt{2\pi}x)\text{exp}(-x^2/2)$, we obtain that 
\begin{align*}
\P\parr*{\abs*{\tilde{U}_{n,\ell}}\geq \sqrt{2\log p - \log\log p + q_\alpha}} \sim \frac{1}{p}\frac{1}{\sqrt{\pi}}\text{exp}(-\frac{1}{2}q_\alpha),
\end{align*}
and therefore $\text{exp}(-\lambda_n) \rightarrow 1-\alpha$. Lastly, since
\begin{align*}
b_{n,1} = p\cdot \parr*{\P\parr*{\abs*{\tilde{U}_{n,\ell}}\geq \sqrt{2\log p - \log\log p + q_\alpha}}}^2 \sim \lambda_n^2/p \rightarrow 0,
\end{align*}
we thus obtain
\begin{align*}
\P\parr*{S_n\leq \sqrt{2\log p-\log\log p + q_\alpha}} \rightarrow 1-\alpha
\end{align*}
as $n\rightarrow \infty$. This completes the proof.
\end{proof}

\subsection{Proof of Proposition \ref{prop:test_example}}
\begin{proof}
First, we have
\begin{align*}
\widetilde{V}_{n,\ell} = \frac{n-1}{n}\widetilde{U}_{n,\ell}  + \frac{1}{2\sqrt{n}\sigma_\ell}(h^\ell_0(0) - \theta_\ell),
\end{align*}
where $\widetilde{V}_{n,\ell} \define \sqrt{n}V_{n,\ell}/(2\sigma_\ell)$, with $V_{n,\ell}$ defined to be $n^{-2}\sum_{i,j=1}^nh^\ell(X^\ell_i,X^\ell_j)$, and $h^\ell_0(x) \define \sgn(x_1)\sgn(x_2)$. Therefore, with $h^\ell_0(0) = \theta_\ell = 0$ as calculated in Section \ref{subsubsec:test_example} in the main paper, it suffices to show that for any given $\gamma > 0$, $\widetilde{V}_{n,\ell}$ satisfies \eqref{eq:MDP} in the main paper uniformly over $x\in[0,\gamma\sqrt{\log n}]$. For this, we will make use of (b) of Corollary \ref{cor:MDP_U} by verifying its conditions. Apparently, $h^\ell$ is centered, nondegenerate and satisfies \eqref{eq:MDP_moment} in the main paper with any $\gamma > 0$. A truncated version of $h^\ell$ defined to be
\begin{align*}
\bar{h}^\ell(x,y) \define \bar{h}^\ell_0(x-y) \text{ and } \bar{h}^\ell_0(x) \define \sgn(x_1)\sgn(x_2)\mathbbm{1}\bbrace*{\abs*{x_1}\leq 2M_1}\mathbbm{1}\bbrace*{\abs*{x_2}\leq 2M_1},
\end{align*}
with $M_1$ being chosen later. Clearly, for any $M_2 > 0$, $h^\ell$ satisfies Condition (B4) in the main paper with $M_1,M_2$.

We now verify the rest of the conditions in (b) of Corollary \ref{cor:MDP_U}. 
Note that, by the calculation on Kendall's tau after Theorem \ref{thm:discontinuous_alpha} in the main paper, $B(t)\asymp \mu_a(t) \asymp 1$ and $F(t) \asymp \log^2(M_1/M_2) + \log^2\log(1/t)$ for sufficiently large $M_1$ and sufficiently small $M_2$ and $t$. Therefore, \eqref{eq:MDP_condition} in the main paper holds by directly calculation, and also $(\abs*{h^\ell_0(0)} + F(t))/n = o(1/(\sqrt{n}(\log n)^2))$. Moreover, by choosing sufficiently large $M_1$ and sufficiently small $M_2$ (depending only on $\gamma,\eta_1,\eta_2$) and using the fact that $J_1 = J_2 = 3$, we can show that
\begin{align*}
\bbrace*{n^2\parr*{\sum_{k=1}^2 J_k}M_2D} \vee \bbrace*{n\sum_{k=1}^2 \P\parr*{\abs*{X^\ell_{1,k}}\geq M_1}} = o(n^{-\gamma^2/2}/\sqrt{\log n}).
\end{align*}
Lastly, we show that with $t \asymp 1/(\sqrt{n}(\log n)^2)$, we can properly choose $M_1$ and $M_2$ such that $t' = t'(t, M_1,M_2)$ defined in \eqref{eq:bias} with $\cal{C}$ in \eqref{eq:discontinuous_support} satisfies $t' = O(1/(\sqrt{n}(\log n)^2))$. By definition of $t'$, it suffices to show that $\max_{0\leq i\leq m-1}s_iv_i = O(1/(\sqrt{n}(\log n)^2))$ and $\max_{0\leq i\leq m-1}Fv_i = O(1/(\sqrt{n}(\log n)^2))$, with $\{v_i\}_{i=0}^{m-1}$ and $\{s_i\}_{i=0}^{m-1}$ defined before Proposition \ref{prop:hoeffding} in the main paper. Since $\bar{h}^\ell$ is upper bounded in absolute value by $1$, the sequence $\{s_i\}_{i=0}^{m-1}$ can be uniformly upper bounded by $1$. For $v_0$, we have
\begin{align*}
v_0^2 &\leq \P\parr*{(\widetilde{X}^\ell_1,\widetilde{X}^\ell_2)\notin [-M_1,M_1]^{4}} + \P\parr*{(\widetilde{X}^\ell_1,\widetilde{X}^\ell_2)\notin \Big\{(x,z)\in\RR^{4}: \abs*{(x_j-z_j) - y_{j,k}}\geq M_2, j\in[2],k\in[J_j]\Big\}}\\
&\leq 2\sum_{k=1}^2 \P\parr*{\abs*{\widetilde{X}^\ell_{1,k}}\geq M_1} +2\parr*{\sum_{k=1}^2 J_k}M_2D,
\end{align*}
where we have used the upper bound $D$ on the density of $\widetilde{X}^\ell_{1,k} - \widetilde{X}^\ell_{2,k}$ uniformly over $k\in[2]$. For $v_1$, we have
\begin{align*}
v_1^2 &\leq \P\parr*{\widetilde{X}^\ell_1\notin [-M_1,M_1]^d} + \sup_{x\in[-M_1,M_1]^2}\P\parr*{\widetilde{X}^\ell_1\notin \Big\{z\in\RR^2: \abs*{z_j - (x_j-y_{j,k})}\geq M_2, j\in[2],k\in[J_j]\Big\}}\\
&\leq \sum_{k=1}^2 \P\parr*{\abs*{\widetilde{X}^\ell_{1,k}}\geq M_1} + 2\parr*{\sum_{k=1}^2 J_k}M_2D,
\end{align*}
where we have used the upper bound $D$ on the density of $\widetilde{X}^\ell_{1,k}$ uniformly over $k\in[2]$. Therefore, under the assumption that $X^\ell_{1,k}$ has finite $\eta_1$th moment for all $k\in[2]$ and $D = D(n) = O(n^{\eta_2})$, by choosing $M_1 = n^{\kappa_1}$ and $M_2 = n^{-\kappa_2}$ for sufficiently large $\kappa_1$ and $\kappa_2$ (depending only on $\eta_1,\eta_2$), it holds that $t' = O(1/(\sqrt{n}(\log n)^2))$.

This completes the proof.
\end{proof}


\section{Proofs of results in Section \ref{sec:discussion}}
\label{proof:discussion}

\subsection{Proof of Proposition \ref*{prop:bern_U_tau}}
\begin{proof}
We essentially follow the same proof of Proposition \ref{prop:bern_U_alpha}, and it suffices to justify that for each $j$, the sequence $\{\widetilde{e}_j(X_i)\}_{i=1}^n$ is still geometrically $\tau$-mixing so we can still apply Lemma \ref{lemma:sample_mean_Bern}. By the second part of Lemma \ref{lemma:tau_coupling}, for any $p < p + n\leq j_1<\ldots<j_m$, there exists a random vector $\parr*{\widetilde{X}_{j_1},\ldots, \widetilde{X}_{j_m}}$ that is identically distributed with $\parr*{X_{j_1},\ldots,X_{j_m}}$, independent of $\sigma(X_i,i\leq p)$ and satisfies
\begin{align*}
\frac{1}{m}\E\bbrace*{\sum_{k=1}^m d\parr*{X_{j_k}, \widetilde{X}_{j_k}}} = \frac{1}{m}\tau\parr*{\sigma(X_i,i\leq p), (X_{j_1},\ldots,X_{j_m})}\leq \gamma_1\exp(-\gamma_2 n),
\end{align*}
where $d(x,y)=\sum_{\ell=1}^d|x_\ell-y_\ell|$ on $\RR^d$. For each fixed $j\geq 1$, $\bbrace*{\widetilde{e}_j\parr*{\widetilde{X}_{j_1}},\ldots, \widetilde{e}_j\parr*{\widetilde{X}_{j_m}}}$ is identically distributed as $\bbrace*{\widetilde{e}_j\parr*{X_{j_1}},\ldots,\widetilde{e}_j\parr*{X_{j_m}}}$, independent from $\sigma\parr*{e_j(X_i),i\leq p}$. Thus by the first part of Lemma \ref{lemma:tau_coupling} and the Condition (A$^\prime$), it holds that
\begin{align*}
\frac{1}{m}\tau\fence*{\sigma\bbrace*{\widetilde{e}_j(X_i),i\leq p}, \bbrace*{\widetilde{e}_j(X_{j_1}),\ldots,\widetilde{e}_j(X_{j_m})}} &\leq \frac{1}{m}\E\bbrace*{\sum_{k=1}^m\abs*{\widetilde{e}_j(X_{j_k})-\widetilde{e}_j(\tilde{X}_{j_k})}}\\
&\leq \frac{1}{m}L\E\bbrace*{\sum_{k=1}^m d\parr*{X_{j_k},\widetilde{X}_{j_k}}}\\
&\leq L\gamma_1\exp(-\gamma_2 n).
\end{align*}
Therefore, $\{\widetilde{e}_j(X_i)\}_{i=1}^n$ is still geometrically $\tau$-mixing, and thus geometrically $\theta$-mixing with coefficient $L\gamma_1\exp(-\gamma_2 n)$. Letting $p = 2+\delta$ in Lemma \ref{lemma:theta_covariance}, and using the fact that 
\begin{align*}
\bbrace*{\E\parr*{\abs*{\widetilde{e}_j(X_1)}^a}}^{1/a} \leq 2\bbrace*{\E\parr*{\abs*{e_j(X_1)}^a}}^{1/a} \leq 2\mu_a
\end{align*}
for all $a\geq 1$, it holds that for any $j\geq 1$,
\begin{align*}
\sigma_j^2 \leq 12\mu_{2+\delta}^{\frac{2+\delta}{1+\delta}}\fence*{\sum_{k=0}^\infty\bbrace*{L\gamma_1\exp(-\gamma_2 k)}^{\delta/(1+\delta)}}\leq \frac{12(\gamma_1 L)^{\frac{\delta}{1+\delta}}}{1-\exp\bbrace*{-\gamma_2 \delta/(1+\delta)}}\parr*{2\mu_{2+\delta}}^{\frac{2+\delta}{1+\delta}}.
\end{align*}
This completes the proof.
\end{proof}


\subsection{Proof of Theorem \ref{thm:random_fourier_tau}}
\begin{proof}
Consider the approximating function $s_{D}$ defined in the proof of Theorem \ref{thm:random_fourier_alpha}. $s_D$ can be expanded with bases $\cos\parr*{2\pi u^\top x}$ and $\sin\parr*{2\pi u^\top x}$ in $\RR^d$. Recall the distance $d(x,y) = \sum_{\ell=1}^d |x_\ell - y_\ell|$ in $\RR^d$. Then it holds that
\begin{align*}
\abs*{\cos(2\pi u^\top x) - \cos(2\pi u^\top y)} \leq 2\pi\abs*{u^\top (x-y)} \leq 2\pi\max_{1\leq\ell\leq d}|u_\ell|\leq 2\pi\|u\|.
\end{align*}
This argument holds similarly for $\sin(2\pi u^\top x)$. Therefore, in order to upper bound the Lipschitz constant $L$ in Condition (A$^\prime$) of Proposition \ref{prop:bern_U_tau}, it suffices to upper bound the following quantity
\begin{align*}
\omega \define \max_{1\leq j\leq m}\parr*{\max_{1\leq i\leq D_1}\|u_{ij}\| \vee \max_{1\leq i\leq D_2}\|v_{ij}\| \vee \max_{1\leq i\leq D_3}\|\tilde{u}_{ij}\| \vee \max_{1\leq i\leq D_4}\|\tilde{v}_{ij}\|},
\end{align*}
where $\{u_i\}_{i=1}^{D_1},\{v_i\}_{i=1}^{D_2},\{\tilde{u}_i\}_{i=1}^{D_3},\{\tilde{v}_i\}_{i=1}^{D_4}$ are the frequencies of the cosine and sine bases of $s_{D_1},s_{D_2},s_{D_3},s_{D_4}$ in the proof of Theorem \ref{thm:random_fourier_alpha}. Note that
\begin{align*}
D_1\vee D_2 \vee D_3 \vee D_4 \leq D \define \Omega\fence*{\frac{md\nm*{\hat{\bar{f}}}_{L^1}^2}{t^2}\log\bbrace*{\frac{8\pi c\diam(\cal{M})\nm*{\hat{\bar{f}}}_{L^1}^{1-1/q}\mu_q\parr*{\hat{\bar{f}}}}{t}}}
\end{align*}
and for each $j\in[m]$ and $L_0 > 0$,
\begin{align*}
\P\parr*{\|u_{ij}\|\vee\|v_{ij}\|\vee\|\tilde{u}_{ij}\|\vee\|\tilde{v}_{ij}\|\geq L_0} \leq 4\frac{\mu_q^q\parr*{\hat{\bar{f}}}}{L_0^q}.
\end{align*}
Therefore, 
\begin{align*}
P\parr*{\omega\geq L_0} \leq 16D\frac{\mu_q^q\parr*{\hat{\bar{f}}}}{L_0^q}. 
\end{align*}
Therefore, we can find a set of realizations $\{u_i\}_{i=1}^{D_1},\{v_i\}_{i=1}^{D_2},\{\tilde{u}_i\}_{i=1}^{D_3},\{\tilde{v}_i\}_{i=1}^{D_4}$ such that 
\begin{align*}
\omega \leq L_0 \define \Omega\fence*{\frac{\mu_q^q\parr*{\hat{\bar{f}}}md\nm*{\hat{\bar{f}}}_{L^1}^2}{t^2}\log\bbrace*{\frac{8\pi c\diam(\cal{M})\nm*{\hat{\bar{f}}}_{L^1}^{1-1/q}\mu_q\parr*{\hat{\bar{f}}}}{t}}}^{1/q}.
\end{align*}
Plugging in $\diam(\cal{M}) = 2M\sqrt{md}$ completes the proof. 
\end{proof}

\subsection{Supporting Lemmas}

\begin{lemma}[Lemma 7, \cite{dedecker2004coupling}]
\label{lemma:mixing_compare}
Let $(\Omega,\cal{A},\P)$ be a probability space, and $\{X_i\}_{i=1}^n$ be a stationary sequence of random variables with values in $\RR$. Let $Q_{|X|}$ be the quantile function of $X$: if $u\in[0, 1]$, $Q_{|X|}(u) = \inf\bbrace*{t\in\RR:\P(|X|\geq t)\leq u}$. Then for each positive integer $k$, it holds that
\begin{align*}
\tau_k(i) \leq 2\int_0^{2\alpha_k(i)} Q_{|X|}(u)du.
\end{align*}
In particular, if $p$ and $q$ are conjugate numbers, then $\tau_k(i) \leq 2\|X\|_p\parr*{2\alpha_k(i)}^{1/q}$ and $\tau(i)\leq 2\|X\|_p\parr*{2\alpha(i)}^{1/q}$.
\end{lemma}

\begin{lemma}[Lemma 3 in \cite{dedecker2004coupling} and Lemma 5.3 in \cite{dedecker2007weak}]
\label{lemma:tau_coupling}
Let $(\Omega,\cal{F},\P)$ be a probability space, $\cal{A}$ a sub $\sigma$-algebra of $\cal{F}$ and $X$ a random variable taking values in a Polish space $(\cal{X},d)$. Assume that $\int d(x,x_0)\P_X(dx)$ is finite for any $x_0\in\cal{X}$. Assume that the probability space considered is rich enough such that there exists a random variable U uniformly distributed over $[0, 1]$, independent of the $\sigma$-algebra generated by $X$ and $\cal{A}$. Then there exists a random variable $\widetilde{X}$, which is measurable with respect to $\cal{A}\vee \sigma(X)\vee \sigma(U)$, independent of $\cal{A}$ and identically distributed as $X$, such that
\begin{align}
\tau(\cal{A},X;d) = \E\fence*{d(X,\widetilde{X})}.
\end{align}
On the other hand, for any random variable $Y$ that is identically distributed with $X$ and independent of $\cal{A}$, it holds that
\begin{align*}
\tau(\cal{A},X;d) \leq \E\fence*{d(X,Y)}.
\end{align*}
\end{lemma}

\begin{lemma}
\label{lemma:theta_covariance}
Let $\{X_i\}_{i=1}^\infty$ be a sequence of real valued random variables with finite $L_p$ moment for some $p > 2$, that is, $\mu_p \define \sup_{i}\bbrace*{\parr*{\E|X_i|^p}^{1/p}} <\infty$. Then, it holds that
\begin{align*}
\sup_i\parr*{\sum_{j > i}\abs*{\text{Cov}\parr*{X_i, X_j}}} \leq 6\mu_p^{\frac{p}{p-1}}\bbrace*{\sum_{k=1}^\infty \theta(k)^{\frac{p-2}{p-1}}},
\end{align*}
where $\theta(k)$ is the $k$th $\theta$-coefficient of the sequence.
\end{lemma}
\begin{proof}
The proof follows essentially from that of Lemma 4.2 in \cite{dedecker2007weak}. In detail, define $X^{(M)}\define (M\wedge X) \vee (-M)$ as the truncated version of $X$, and $\widetilde{X} \define X^{(M)} - \E\bbrace*{X^{(M)}}$. It can be readily checked that $\widetilde{X}$, as a function of $X$, is bounded by $2M$ and also Lipschitz with constant $1$. For any positive integers $i < j$, it holds that
\begin{align*}
\abs*{\cov\parr*{X_i, X_j}} \leq \abs*{\cov\parr*{\widetilde{X}_i, \widetilde{X}_j}} + \abs*{\cov\parr*{X_i - \widetilde{X}_i, X_j}} + \abs*{\cov\parr*{\widetilde{X}_i, X_j - \widetilde{X}_j}}.
\end{align*}
For the first term, by definition of the $\theta$-coefficient, 
\begin{align*}
\abs*{\cov\parr*{\widetilde{X}_i, \widetilde{X}_j}} \leq 2M\theta(j-i).
\end{align*}
For the second term, by H\"older's inequality, we have
\begin{align*}
\abs*{\cov\parr*{X_i-\widetilde{X}_i, X_j}} &=\abs*{\E\fence*{\bbrace*{X_i-\E\parr*{X_i} - \widetilde{X}_i}X_j}}\\
&\leq 2\|X_j\|_p\|X_i-X_i^{(M)}\|_{q}\\
&\leq 2\mu_p^p M^{1-p/q},
\end{align*}
where $q = p/(p-1)$ is the conjugate number of $p$ and in the third step we use the fact that $p > q$ and 
\begin{align*}
\E\fence*{|X_i|^q\mathbbm{1}\bbrace*{|X_i|\geq M}} \leq \bbrace*{\E\parr*{|X_i|^p}}^{q/p}\bbrace*{\P\parr*{|X_i|\geq M}}^{1-q/p} \leq \mu_p^qM^{q-p}\mu_p^{p-q} = \mu_p^pM^{q-p}.
\end{align*}
The third term can be similarly bounded by $2\mu_p^p M^{1-p/q}$. Putting together the pieces gives
\begin{align*}
\abs*{\cov\parr*{X_i,X_j}} \leq 2M\theta(j-i) + 4\mu_p^p M^{2-p}.
\end{align*}
Choosing $M=\bbrace*{\mu_p^p/\theta(j-i)}^{1/(p-1)}$ gives
\begin{align*}
\abs*{\cov\parr*{X_i,X_j}} \leq 6\mu_p^{\frac{p}{p-1}}\theta(j-i)^{\frac{p-2}{p-1}}.
\end{align*}
Summing over all $j > i$ and taking the supremum over all $i$, we have
\begin{align*}
\sup_i\bbrace*{\sum_{j>i}\abs*{\cov\parr*{X_i,X_j}}} \leq 6\mu_p^{\frac{p}{p-1}}\bbrace*{\sum_{k=1}^\infty \theta(k)^{\frac{p-2}{p-1}}}.
\end{align*}
This completes the proof.
\end{proof}

\end{document}